\documentclass[10pt]{amsart}

\usepackage{amsfonts}
\usepackage{amssymb}
\usepackage{amsmath, amsthm, latexsym, amscd}
\usepackage[all]{xy}
\usepackage{hyperref}
\usepackage{subcaption}
\usepackage[margin=1.5in]{geometry}
\usepackage{enumerate}
\usepackage{tikz}\usetikzlibrary{matrix}\usetikzlibrary{arrows}
\usepackage{xcolor}
\usepackage{tikz-cd}
\tikzset{%
    symbol/.style={%
        draw=none,
        every to/.append style={%
            edge node={node [sloped, allow upside down, auto=false]{$#1$}}}
    }
}

\def \Z{\mathbb{Z}}
\def \R{\mathbb{R}}
\def \Q{\mathbb{Q}}
\def \A{\mathcal{A}}

\def \E{\mathcal{E}}
\def \F{\mathcal{F}}
\def \GG{\mathcal{G}}

\def \O{\mathcal{O}}
\def \I{\mathcal{I}}

\def \k{{\bf k}}
\def \e{{\bf e}}

\def \B{\mathbb{B}}
\def \L{\mathcal{L}}

\def \bs{{\bf s}}
\def \br{{\bf r}}

\def\v{\mathfrak{v}}
\def\w{\mathfrak{w}}
\def\uu{\mathfrak{u}}

\def \Trop{\operatorname{Trop}}

\def \In{\operatorname{in}}

\def \Proj{\operatorname{Proj}}
\def \gr{\operatorname{gr}}
\def \Spec{\operatorname{Spec}}
\def \Hom{\operatorname{Hom}}

\def \GL{\operatorname{GL}}

\def \Sh{\operatorname{Sh}}
\def \Sym{\operatorname{Sym}}
\def \can{\operatorname{can}}
\def \GF{\textup{GF}}

\def\Div{\operatorname{Div}}

\def\Pic{\operatorname{Pic}}

\newcommand{\bx}{{\bf x}}
\newcommand{\by}{{\bf y}}
\newcommand{\bt}{{\bf t}}
\newcommand{\N}{N_{\mathbb{Q}}}
\newcommand{\M}{M_{\mathbb{Q}}}

\newcommand{\m}{{\bf m}}

\DeclareMathOperator{\Mod}{Mod}
\DeclareMathOperator{\Vect}{Vect}
\DeclareMathOperator{\Alg}{Alg}
\DeclareMathOperator{\colim}{\varinjlim}

\theoremstyle{plain}
\newtheorem{theorem}{Theorem}[section]
\newtheorem{lemma}[theorem]{Lemma}
\newtheorem{proposition}[theorem]{Proposition}
\newtheorem{corollary}[theorem]{Corollary}
\newtheorem{conjecture}[theorem]{Conjecture}

\newtheorem{algorithm}[theorem]{Algorithm}

\theoremstyle{definition}
\newtheorem{example}[theorem]{Example}
\newtheorem{definition}[theorem]{Definition}
\newtheorem{remark}[theorem]{Remark}

\begin{document}

\title{Toric flat families, valuations, and applications to projectivized toric vector bundles}

\author{Kiumars Kaveh}
\address{Department of Mathematics, University of Pittsburgh,
Pittsburgh, PA, USA.}
\email{kaveh@pitt.edu}

\author{Christopher Manon}
\address{Department of Mathematics, University of Kentucky, Lexington, KY, USA}
\email{Christopher.Manon@uky.edu}

\date{\today}
\thanks{The first author is partially supported by National Science Foundation Grant DMS-2101843 and a Simons Collaboration Grant (award number 714052). The second author is partially supported by National Science Foundation Grant DMS-2101911 and a Simons Collaboration Grant (award number 587209).}
\subjclass[2010]{}
\keywords{}

\begin{abstract}
Using the notion of a valuation into the semifield of piecewise linear functions, we give a classification of torus equivariant flat families of finite type over a toric variety base by certain piecewise linear maps between fans. As a consequence, we derive a classification of toric vector bundles phrased in terms of tropicalized linear spaces. We use these tools to give a characterization of the Mori dream space property for a projectivized toric vector bundle. 
\end{abstract}

\maketitle

\tableofcontents

\section{Introduction}
We use ideas from Gr\"obner theory and tropical geometry to give a classification of torus equivariant flat families over a toric variety base (toric flat families), far extending Klaychko's classification of equivariant vector bundles on toric varieties (toric vector bundles). We use our techniques to give an algorithm to establish that a given projectivized toric vector bundle is a Mori dream space. In particular, this recovers many known constructions of such Mori dream spaces in the literature (\cite{Gonzalez}, \cite{GHPS}, \cite{Hausen-Suzz}, \cite{Nodland}) and gives a systematic way to produce many new classes of examples. Below we first give an overview of the paper and then explain the main results in more detail.


Let $R$ be a positively graded finitely generated $\k$-domain with $X=\Proj(R)$ the corresponding irreducible projective variety with $\dim(X) = n$. Given a full rank valuation $\v: R\setminus\{0\} \to \Z^{n+1}$, the theory of Newton-Okounkov bodies assigns to $R$ a convex body $\Delta(R, \v) \subset \R^n$ which encodes asymptotic information about the Hilbert function of $R$ and hence geometric information about $(X, \mathcal{O}(1))$. When the value semigroup $S(R, \v):=\v(R \setminus\{0\})$ is finitely generated, we say that $(R, \v)$ has a \emph{finite Khovanskii basis}. In this case one has a flat family over the affine line $\mathbb{A}^1$ with generic fiber $X$ and the special fiber $X_0$, the projective toric variety associated to the semigroup $S(R, \v)$. Such a family is called a \emph{toric degeneration} (see \cite{Anderson}). The Khovanskii basis theory is concerned with the computational side of the Newton-Okounkov bodies theory. 
It provides a means to generalize Gr\"obner basis theory for polynomial rings to much larger classes of finitely generated algebras $R$ (see \cite{Kaveh-Manon-NOK, Ehrmann}). 

In this paper, we replace $\Z^{n+1}$-valued valuations with valuations with values in the semifield $\mathcal{O}_N$ of integral piecewise linear functions on a finite rank lattice $N$. We introduce a notion of Khovanskii basis in this setting. We see that, extending the construction of toric degenerations over an affine line from $\Z^{n+1}$-valued valuations, finite Khovanskii bases in this setting give rise to flat families with general fiber $X$ and over a toric variety base $Y(\Sigma)$. In particular, this applies to vector bundles over a toric variety $Y(\Sigma)$ and yields a systematic and algorithmic study of global coordinate rings (Cox rings) of projectivized toric vector bundles. We summarize our results in this regard as follows:

\begin{enumerate}
    \item[$\bullet$] After fixing some algebraic data, we classify toric vector bundles by points on a certain polyhedral fan (Proposition \ref{prop-diagramfan}), and we show that special faces of this fan characterize natural families of bundles with Mori dream space property (Proposition  \ref{prop-formal}). 
    \item[$\bullet$] We give an algorithm (Algorithm \ref{alg-StrongKhovanskii}) which constructs a finite generating set of the Cox ring of a projectivized toric vector bundle, provided one exists. The algorithm does not  terminate if one does not exist. 
    \item[$\bullet$] We characterize the generating sets of Cox rings of projectivized toric vector bundles using \emph{prime points} on an associated tropical variety (Theorem \ref{thm-main-StrongKhovanskii}).
\end{enumerate}

\noindent
The above recover several finiteness results from the literature (\cite{Gonzalez}, \cite{GHPS}, \cite{Hausen-Suzz}, \cite{Nodland}), and obtain new classification results for bundles with the Mori dream space property. 

In a future work, we plan to address applications of our approach to geometry of matroids. It seems that the point of view, introduced in the present paper, of regarding a toric vector bundle as a tropical point (over the semifield of piecewise linear functions), gives a very natural way to construct and extend the tautological classes of matroids as introduced in \cite{Eur}.

To prove the above results, we develop a theory of Khovanskii bases for valuations with values in the semialgebra of piecewise linear functions. This requires us to carefully study the category of vector spaces equipped with prevaluations into the piecewise linear semialgebra. An important ingredient is establishing an adjunction between this category and the category of $T$-equivariant quasi-coherent sheaves on a toric variety.
\\ 



\subsection{Classification of toric flat families}


We take the ground field $\k$ to be an algebraically closed field. We let $T_N$ to be the algebraic torus with cocharacter lattice $N$ and character lattice $M=\Hom(N, \Z)$. We denote by $Y(\Sigma)$ the toric variety associated to a fan $\Sigma$ in $N_\Q = N \otimes \Q$. We let $\O_\Sigma$ denote the set of all piecewise linear functions on $|\Sigma|$, the support of $\Sigma$, that are integral with respect to $N$. The set $\O_\Sigma$ together with the tropical operations of $\min$ and $+$ is a semifield which we refer to as the \emph{piecewise linear semifield} on $\Sigma$. In particular, for each $\phi \in \O_\Sigma$ there is a finite subdivision of $|\Sigma|$ into convex rational cones $\sigma$ such that the restriction $\phi\!\mid_\sigma$ is linear. When $|\Sigma| = N_\Q$ we denote this semifield by $\O_N$. 

\begin{definition} \label{def-tff-intro}
By a \emph{toric flat family} over $Y(\Sigma)$ we mean a flat affine family $\pi: \mathcal{X} \to Y(\Sigma)$ equipped with an action of the torus $T_N$ lifting its action on the toric variety $Y(\Sigma)$. In particular, the fibers of $\pi$ are affine schemes. 
\end{definition}


Toric flat families include many important and well-studied classes of varieties such as torus equivariant vector bundles and principal bundles on toric varieties (\emph{toric vector bundles} and \emph{toric principal bundles}), see \cite{Klyachko, Biswas, Kaveh-Manon-Building}. Another class of toric flat families are \emph{toric degenerations}, i.e. toric flat families over the affine line. Toric degenerations have a close connection with the theory of Newton-Okounkov bodies (see \cite{Anderson, KMM}) and toric degenerations of special classes of varieties have been studied extensively by various authors. 

Before stating our main results, we give a brief review of some previous work in the literature. Toric vector bundles were first classified by Kaneyama in \cite{Kaneyama} using certain cocycles.  Klyachko then gave a classification \cite{Klyachko} in terms of compatible $\Z$-filtrations in a vector space. More recently, the authors \cite{Kaveh-Manon-Building}, and Biswas, Dey, and Poddar \cite{Biswas, BDP-nonsing} have given different classifications of toric principal bundles. The classification of Biswas, Dey, and Poddar is along the lines of Kaneyama's, utilizing certain cocycles. The authors' classification in \cite{Kaveh-Manon-Building} is in terms of piecewise-linear maps to \emph{Tits buildings}, and should be viewed as an extension of Klyachko's result.

Klyachko's work has lead to a number of powerful results and generalizations. In \cite{Payne}, Payne uses Klyachko's classification to study moduli space of toric vector bundles over $Y(\Sigma)$ with fixed total Chern class.  In \cite{Perling}, Perling gives a Klyachko-like classification of torus-equivariant sheaves on a toric variety in terms of representations of certain quivers derived from the fan $\Sigma$. We also mention the work of DiRocco, Jabbusch and Smith \cite{DJS} where they give a criterion for global generation of a toric vector bundle in terms of an associated collection of convex polytopes called a \emph{parliament of polytopes}. 

The general approach in the present paper encompasses and extends many previous works in the literature. We describe the connection between \cite{DJS} in Section \ref{sec-DJS}. We also consider torus-equivariant sheaves in Section \ref{category}, where we attach a \emph{prevalued vector space} to any such sheaf. This leads to an adjunction of categories which is not an equivalence in general. Nevertheless, it is an equivalence on a subcategory which we call \emph{eversive} sheaves, and in particular contains reflexive sheaves. 


Klyachko's classification is very close to the classification of toric vector bundles  by tropical points (Theorem \ref{thm-main-bundle}) that we obtain as a special case of our classification of toric flat families (Theorem \ref{thm-main-family}). We remark that Theorem \ref{thm-main-bundle} is essentially contained in Payne's observation in \cite{PayneModuli, PayneFan} that the Klyachko data of a toric vector bundle can be used to construct a filtration-valued function on the support $|\Sigma|$ of the fan. 

Throughout we assume that $\mathcal{X}$ comes equipped with a $\mathbb{G}_m$-action preserving the fibers and commuting with the $T_N$-action. From the $T_N$-equivariance it follows that the fibers of $\mathcal{X}$ over the open orbit in $Y(\Sigma)$ (the general fibers) are all isomorphic to $\Spec(A)$, for a finitely generated positively graded $\k$-domain $A$. Specifically, we fix a point $x_0$ in the open orbit in $Y(\Sigma)$ and we take $A$ to be the coordinate ring of the fiber over $x_0$. One can think of $x_0$ as the identity element in the torus $T_N$. 


Let $\A$ be the sheaf of $\O_{Y(\Sigma)}$-algebras such that $\Spec_{Y(\Sigma)}(\A) = \mathcal{X}$.
The $\mathbb{G}_m$-action on $\mathcal{X}$ corresponds to a positive grading $\A= \bigoplus_{n \geq 0} \A_n$. Each $\A_n$ can be regarded as the sheaf of sections of a $T_N$-equivariant vector bundle on $Y(\Sigma)$. We thus can regard $\A$ as an infinite dimensional $T_N$-equviariant vector bundle on $Y(\Sigma)$. The Klyachko classification of toric vector bundles then implies that $\A$ corresponds to a system of filtrations on the $\k$-algebra $A$, regarded as an infinite dimensional vector space. In our setting, this system of filtrations gives rise to a map: $$\v: A \to \O_\Sigma.$$ When the family has reduced and irreducible fibers, it is straightforward to verify that the map $\v$ satisfies the axioms of a \emph{valuation} from an algebra to a semifield. 

We point out that in the valuation theory literature in commutative algebra, a valuation takes values in a totally ordered abelian group such as $\Z^d$. 
Given a domain $A$ and a valuation $\v: A \setminus\{0\} \to \Z^d$, the theory of Khovanskii bases attempts to generalize basic constructions and techniques from Gr\"obner theory of polynomial ring to $(A, \v)$ (see \cite{Kaveh-Manon-NOK, Ehrmann}).  One of the purposes of the present paper is to show that valuations with values in the piecewise linear semifield are also important in algebraic geometry and are the right gadgets to classify toric flat families.
When the family $\mathcal{X}$ is finite-type, the information in $\v$ can be captured in its values on a certain set of generators $\mathcal{B} \subset A$. In sympathy with the case when $\v$ takes values in $\Z^d$, we call $\mathcal{B}$ a \emph{Khovanskii basis} of $\v$. 
The following is our main theorem summarizing these ideas (see Section \ref{Khovanskii}).

\begin{theorem}[Toric flat families as valuations] \label{thm-main-family}
Let $\pi: \mathcal{X} \to Y(\Sigma)$ be a toric flat family of finite type with reduced, irreducible fibers and let $A$ denote the coordinate ring of general fibers. Such families are classified by valuations 
$\v: A \to \O_\Sigma$ having a finite Khovanskii basis $\mathcal{B}$ with the following property: for each $\sigma \in \Sigma$, $A$ has a certain vector space basis $\mathbb{B}_\sigma$ consisting of monomials in $\mathcal{B}$ such that $\v(b)_{| \sigma}$ is a linear function, for all $b \in \mathbb{B}_\sigma$. We refer to $\mathbb{B}_\sigma$ as a \emph{linear adapated basis} (see Section \ref{freesection}).
\end{theorem} 

Next we interpret the data of a valuation $\v: A \to \O_\Sigma$ as a \emph{piecewise linear map} between fans. 
This point of view is motivated by the classification of toric vector bundles in terms of piecewise linear maps to the Tits buildings of general linear groups (see \cite{Kaveh-Manon-Building}, also \cite{Klyachko}). As above let $\mathcal{B}$ be a finite Khovanskii basis for $(A, \v)$. Let $I$ be the ideal of relations among the generating set $\mathcal{B}$. It gives a presentation of $A$ as a quotient of a polynomial ring by $I$. In Section \ref{Khovanskii} we define a certain subfan $\mathcal{K}(I)$ of the Gr\"obner fan of $I$. Motivated by the theory of buildings we define an \emph{apartment} in $\mathcal{K}(I)$ to be the subcomplex obtained by intersecting  $\mathcal{K}(I)$ with a maximal face $\tau$ of the Gr\"obner fan of $I$. 
Our next result, roughly speaking, states that different ways to ``fold'' the fan $\Sigma$ into the fan $\mathcal{K}(I)$ correspond to distinct toric flat families over $Y(\Sigma)$ with general fiber $\Spec(A)$ and Khovanskii basis $\mathcal{B}$. 
More precisely, we have the following (Proposition \ref{prop-diagram}). 

\begin{corollary}[Toric flat families as peicewise linear maps] \label{cor-main-family}
The information of a toric flat family $\pi:\mathcal{X} \to Y(\Sigma)$ with a finite Khovanskii basis $\mathcal{B}$ is equivalent to a 
piecewise linear map $\Phi: |\Sigma| \to \mathcal{K}(I)$ which maps each cone in $\Sigma$ to an apartment in $\mathcal{K}(I)$. We call such a piecewise linear map a \emph{$\Sigma$-adapted map} (see Section \ref{Khovanskii}).
\end{corollary}

The above corollary says that $\mathcal{K}(I)$ can be thought of as a ``classifying space'' for the toric flat families with general fiber $\Spec(A)$ and Khovanskii basis $\mathcal{B}$. 

In Section \ref{Khovanskii} (Proposition \ref{prop-diagram}) we show that $\Phi$ is captured by its values on the rays of $\Sigma$, which we organize into a matrix $D(\Phi)$ called the \emph{diagram} of $\Phi$.  This gives a systematic combinatorial method for constructing toric flat families over $Y(\Sigma)$. For example when $\Sigma$ is a simplicial fan, a diagram corresponds to choosing a point $w(\varrho) \in \mathcal{K}(I)$, for every ray $\varrho \in \Sigma(1)$, such that for all $\varrho$ from a face $\sigma \in \Sigma$ the $w(\varrho)$ land in a common apartment of $\mathcal{K}(I)$.  

When the family has reduced and irreducible fibers, and hence $\v$ is a valuation, $\Phi = \Phi_\v$ in fact takes values in the tropical variety $\Trop(I)$. In the particular case when the toric family is a toric vector bundle we get the following. 
\begin{theorem}[Toric vector bundles as tropical points] \label{thm-main-bundle}
Let $E$ be an $r$-dimensional vector space. 
The toric vector bundles over $Y(\Sigma)$, with $E$ as general fiber, correspond to valuations $\v: \Sym(E) \to \O_\Sigma$ such that for each $\sigma \in \Sigma$, the restriction $\v\!\mid_\sigma: \Sym(E) \to \O_\sigma$ has a Khovanskii basis given by the monomials of a vector space basis $\mathcal{B}_\sigma \subset E$.
In this case, the ideal $I$ is a linear ideal $L$ and one can think of the information of a toric vector bundle in three differnt ways:
\begin{itemize}
\item[(i)] As a point $\Phi \in \Trop_{\O_\Sigma}(L)$ in the tropical variety over the semifield $\O_\Sigma$ (see Section \ref{background} for the notion of tropical variety over a semifield).
\item[(ii)] As a piecewise linear map from $|\Sigma|$ to $\Trop(L)$ (see Section \ref{Khovanskii}). 
\item[(iii)] As a matrix with rows in $\Trop(L)$ (see Section \ref{Khovanskii}).    
\end{itemize}
\end{theorem}

The idea in part $\textup{(ii)}$ of Theorem \ref{thm-main-bundle}, to encode the data of a toric vector bundle as a family of filtrations of a vector space indexed by the points in the support of the fan $|\Sigma|$, is not new and first appeared in the work of Payne \cite{PayneFan}.  This construction was used in the work of Hering, Payne, and Musta\c{t}\u{a} \cite{HMP} to give a combinatorial characterization of ampleness for toric vector bundles.  Corollary \ref{cor-main-family} and Theorem \ref{thm-main-bundle} can be thought of as extensions of the well-known fact that toric line bundles over $Y(\Sigma)$ are classified by integral piecewise linear maps $\phi: |\Sigma| \to \Z$, to toric vector bundles: $\Phi: |\Sigma| \to \Trop(L)$, and toric flat families: $\Phi: |\Sigma| \to \mathcal{K}(I)$.

\subsection{Mori dream space property and strong Khovanskii bases}
Next we address the problem of finite generation of the total coordinate ring of a toric flat family. In general, this is a notoriously difficult problem. For example, a very special case of this problem that asks which projectivized toric vector bundles are Mori dream spaces is already very difficult and has been subject of much recent research. Our valuation theory approach gives a systematic way to attack this problem and to recover many of the previously known examples of such Mori dream spaces as well as many new examples and constructions.

The question of when a projectivized toric vector bundle is a Mori dream space goes back to the work of Hering, Musta\c{t}\u{a}, and Payne \cite{HMP}. The first results along these lines are found in the works of Gonz\'alez \cite{Gonzalez}  and Hausen and S\"u\ss \ \cite{Hausen-Suzz}.  In \cite{Gonzalez}, Gonz\'alez shows that the projectivization of every rank $2$ vector bundle over a projective toric variety is a Mori dream space.  In \cite{Hausen-Suzz}, Hausen and S\"u\ss \ show that the projectization of the tangent bundle of any smooth, projective toric variety is a Mori dream space.  The common thread of these two results is that the Mori dream space property can be obtained by restricting the number of steps in the Klyachko filtrations of the associated bundle to at most two.  Such bundles are called \emph{sparse} bundles in Section \ref{sec-Cox-TVB}.  Sparse bundles are shown to always have Mori dream space projectivizations by Gonz\'alez, Hering, Payne, and S\"u\ss \ in \cite{GHPS} and by N\"odland in \cite{Nodland}.   We recover this result as Corollary \ref{cor-sparse} of Proposition \ref{prop-formal}.  In particular, the sparse condition is a special case of membership in certain faces of the fan which we introduce  (see Proposition \ref{prop-weakfan}).  Using this tropical approach, we can produce new families of bundles whose projectivizations are Mori dream spaces.  For example, the \emph{uniform} bundles analyzed in Corollary \ref{cor-uniform} are defined by requiring the ideal $I$ to be a generic linear ideal. 

In the usual Khovanskii basis theory, one considers a finitely generated domain $A$ equipped with a valuation $\v: A \setminus \{0\} \to \Z^d$. A set $\mathcal{B} \subset A$ is then a \emph{Khovanskii basis} for $(A, \v)$ if the images of the elements in $\mathcal{B}$ generate the associated graded algebra $\textup{gr}_\v(A)$. Under mild assumptions, it follows that $\mathcal{B}$ is a generating set for $A$ as well. When $\mathcal{B}$ is finite, essentially one can reduce computational problems regarding $A$ and ideals in it to computations in $\textup{gr}_\v(A)$. There is an algorithm to extend a given subset of $A$ to a finite Khovanskii basis. This algorithm terminates in finite time if and only $(A, \v)$ has a finite Khovanskii basis (\cite[Algorithm 2.18 and Corollary 2.19]{Kaveh-Manon-NOK}).

We apply these ideas to the problem of finite generation of the total coordinate ring of a toric flat family. We give an analogue of the algorithm above for a valuation $\v: A \to \mathcal{O}_\Sigma$ to find a finite Khovanskii basis for the total coordinate ring of the family. 
More precisely, let $\A$ be a sheaf of algebras on a smooth toric variety $Y(\Sigma)$. 
In Section \ref{StrongKhovanskii} we consider the total coordinate ring $$\mathcal{R}(\mathcal{A}) = \bigoplus_{\mathcal{L} \in \Pic(Y(\Sigma))} H^0(Y(\Sigma), \mathcal{L}\otimes_{Y(\Sigma)} \A)$$ of global sections of $\A$. In the case of a toric vector bundle $\E$, this ring corresponds to the Cox ring $\mathcal{R}(\E)$ of the associated projective bundle $\mathbb{P}\E$. Let $\v: A \to \O_\Sigma$, $\Phi:|\Sigma| \to \mathcal{K}(I)$ and diagram $D(\Phi)$ be the quasivaluation, piecewise linear map and the diagram corresponding to the family respectively. We wish to give conditions on $\Phi: |\Sigma| \to \mathcal{K}(I)$, or equivalently the diagram $D(\Phi)$, which ensure that $\mathcal{R}(\mathcal{A})$ is finitely generated, and understand the piecewise linear geometry of the corresponding space of maps $\Phi$.  

This leads us to define the notion of a strong Khovanskii basis. A generating set $\mathcal{B} \subset A$ is said to be a \emph{strong Khovanskii basis} if it is the image of a generating set of $\mathcal{R}(\mathcal{A})$  (Section \ref{StrongKhovanskii}).  A strong Khovanskii basis is always a Khovanskii basis, but not vice-versa. In general, finite strong Khovanskii bases do not exist.   Our first contribution in this regard is Algorithm \ref{alg-StrongKhovanskii}, which constructs a finite strong Khovanskii basis if one does exist.  This should be compared to \cite[Algorithm 2.18]{Kaveh-Manon-NOK}. In particular, when applied to the Cox ring $R(\E)$ of a toric vector bundle, Algorithm \ref{alg-StrongKhovanskii} will terminate in finite time if and only if the Cox ring $R(\E)$ is finitely generated. In this case, the constructed set $\mathcal{B} \subset \Sym(E)$ will contain, as its degree $1$ piece, a representation of the matroid defined by DiRocco, Jabbusch, and Smith in \cite{DJS}.    We mention the interesting work of N\o dland \cite{Nodland}, where it is shown that a generating set of $R(\E)$ can be taken to be a (possibly infinite) union of (representations of) DiRocco-Jabbusch-Smith matroids. 

More generally, we characterize the existence of finite strong Khovanskii bases for $\A$ with reduced and irreducible fibers (that is, for valuations $\v: A \to \O_\Sigma$) using the notion of \emph{prime points} on a tropical variety introduced in \cite{Kaveh-Manon-NOK}.  Let $n = |\Sigma(1)|$ be the number of rays in $\Sigma$. 



\begin{theorem}   \label{thm-main-StrongKhovanskii}
Let $(A, \v)$ be a valuation as above with Khovanskii basis $\mathcal{B} \subset A$ and diagram $D$, and let $\I_\mathcal{B}$ be the ideal as in Proposition \ref{prop-ideal-StrongKhovanskii}. Let $\tilde{w}_1, \ldots, \tilde{w}_n \in \Trop(\I_\mathcal{B})$ be the lifts of the rows of $D$ (see Theorem \ref{thm-primepointlift}). Then $\mathcal{B}$ is a strong Khovanskii basis if and only if each $\tilde{w}_i$ is a prime point, that is, the initial ideal corresponding to $\tilde{w}_i$ is prime.
\end{theorem}


We use Theorem \ref{thm-main-StrongKhovanskii} to study the Mori dream space property of projectivized toric vector bundles in several cases.  As an application, we give general combinatorial conditions on the diagram which characterize those Cox rings $R(\E)$ which are presented as a complete intersection (Proposition \ref{prop-formal}). 

In Section \ref{sec-diagrams} we introduce the set $\Delta(\Sigma, \Trop_\star(I))$ (Definition \ref{def-diagramfan}) of diagrams $D(\Phi)$ for bundles over $Y(\Sigma)$ with fiber $E$ and Khovanskii basis $\mathcal{B} \subset \Sym(E)$.  We show that $\Delta(\Sigma, \Trop_\star(I))$ is the support of a polyhedral fan, and that the subset of diagrams for which $\mathcal{B}$ is a strong Khovanskii basis is a union of polyhedral cones (Proposition \ref{prop-coneStrongKhovanskii}).  


\begin{conjecture}\label{conj-fan}
The set $\Delta(\Sigma, \Trop_\star(I))$ is the support of a fan $\F$ with the property that the set of those diagrams corresponding to bundles for which $\mathcal{B}$ is a strong Khovanskii basis is a union of faces of $\F$. 
\end{conjecture}

Proposition \ref{prop-weakfan} establishes a weak form of this conjecture. We use a related result (Proposition \ref{prop-add-subtract}) to describe alterations to the diagram $D(\Phi)$ which preserve the Mori dream space property.  In particular, we give polyhedral conditions implying that the pullback of a Mori dream space bundle along a toric blow-up is a Mori dream space bundle (Corollary \ref{cor-toricblowup}). 






We finish the introduction by looking at two ingredients of our valuation theoretic/tropical approach to study toric flat families which we believe are of interest as their own separate topics.
\subsection{Tropical geometry over $\O_N$}
As before, let $N$ and $M$ be dual lattices. Let $\k(T_N)$ denote the field of rational functions on the torus $T_N$.  

\begin{definition}
For $p(\bt) = \sum C_m \bt^m \in \k[T_N]$ let $\w_N(p) = \min\{m \mid C_m \neq 0\}$.  In particular $\w_N(p)$ is the support function of the Newton polytope of $p$.  Let $\w_N: \k(T_N) \to \O_N$ denote the corresponding extension of $\w_N$ to the field of quotients: $\w_N(\frac{p}{q}) = \w_N(p) - \w_N(q)$. 
\end{definition}

Properties of Newton polytopes imply that $\w_N$ respects multiplication: $\w_N(pq) = \w_N(p) + \w_N(q)$, and addition: $\w_N(p + q) \geq \min\{\w_N(p), \w_N(q)\}$.   In particular, $\w_N$ is a \emph{valuation} from $\k(T_N)$ to $\O_N$, and the image of $\k[T_N] \subset \k(T_N)$ under $\w_N$ is precisely the set of support functions for polyhedra with vertices in $M$. Moreover, it is well-known that any piecewise-linear function can be represented as the difference of convex functions, it follows that $\w_N$ is onto. In  \cite{SpeyerSturmfelsTropMath} it was suggested by Sturmfels and Speyer that one could develop tropical geometry over a tropical semiring of polyhedra; we view Theorem \ref{thm-main-bundle} as a first step in this direction: the points on tropical linear spaces over $\O_N$ correspond to toric vector bundles.   In Section \ref{examples} we explore this perspective further by showing that the operation of \emph{tropicalizing} the solution to a system of equations over $\k(T_N)$ can be used to define toric vector bundles and flat toric families.  Asking if all such bundles arise in this way is the analogue of asking if the \emph{Fundamental Theorem of Tropical Geometry} \cite{MSt} holds over $\O_N$. 

\subsection{The category of prevalued vector spaces}

We let $\hat{\O}_\Sigma$ denote a semiring which consists of certain limits of elements of $\O_\Sigma$ (see Section \ref{category}).  In Section \ref{category}, we introduce the category $\Vect_\Sigma$ of vector spaces equipped with a \emph{prevaluation} into $\hat{\O}_\Sigma$.  For the technical core of the paper, we show that there is an adjoint pair of functors $\mathcal{L}, \mathcal{R}$ which relate $\Vect_\Sigma$ to the category $\Sh_{Y(\Sigma)}^M$ of $T_N$-equivariant sheaves on the toric variety $Y(\Sigma)$.  Locally for $\sigma \in \Sigma$, the operation $\mathcal{R}: \Vect_\sigma \to \Sh_{Y(\sigma)}^M$ is akin to taking the Rees algebra of a filtration, and $\mathcal{L}: \Sh_{Y(\sigma)}^M \to \Vect_\sigma$ is its left adjoint. This is related to work of Perling \cite{Perling} and Klyachko \cite{Klyachko} by the natural relationship between filtrations and prevaluations (see Section \ref{background}).  For a general $T_N$-sheaf $\F$, the corresponding prevaluations may take values in functions from $\hat{\O}_\Sigma$ which have an infinite number of domains of linearity, or may take infinite values, depending on the algebraic properties of $\F$.  For nicely behaved sheaves such as those which are locally projective or locally free, there is no such issue and we get a prevaluation with values in $\O_\Sigma$.  On these sheaves and a broader category of sheaves which we call \emph{eversive} sheaves, the functors $\mathcal{L}$ and $\mathcal{R}$ become an equivalence, and we obtain classification results for toric sheaves by prevaluations. It is shown in Example \ref{ex-reflexive} that reflexive sheaves are eversive but not vice versa. Here the term ``eversive" comes from the word ``eversion," meaning the act of turning something inside out.  The term is meant to suggest taking the prevaluation associated to a sheaf is like revealing inside of the sheaf as in turning it inside out.\\

\noindent{\bf Acknowledgements:} We would like to thank Laura Escobar and Megumi Harada for patiently reading the paper and giving valuable comments. We also thank Sam Payne for useful conversations. We thank Courtney George for a much clearer formulation of Corollary \ref{cor-uniform}. 

\vspace{.5cm}

\noindent{\bf Notation:}
\begin{itemize}
\item $N$, a finite rank lattice with dual lattice $M$. We let $N_\R = N \otimes \R$ and $M_\R = M \otimes \R$.
\item $\rho$, an element of $N$.
\item $\Sigma$, a rational polyhedral fan in $N_\Q$.
\item $\sigma$, a rational polyhedral cone in $N_\Q$, usually a face of $\Sigma$.
\item $\varrho$, a ray in $\Sigma$, namely an element of $\Sigma(1)$.
\item $\O_N$, the semifield of integral piecewise linear functions on $N_\R$ with the operations $\min$ and $+$.
\item $\O_\Sigma$, the semifield of piecewise linear functions on $|\Sigma| \subset N_\R$.
\item $Y(\sigma)$, the affine toric variety corresponding to a polyhedral cone $\sigma \subset N_\R$.
\item $S_\sigma$, the coordinate ring of $Y(\sigma)$.
\item $Y(\Sigma)$, the toric variety corresponding to a rational polyhedral fan $\Sigma$.
\item $\v$, a valuation or quasivaluation (Section \ref{background}).
\item $\gr_\v(A)$, the associated graded algebra of a quasivaluation $\v$ (Section \ref{background}).
\item $\Trop(I)$, the tropical variety of $I \subset \k[y_1, \ldots, y_m]$ (Section \ref{background}).
\item $\Vect_\Sigma$, the category of prevalued vector spaces with values in $\O_\Sigma$ (Section \ref{category}).
\item $\Mod^M_{S_\sigma}$, the category of $M$-graded $S_\sigma$-modules (Section \ref{category}).
\item $\Sh^M_{Y(\Sigma)}$, the category of $T_N$-equivariant quasicoherent sheaves on $Y(\Sigma)$ (Section \ref{category}).
\item $\mathcal{B} \subset \k[y_1, \ldots, y_m]/I \cong A$, the image of the polynomial generators $\{y_1, \ldots, y_m\}$ under the presentation map $\k[y_1, \ldots, y_m] \to A$ (Section \ref{Khovanskii}).
\item $\mathcal{K}(I)$, the set of quasivaluations with Khovanskii basis $\mathcal{B}$  (Section \ref{Khovanskii}).  
\item $\GF(I)$, the Gr\"obner fan of $I \subset \k[y_1, \ldots, y_m]$ (Section \ref{Khovanskii}). 
\item $\Phi: |\Sigma| \to \Trop(L)$, a piecewise linear map to the tropicalized linear space associated to a linear ideal $L \subset \k[y_1, \ldots, y_m]$ (Section \ref{Khovanskii}).
\item $D(\Phi)$, the diagram of $\Phi$ (Section \ref{Khovanskii}).
\item $\Delta(\Sigma, L)$, the set of diagrams adapted to $\Sigma$ with rows in $\Trop(L)$ (Section \ref{StrongKhovanskii}).
\item $\mathcal{R}(A, \hat{\v})$, the total section ring associated to a valuation $\v: A \to \O_\Sigma$ (Section \ref{StrongKhovanskii}).
\item $\mathcal{R}(\E)$, the Cox ring of the projectivized toric vector bundle $\mathbb{P}\E$ (Section \ref{sec-Cox-TVB}).
\end{itemize}

\section{Background on valuations}\label{background}
In this section we collect the basic definitions and background material for prevaluations on vector spaces, and valuations and quasivaluations on algebras. We start by introducing the notion of a \emph{valuation} on an algebra with values in an idempotent semialgebra. We call a valuation on a vector space a \emph{prevaluation}. 
The term prevaluation appears in \cite[Section 2.1]{KK} where the notion of a prevaluation with values in a totally ordered group is defined. 

\begin{definition}
A \emph{semialgebra} $(\O, \oplus, \odot)$ is a set $\O$ with two binary operations $\oplus$ and $\odot$ that satsify the same axioms as addition and multiplication in a ring, with the exception that there are not necessarily additive inverses. We denote the identity element with respect to $\oplus$ by $\infty$. If all the elements beside the additive inverse $\infty$ have multiplicative inverses then $\O$ is called a \emph{semifield}. 
A semialgebra $\O$ is called \emph{idempotent} if $a \oplus a = a$ for all $a \in \O$. 
\end{definition}

First important example of an idempotent  semifield is the set $\overline{\Q} = \Q \cup \{\infty\}$ together with the operations $a \oplus b = \min(a, b)$ and $a \odot b = a + b$. It is usually referred to as the \emph{tropical semifield}. Clearly $\overline{\Z}=\Z \cup \{\infty\}$ is a subsemifield of $\overline{\Q}$. 

The following are important examples of (idempotent) semifields and semialgebra in this paper. 
\begin{example}[Semifield of functions on a set]
Let $\O_X$ be the set of $\overline{\Q}$-valued functions on a set $X$ with addition $\oplus$ and multiplication $\odot$ to be the pointwise $\min$ and $+$ of functions respectively. The additive identity, also denoted $\infty$, is the function which assigns $\infty$ to every point. 
\end{example}

\begin{example}[Semifield of piecewise linear functions on a lattice] \label{ex-PL-semifield}
Let $N$ be a lattice, that is, a free abelian group of finite rank. Ley $N_\Q = N \otimes \Q$. Recall that a \emph{piecewise linear function} on $N_\Q$ is a function $f: N_\Q \to \Q$ for which there is a complete fan in $N_\Q$ such that $f$ is linear restricted to each cone in the fan. A piecewise linear function $f$ is called \emph{integral} if it maps $N$ to $\Z$. The set of all piecewise linear functions on $N_\Q$ together with operations of $\min$ and $+$ of functions is a semifield. The set of integral piecewise linear functions is a subsemifield which we denote by $\O_N$. 

More generally, if $\Sigma$ is a (not necessarily complete) fan in $N_\Q$, we define $\O_\Sigma$ to be the semifield of all integral piecewise linear function on $|\Sigma|$, the support of $\Sigma$.
This semifield plays a most important role in this paper.  We will also need the semialgebra $\hat{\O}_\Sigma$ obtained by taking limits of elements in $\O_\Sigma$. More precisely, $\hat{\O}_\Sigma$ is the semialgebra of functions $\psi: |\Sigma| \to \overline{\Q}$ such that $\psi(\ell\rho) = \ell \psi(\rho)$ for all $\ell \in \Q_{\geq 0}$ and $\rho \in |\Sigma|$. Note that a function $\psi \in \hat{\O}_\Sigma$ is allowed to attain infinite values on points of $\Sigma$. Hence $\hat{\O}_\Sigma$ is not a semifield. 
\end{example}

\begin{example}[Semialgebra of convex polytopes]
Let $M$ be the dual lattice to $N$. The set of all convex polytopes in $M_\R = M \otimes \R$ has a natural structure of a semialgebra. The sum of two polyotpes is defined to be the convex hull of their union, and the product of two polytopes is defined to be their Minkowski sum. The subset of lattice polytopes is a subsemialgebra. Recall that a lattice polytopes is a polytope whose vertices lie in $M$. The semialgebra of convex polytopes in $M_\R$ can be identified with the semialgebra of concave piecewise linear functions. The identification is given by sending a convex polytope to its \emph{support function}.    
\end{example}


Any idempotent semialgebra $\O$ has an intrinsic partial ordering $\succeq$ defined by $a \succeq b$ if $a \oplus b = b$. 

\begin{definition}   \label{def-valuation}
For a $\k$-algebra $A$, a \emph{quasivaluation} $\v: A \to \O$ is a function which satisfies the following:
\begin{enumerate}
\item $\v(fg) \succeq \v(f) \odot \v(g),$
\item $\v(f + g) \succeq \v(f) \oplus \v(g),$
\item $\v(C) = 0$, for all $C \in \k \setminus \{0\},$
\item $\v(0) = \infty.$
\end{enumerate}
A quasivaluation $\v$ is said to be a \emph{valuation} if $\v(fg) = \v(f) \odot \v(g)$. A \emph{prevaluation} on a vector space $E$ is a function which satisfies $(2)-(4)$ above. 
\end{definition}

Let $A$ be a $\k$-algebra with finite generating set $\mathcal{B}=\{b_1, \ldots, b_d\} \subset A$.  We let $\pi: \k[\bx] \to A$ be the associated presentation, with ideal $I_\mathcal{B} = \ker(\pi)$.  For any monomial $\bx^\alpha \in \k[\bx]$ with $\alpha = (\alpha_1, \ldots, \alpha_d)$ there is a function $ev_{\bx^\alpha}: \O^\mathcal{B} \to \O$ defined by $$ev_{\bx^\alpha}: (\psi_1, \ldots, \psi_d) \mapsto \bigodot_{i =1}^d \psi_i^{\odot \alpha_i}.$$ Following \cite[5.1]{Giansiracusa}, the tropical variety $\Trop_\O(I_\mathcal{B}) \subset \O^\mathcal{B}$ is defined to be the set of tuples $(\psi_1, \ldots, \psi_d)$ such that for any polynomial $\sum_{j = 1}^m C_j\bx^{\alpha(j)} \in I_\mathcal{B}$ we have:

\begin{equation}\label{bending}
 \bigoplus_{j=1}^m ev_{\bx^{\alpha(j)}} = \bigoplus_{j \neq i} ev_{\bx^{\alpha(j)}}
\end{equation}
for any $1 \leq i \leq m$.  The following is a well-known relationship between valuations and the tropical variety (see \cite{Payne, Giansiracusa}). 

\begin{proposition}\label{trop}
Let $\v: A \to \O$ be a valuation and $\mathcal{B} = \{b_1, \ldots, b_d\} \subset A$ a generating set. Then the tuple $(\v(b_1), \ldots \v(b_d)) \in \O^\mathcal{B}$ is a point in the tropical variety $\Trop_{\O}(I_\mathcal{B})$.  
\end{proposition}

If $\v: E \to \O$ is a prevaluation and $\psi \in \O$, we let $F_\psi(\v) = \{f \mid \v(f) \geq \psi\}$. It is straightforward to verify that $F_\psi(\v)$ is a subspace of $E$.  We use a different notation to distinguish the special case $\O = \overline{\Q}$, setting $G_r(\v) = \{f \mid \v(f) \geq r\}$ for $r \in \overline{\Q}$. Similarly, we let $G_{> r}(\v) = \{f \mid \v(f) > r \}$. The spaces $G_r(\v)$ fit together to form a $\Q$ filtration of $E$ by $\k$ vector spaces. The \emph{associated graded} vector space $\gr_\v(E)$ is defined to be the direct sum:
\begin{equation}  \label{equ-gr-E}
\gr_\v(E) = \bigoplus_{r \in \Q} G_r(\v)/G_{> r}(\v).
\end{equation} 
If $\v: A \to \overline{\Q}$ is a quasivaluation, it is straightforward to show that $\gr_\v(A)$ is $\k$-algebra.   The function $\v$ can be recovered as $\v(f) = \max\{ r \mid f \in G_r(\v) \}$, where $\v(f)$ is taken to be $\infty$ if the maximum is never attained.  The following definition is from \cite[Section 2.5]{Kaveh-Manon-NOK}.

\begin{definition}\label{adaptedbasis}
A vector space basis $\B \subset E$ is said to be an \emph{adapted basis} for a prevaluation $\v: E \to \overline{\Q}$ if $\B \cap G_r(\v)$ is a basis of $G_r(\v)$ for each $r \in \Q$. 
\end{definition}
If $b_i$, $i =1, \ldots k$ are from an adapted basis then $\v(\sum_{i =1}^k C_i b_i) = \oplus_{i =1}^k \v(b_i)$. This identity simplifies computations for prevaluations with adapted bases.  

We finish this section by recalling the notion of Khovanskii basis for a quasivaluation on an algebra with values in $\overline{\Q}$. Khovanskii bases for quasivaluations with values in a free abelian group are the subject of \cite{Kaveh-Manon-NOK}. 

\begin{definition}\label{khovanskiibasisdefinition}
Let $A$ be a $\k$-algebra, and $\v: A \to \overline{\Q}$ be a quasivaluation. A subset $\mathcal{B} \subset A$ is said to be a \emph{Khovanskii basis} of $\v$ if the equivalence classes $\overline{\mathcal{B}} \subset \gr_\v(A)$ generate $\gr_\v(A)$ as a $\k$-algebra. 
\end{definition}

\section{Equivariant sheaves and the category $\Vect_\Sigma$}\label{category}

Recall that $\O_\Sigma$ is the semifield of piecewise linear functions with a finite number of linear domains defined on $|\Sigma|$, and $\hat{\O}_\Sigma$ is the semialgebra of functions $\psi: |\Sigma| \to \overline{\Q}$ such that $\psi(\ell\rho) = \ell \psi(\rho)$ for all $\ell \in \Q_{\geq 0}$ and $\rho \in |\Sigma|$. A function $\psi \in \hat{\O}_\Sigma$ is allowed to attain infinite values on points of $\Sigma$ and thus $\hat{\O}_\Sigma$ is not a semifield.  We regard elements in $\hat{\O}_\Sigma$ as limits of elements in $\O_\Sigma$ (see the end of Example \ref{ex-PL-semifield}). 

In this section, we define a category $\Vect_\Sigma$ of vector spaces equipped with a \emph{prevaluation} into $\hat{\O}_\Sigma$.  

\begin{definition}\label{prevaluedvectorspaces}
Let $\Vect_\Sigma$ be the category whose objects are pairs $(E, \v)$, where $E$ is $\k$-vector space and $\v: E \to \hat{\O}_\Sigma$ is a prevaluation over $\k$. We call a pair $(E, \v)$ a \emph{prevalued vector space}. A morphism $\phi: (E, \v) \to (D, \w)$ of prevalued vector spaces is a $\k$-linear map $\phi: E \to D$ such that $\v(f) \leq \w(\phi(f))$ for all $f \in E$. We note that, strictly speaking, $\Vect_\Sigma$ only depends on the support of $\Sigma$.  
\end{definition}

\begin{definition}
For any prevalued vector space $(E, \v)$ and $\rho \in |\Sigma|$ we let $\v_\rho: E \to \overline{\Q}$ denote the prevaluation obtained by composition with $\rho$: $$\v_\rho(f) = \v(f)(\rho).$$ We let $G^\rho_r(\v) = \{ f \mid \v(f)(\rho) \geq r\}$. 
\end{definition}

Observe that $G^\rho_r(\v)$ is the $r$-th filtration space $G_r(\v_\rho)$ of the $\overline{\Q}$-filtration associated to $\v_\rho$.  The spaces $G^\rho_r(\v)$ can also be used to recover the prevaluation $\v$ as $\v(f)(\rho) = \max\{r \mid f \in G^\rho_r(\v)\}$.  We sometimes refer to the spaces $G^\rho_r(\v)$ as the \emph{Klyachko spaces} of $\v$. The filtrations defined by these subspaces are the main players in Klyachko's classification of torus equivariant vector bundles on toric varieties \cite{Klyachko}. 

\begin{definition}
We let $\Sh_{Y(\Sigma)}^M$ denote the category of $T_N$-equivariant quasi-coherent sheaves of $\mathcal{O}_{Y(\Sigma)}$-modules on the toric variety $Y(\Sigma)$.  
\end{definition}

We will construct a functor $\L: \Sh_{Y(\Sigma)}^M \to \Vect_\Sigma$, and show
that a sheaf $\F \in \Sh_{Y(\Sigma)}^M$ is determined by its image, the prevalued vector space $\L(\F)$, if it is from a certain full subcategory $\Sh_{Y(\Sigma)}^{M, ev}$. We do this by constructing an adjoint functor $\mathcal{R}$ to $\mathcal{L}$ over each affine toric chart (Section \ref{functors}). This construction enables us to prove Theorem \ref{thm-main-family} and show that a toric flat family $\pi: \mathcal{X} \to Y(\Sigma)$ is determined by a finite combinatorial data which we call the \emph{diagram} of $\pi$ in Section \ref{Khovanskii}. 

\subsection{Properties of $\Vect_\Sigma$}
We need to define a few constructions and operations on prevalued vector spaces. 

\begin{definition} 
For a vector space $E$ and $\psi \in \hat{\O}_\Sigma$, we let $(E, \psi) \in \Vect_\Sigma$ denote the vector space equipped with the prevaluation which assigns to every non-zero element of $E$ the function $\psi$.
\end{definition}

\begin{definition}\label{def-directsum}
For $(E, \v), (D, \w) \in \Vect_\Sigma$ the \emph{direct sum} is $(E, \v) \oplus (D, \w) = (E \oplus D, \v \oplus \w)$, where $(\v \oplus \w )(f + g) = \v(f) \oplus \w(g) \in \hat{\O}_\Sigma$. 
\end{definition}

It is straightforward to show that $\oplus$ is both a product and a coproduct in $\Vect_\Sigma$ with the trivial prevalued vetor space $(\{0\}, \infty)$ as the identity object. 
We use $\oplus$ to extend the notion of adapted basis for pervaluations with values in $\overline{\Q}$ (Definition \ref{adaptedbasis}) to prevaluations with values in $\hat{\O}_\Sigma$. 

\begin{definition}\label{linearadaptedbasis}
We say $(E, \v)$ has an \emph{adapted basis} $\B \subset E$ if the natural maps $(\k, \v(b)) \to (E, \v)$, $1 \to b$ define an isomorphism $(E, \v) \cong \bigoplus_{b \in \B} (\k, \v(b))$. Here $(\k, \v(b))$ denotes the $1$-dimensional $\k$-vector space equipped with the prevaluation that assigns $\v(b)$ to every nonzero element. The set $\B$ is said to be a \emph{linear adapted basis} if, for every $b \in \B$, $\v(b)$ is the restriction of a linear function in $M$ to $|\Sigma|$. 
\end{definition}

By Definition \ref{def-directsum}, $\B$ is an adapted basis of $(E, \v)$ if and only if $\v(\sum_i C_i b_i) = \min\{\v(b_i) \mid C_i \neq 0 \}$ for any vector $\sum_i C_ib_i \in E$.

\begin{definition}\label{def-pushandpull}
Let $(E, \v)\in \Vect_\Sigma$ and $\pi: E \to D$ be a surjection of vector spaces. The \emph{pushforward} prevaluation $\pi_*(\v): D \to \hat{\O}_\Sigma$ is defined by: $$\pi_*(\v)(g)(\rho) = \max\{ \v(f)(\rho) \mid \pi(f) = g\},\quad \forall \rho \in |\Sigma|,$$where it is understood that $\pi_*(\v)(g)(\rho) = \infty$ if the maximum is never attained. Similarly, given $\phi: F \to E$, we have the \emph{pullback} $\phi^*(\v)$, defined by $\phi^*(\v)(f) = \v(\phi(f))$.  
\end{definition}

The colimit $\colim (E_i, \v_i)$, where $i$ is from an index category $I$, is constructed by taking $\colim \v_i: \colim E_i \to \hat{\O}_\Sigma$ to be the pushforward of the prevaluation on $\bigoplus_{i \in I} (E_i, \v_i)$ under the quotient map $\pi_I: \bigoplus_{i \in I} E_i \to \colim E_i$. In particular, $(\colim \v_i)(f) = \max\{\oplus_{i =1}^\ell \v_i(f_i) \mid \sum_{i =1}^\ell \pi_I(f_i) = f\}$.   

\begin{definition}
The \emph{tensor product} $(E_1, \v_1) \otimes (E_2, \v_2)$ is the pair $(E_1\otimes_\k E_2, \v_1 \star \v_2)$, where $\v_1\star \v_2: E_1 \otimes_\k E_2 \to \hat{\O}_\Sigma$ is the prevaluation with associated subspaces $G^\rho_r(\v_1 \star \v_2)$ defined as follows:
\begin{equation}
G^\rho_r(\v_1 \star \v_2) = \sum_{s + t = r} G^\rho_s(\v_1) \otimes_\k G^\rho_t(\v_2) \subset E_1 \otimes_\k E_2.
\end{equation}
That is, $(\v_1\star \v_2)(f)(\rho) = \max\{r \mid f \in G^\rho_r(\v_1 \star \v_2)\}$, for $\rho \in |\Sigma|$, $f \in E_1 \otimes_\k E_2$.
\end{definition}
\noindent
The prevalued vector space $(\k, 0)$ is the multiplicative identity object in $\Vect_\Sigma$.  Also, for any simple tensor $x\otimes y \in E_1 \otimes E_2$ we have $(\v_1 \star \v_2)(x \otimes y ) = \v_1(x) \odot \v_2(y)$.

 It is straightforward to show that a commutative algebra object $(A, \v) \in \Vect_\Sigma$ is the same information as a commutative $\k$-algebra $A$ equipped with a quasivaluation $\v: A \to \hat{\O}_\Sigma$.   We let $\Alg_\Sigma$ denote the category of commutative algebra objects in $\Vect_\Sigma$.  Tensor product also allows us to make sense of Schur functors in $\Vect_\Sigma$.

\begin{definition}\label{def-schur}
Let $\lambda$ be a partition of $\{1, \ldots, n\}$. The \emph{Schur functor} $S_\lambda: \Vect_\Sigma \to \Vect_\Sigma$ takes a prevalued vector space $(E, \v)$ to $(S_\lambda(E), s_\lambda(\v)) \subset (E, \v)^{\otimes n}$, 
where $s_\lambda(\v)$ is the pullback of $\v^{\star n}$ under the inclusion map $S_\lambda(E) \subset E^{\otimes n}$. 
\end{definition}

The categories we deal with in this paper are \emph{symmetric monoidal} and \emph{Cauchy-complete} (see \cite{nlab:schur}).  It follows that the functor $S_\lambda$ makes sense in any of these categories, and all strictly monoidal functors we consider commute with any $S_\lambda$.    

\subsection{The functors $\L$ and $\mathcal R$}\label{functors} 
In this section, we introduce a functor $\L$ from the category of $T_N$-equivariant sheaves to the category of prevalued vector spaces. This construction is a generalization of Klyachko's construction which assigns a collection of filtrations to a toric vector bundle. Klyachko's construction itself is a generalization of assigning a  piecewise linear function to a toric line bundle.

When the base toric variety is affine, we also define an adjoint functor $\mathcal{R}$ from prevalued vector spaces to $T_N$- equivariant sheaves. That $\L$ and $\mathcal{R}$ are adjoint functors is a main result of this section (Theorem \ref{thm-adjoint}). The proof is postponed to Section \ref{proofs}.  We use this to show that $\L$ gives an equivalence of categories between certain subcategory of $T_N$-equivariant sheaves on a toric variety (not necessarily affine) and a full subcategroy of prevalued vector spaces (Theorem \ref{eversiveglobalequivalence}). 

As usual $Y(\Sigma)$ denotes a toric variety with fan $\Sigma$. We fix a point in the open orbit and identify the open orbit with the torus $T_N$. For $\sigma \in \Sigma$ we denote the corresponding toric affine chart by $Y(\sigma)$.

First we consider the affine case. A cone $\sigma \in \Sigma$ determines an affine semigroup $\sigma^\vee \cap M$, where $\sigma^\vee = \{u \mid \langle \rho, u \rangle \leq 0 \ \ \forall \rho \in \sigma \} \subset \M$ is the dual cone of $\sigma$. We let $S_0 = \k[T_N]$ denote the coordinate ring of $T_N$, and $S_\sigma \subset S_0$ denotes the affine semigroup algebra associated to $\sigma^\vee$. 

\begin{remark}   \label{rem-dual-cone}
Our definition for the dual cone $\sigma^\vee$ is the negative of the convention found in the literature on toric varieties (e.g \cite{CLS} and \cite{GBCP}). This is to conform with the $\min$ convention for valuations and tropical geometry. 
\end{remark}

Let $\Mod_{S_\sigma}^M$ denote the category of $M$-graded $S_\sigma$-modules. Let $R_m \subset R$ denote the $T_N$-isotypical component associated to a character $m \in M$, so that $R = \bigoplus_{m \in M} R_m$. If $\sigma \in \Sigma$ is a cone with affine open toric chart $Y(\sigma) \subset Y(\Sigma)$ and $\F\in \Sh_{Y(\Sigma)}^M$, then $\Gamma(Y(\sigma), \F) \in \Mod_{S_\sigma}^M$.  
The maximal ideal $\m \subset S_\sigma$ defining the identity element in the open orbit $T_N \subset Y(\sigma)$ is the ideal generated by the forms $\chi_u -1$ for $u \in \sigma^\vee \cap M$. For a module $R \in \Mod_{S_\sigma}^M$ let $E_R = R/\m R$, and let $\phi_R: R \to E_R$ be the natural surjection. 

We now explain the construction of a valuation $\v_R$ associated to a module $R \in \Mod_{S_\sigma}^M$.
We define $F_m(R) \subset E_R$ to be the image of $R_m$ under $\phi_R$. For $\rho \in \sigma \cap N$ and $r \in \Q$ let $G^\rho_r(R) \subset E_R$ be the subspace defined by:

\begin{equation}
G^\rho_r(R) = \sum_{\langle \rho, m \rangle \geq r} F_m(R).
\end{equation}
Observe that $G^\rho_r(R) \supseteq G^\rho_s(R)$ whenever $r \leq s$. So for each $\rho \in \sigma \cap N$ we have a decreasing filtration $\{G^\rho_r(R)\}_{r \in \Q}$. 


\begin{definition}[Functor $\mathcal{L}$]  \label{def-functor-L}
For $f \in E_R$, we let $\v_R(f): \sigma \cap N \to \overline{\Q}$ be the prevaluation defined by $$\v_R(f)(\rho) = \max\{r \mid f \in G^\rho_r(R)\}, ~~\forall \rho \in \sigma \cap N.$$ 
We define $\L(R)$ to be the prevalued vector space $(E_R, \v_R) \in \Vect_\sigma$. 
\end{definition}

Next, for any $(E, \v) \in \Vect_\sigma$ and $m \in M$ we can consider the space $F_m(\v) = \{ f \mid \v(f) \geq m \}$.   If $u \in \sigma^\vee \cap M$ then $m \geq m + u$ as functions on $\sigma$, so $F_m(\v) \subset F_{m + u}(\v)$. 

\begin{definition}[Functor $\mathcal{R}$]
For an object $(E, \v)$ of $\Vect_\sigma$ we define the Rees module $\mathcal{R}(E, \v) \in \Mod_{S_\sigma}^M$ to be the following $M$-graded vector space:
\begin{equation}
\mathcal{R}(E, \v) = \bigoplus_{m \in M} F_m(\v).
\end{equation} 
We let $\chi_u \in S_\sigma$ act on $\mathcal{R}(E, \v)$ by the inclusion map: $F_m(\v) \subset F_{m+u}(\v)$. 
\end{definition}

\begin{theorem}\label{thm-adjoint}
The constructions $\L: \Mod_{S_\sigma}^M \to \Vect_\sigma$ and $\mathcal{R}: \Vect_\sigma \to \Mod_{S_\sigma}^M$ define an adjoint pair of functors. 
\end{theorem}

The proof of Theorem \ref{thm-adjoint} can be found in Section \ref{proofs}. As part of the data of the adjunction, we obtain maps $\eta:\mathcal{L}\mathcal{R}(E, \v) \to (E, \v)$ and $\epsilon: R \to \mathcal{R}\mathcal{L}(R)$, these can be described as follows. For $(E, \v) \in \Vect_\sigma$, the space $\mathcal{L}\mathcal{R}(E, \v)$ has underlying vector space $\colim F_m(\v) \subset E$ and prevaluation defined by the spaces $G^\rho_r(\mathcal{R}(E, \v)) = \sum_{\langle \rho, m \rangle \geq r} F_m(\v) \subset G^\rho_r(\v)$.  This gives the map $\eta$.   On the other hand, the object $\mathcal{R}\mathcal{L}(R)$ is the $M$-graded module $\bigoplus_{m \in M} F_m(\v_R)$, where $F_m(\v_R) = \bigcap_{\rho \in \sigma \cap N} G^\rho_{\langle \rho, m \rangle}(R)$. This gives the map $\epsilon$.   

\begin{definition}\label{def-eversive}
We say $(E, \v)$  (respectively $R$) is \emph{eversive} if the map $\eta$ (respectively $\epsilon$) is an isomorphism. We let $\Vect^{ev}_\sigma \subset \Vect_\sigma$ and $\Mod_{S_\sigma}^{M, ev} \subset \Mod_{S_\sigma}^M$ denote the full subcategories on eversive objects. 
\end{definition}

\begin{proposition}\label{eversive}
For any $R \in \Mod_{S_\sigma}^M$ and $(E, \v) \in \Vect_\sigma$, both $\mathcal{R}(E, \v)$ and $\mathcal{L}(R)$ are eversive.  Furthermore, $\epsilon$ and $\eta$ are isomorphisms if and only if the corresponding objects are eversive.  The functors $\mathcal{R}$ and $\mathcal{L}$ define an equivalence of categories $\Vect_\sigma^{ev} \cong \Mod_{S_\sigma}^{M, ev}$. 
\end{proposition}

\begin{example}[Eversive toric ideals]\label{ex-eversiveideals}
We give a description of the $M$-graded ideals which are eversive. Let $I \subset S_\sigma$ be a $M$-graded ideal, and let $\Omega(I) \subset M \cap \sigma^\vee$ be the set of characters in $I$.  For a subset $X \subset M$ we define the $\sigma-$hull $|X|_\sigma$ to be the set of $m$ with the property that $\langle \rho , m \rangle \leq r$ whenever $\langle \rho, v \rangle \leq r$ for all $v \in X$. We claim that $I$ is eversive if and only if $\Omega(I) = |\Omega(I)|_\sigma$. 

First, $E_I \cong \k$, and for any $\rho \in \sigma$, $G^\rho_r(I) \subset E_I$ is non-zero if and only if there is some $u \in \Omega(I)$ with $\langle \rho, u \rangle \geq r$. Now, for any $m \in M$, the $m$-isotypical space of $\mathcal{R}\L(I)$ is $F_m(\v_I) = \bigcap_{\langle \rho, m \rangle = r} G^\rho_r(I)$. The latter is nonzero if and only if for all $\rho \in \sigma$, $r\in \Q$ such that $\langle \rho, m \rangle =r$ there exists a $u \in \Omega(I)$ such that $\langle \rho, u \rangle \geq r$.  As a consequence, $F_m(\v_I)$ is non-zero if and only if for all $\rho \in \sigma$, $\langle \rho, m \rangle$ is less than the largest value $\rho$ obtains on $\Omega(I)$. This second condition is equivalent to $m \in |\Omega(I)|_\sigma$. 
\end{example}

\begin{example}[Reflexive sheaves vs eversive sheaves]\label{ex-reflexive}
Recall that a sheaf $\F$ is said to be \emph{reflexive} if it is isomorphic to its own double-dual: $\F \cong (\F^\vee)^\vee$.  In \cite[Section 5.5]{Perling}, Perling characterizes $T-$linearized reflexive sheaves as those sheaves which are determined by the filtrations defined by the ray generators.  In particular, over an affine toric variety $Y(\sigma)$, $R = \bigoplus_{m \in M} F_m$ is reflexive if and only if $F_m = \bigcap_{\rho_i \in \sigma(1)} G^{\rho_i}_{\langle \rho_i, m\rangle}$. Observe that we always have the inclusions $F_m \subset F_m(\v_R) \subset \bigcap_{\rho_i \in \sigma(1)} G^{\rho_i}_{\langle \rho_i, m\rangle}$, hence if $R$ is reflexive these inclusions must be equalities, and $R$ must also be eversive by Definition \ref{def-eversive}.  

On the other hand, Example \ref{ex-eversiveideals} provides a way to find eversive sheaves which are not reflexive.  We let $I = \langle xy^4, x^2y^2, x^4y\rangle \subset \k[x, y]$; this ideal defines an eversive sheaf by Example \ref{ex-eversiveideals}.  However, in this case $F_{(1, 1)} = 0 \neq G^{(-1, 0)}_{-1} \cap G^{(0, -1)}_{-1} = \k$, so $I$ cannot be reflexive.  
\end{example}

We let $\Alg_{S_\sigma}^M$ be the category of commutative algebras in $\Mod_{S_\sigma}^M$.  If $R \in \Alg_{S_\sigma}^M$ then $R/\m R \cong E_R$ is a $\k$-algebra, and $\v_R$ is a quasivaluation on $E_R$.  Similarly, $\mathcal{R}(A, \v)$ is an $S_\sigma$-algebra if $(A, \v) \in \Alg_\sigma$.   Let $\Alg_\sigma^{ev}$ and $\Alg_{S_\sigma}^{M, ev}$ be the categories of algebras whose underlying objects are eversive, then it is straightforward to show that $\Alg_\sigma^{ev} \cong \Alg_{S_\sigma}^{M, ev}$.

Finally, we construct the functor $\L$ for $T_N$-equivariant quasicoherent sheaves $\Sh_{Y(\Sigma)}^M$. 
\begin{definition}\label{eversivesheaf}
Let $\Sh^{M, ev}_{Y(\Sigma)}$ be the full subcategory of $\Sh^M_{Y(\Sigma)}$ consisting of sheaves $\F$ such that $\Gamma(Y(\sigma), \F) \in \Mod_{S_\sigma}^{M, ev}$ for each $\sigma \in \Sigma$. 
\end{definition}

Recall that a functor $F: (\mathcal{C}, \otimes)\to (\mathcal{D}, \boxtimes)$ between symmetric monoidal categories is said to be \emph{strictly monoidal} if $F(A \otimes B)$ is naturally isomorphic to $F(A) \boxtimes F(B)$.

\begin{theorem}\label{eversiveglobalequivalence}
There is a monoidal functor $\mathcal{L}: \Sh^M_{Y(\Sigma)} \to \Vect_\Sigma$.  This functor gives an equivalence of $\Sh^{M, ev}_{Y(\Sigma)}$ and a full subcategory of $\Vect_\Sigma$.  
\end{theorem}

The functor $\L$ allows us to describe morphisms between equivariant sheaves using prevaluations.  In the next corollary we give such a description of the space of maps when the domain is the invertible sheaf corresponding to a $T_N$-Cartier divisor. We say that a function $\psi: \Sigma \cap N \to \Z$ is {\it piecewise linear  on the fan $\Sigma$} if the restriction $\psi\!\mid_\sigma$ is linear for each cone $\sigma \in \Sigma$. It is well-known (see \cite[Chapter 3]{Fulton}) that piecewise linear functions on $\Sigma$ correspond to $T_N$-Cartier divisors on $Y(\Sigma)$.  For such a $\psi$, we let $\O(\psi)$ denote the associated invertible sheaf on $Y(\Sigma)$.  In particular, for $m \in M$, $\O(m)$ denotes the linearization of the structure sheaf of $Y(\Sigma)$ corresponding to the character $\chi_m$.  Also, recall that for any $(E,\v) \in \Vect_\Sigma$, $F_\psi(\v) \subset E$ denotes the space of those $f \in E$ such that $\v(f) \geq \psi$.  

\begin{corollary}\label{cor-invertiblehom}
Let $\psi$ be a piecewise linear function on $\Sigma$, and let $\F \in \Sh_{Y(\Sigma)}^{M, ev}$ with $\mathcal{L}(\F) = (E, \v)$. Then $\Hom_{\Sh_{Y(\Sigma)}^M}(\O(\psi), \F)$ is isomorphic to $F_\psi(\v)$. 
\end{corollary}

\subsection{Connection to Perling's $\Sigma$-families}
In \cite{Perling}, Perling proves that the category $\Sh_{Y(\Sigma)}^M$ is equivalent to the category of \emph{$\Sigma$-families}.  In the affine case, the $\sigma$-family of a module $R \in \Mod_{S_\sigma}^M$ is the directed system formed by the isotypical spaces $R_m$ for $m \in M$ and the linear maps $\chi_u: R_m \to R_{m + u}$ defined by the characters in $S_\sigma$.  

In the case that $R$ is eversive, the functor $\L: \Mod_{S_\sigma}^M \to \Vect_\sigma$ recovers the data of the $\sigma$-family as the spaces $F_m(\v_R) \subset E_R$, along with the inclusions $F_m(\v_R) \subset F_{m + u}(\v_R)$.  However, $\L$ can forget some of the data of the $\sigma$-family when $R$ is not eversive. For example, if $R$ is torsion, $E_R$ and all spaces $F_m(\v_R)$ are $0$.  Example \ref{ex-eversiveideals} also gives a way to produce $M$-graded ideals $I \subset S_\sigma$ where $\L(I)$ forgets some of the data of the associated $\sigma$-family.

\subsection{Free, projective, and flat modules}\label{freesection}
In this section we see how the functor $\L$ behaves on free, projective and flat modules.
It is critical for the proof of our classification theorem (Theorem \ref{thm-main-family}). In particular, we show that projective modules are eversive.

We start by observing that a direct sum is eversive if and only if the components are.
We leave the proof to the reader. 
\begin{proposition}\label{eversivesum}
In $\Mod_{S_\sigma}^M$ and $\Vect_\sigma$, a direct sum is eversive if and only if its components are eversive. 
\end{proposition}

Let $P$ be a free $M$-graded $S_\sigma$-module, then we can write $P = \bigoplus_{b_i \in \B} S_\sigma \, b_i$ for $\B$ an $M$-homogeneous basis. We let $\deg(b_i) = \lambda_i \in M$.  For any $m \in M$ we have a direct sum decomposition $P_m = \bigoplus_{u_i + \lambda_i = m} \k\,\chi_{u_i}\, b_i$.  

\begin{proposition}\label{free}
The functors $\mathcal{L}$ and $\mathcal{R}$ give an equivalence of categories between $M$-homogeneous free modules and prevalued vector spaces with linear adapted bases (Definition \ref{linearadaptedbasis}). Furthermore, if $P = \bigoplus_{b_i \in \B} S_\sigma
, b_i$ is free with $\deg(b_i) = \lambda_i$, then $P$ is eversive, $\mathcal{L}(P) \cong \bigoplus_{b_i \in \B} (\k, \lambda_i)$, and $F_m(P)$ is isomorphic to the subspace of $E_P$ with basis $\B_m = \{b_i \mid m -\lambda_i \in \sigma^\vee \cap M\}$.
\end{proposition}

\begin{proof}
By Proposition \ref{eversivesum} and Theorem \ref{eversiveglobalequivalence} it suffices to note that if $\deg(b) = \lambda$ then $\mathcal{L}(S_\sigma\,b) = (\k, \lambda)$ and $\epsilon: \mathcal{R}\mathcal{L}(S_\sigma\, b) \cong S_{\sigma}\, b$.
\end{proof}

Propositions \ref{free} and \ref{eversivesum} imply that projective modules in $\Mod_{S_\sigma}^M$ are eversive and correspond under $\mathcal{L}$ to summands of spaces with a linear adapted basis.   From now on we identify the category of $T_N$-vector bundles on $Y(\Sigma)$ with the full subcategory of locally-free $T_N$-sheaves. 

\begin{corollary}\label{finiteequivalence}
Let $(E, \v) \in \Vect_N$ with $E$ a finite dimensional vector space. The following are equivalent.

\begin{enumerate}
\item The image $\v(E)$ is a finite subset of $\O_N$.
\item There is a finite complete fan $\Sigma$ with $(E, \v) = \mathcal{L}(\E)$ for a locally free $\E \in \Sh_{Y(\Sigma)}^M$.
\item There is a finite complete fan $\Sigma$ with $(E, \v) = \mathcal{L}(\F)$ for a coherent $\F \in \Sh_{Y(\Sigma)}^M$.
\end{enumerate} 
\end{corollary}

\begin{proof}
 Given $\v: E \to \O_N$ with $\v(E) \subset \O_N$ finite we find any fan $\Sigma$ such that $\v(f)\!\!\mid_\sigma \in M_\sigma$ for all $\sigma \in \Sigma$ and all $f \in E$. The image $\v\!\!\mid_\sigma(E \setminus \{0\}) \subset \O_\sigma$ is then a finite set $\lambda_1, \ldots, \lambda_k \in M_\sigma \subset \O_\sigma$. By construction, the partial order on $\O_\sigma$ must restrict to a total ordering on the $\lambda_i$.   We let $\gr_\sigma(E)$ be the associated graded vector space, and we choose a basis $\hat{\B}_\sigma \subset \gr_\sigma(E)$. Any lift $\B_\sigma \subset E$ of this basis is adapted to $\v\!\!\mid_\sigma$, this proves $(1) \to (2)$.  The case $(2) \to (3)$ is immediate.   For $(3) \to (1)$, consider an $S_\sigma$-module $R$ generated by $f_1, \ldots, f_N \in R$ with $\deg(f_i) = \lambda_i$. If $p = \sum_{i =1}^\ell \chi_{u_i}f_i$, then $\phi_R(p) = q \in \sum F_{\lambda_i}(R)$, so $\v_R(q) \geq \bigoplus_{i =1}^\ell \lambda_i$. Furthermore, $R_m \neq 0$ only if $m \in \bigcup_{i =1}^N (\lambda_i + (\sigma^\vee \cap M))$. It follows that $\v_R(q)$ is the maximum over a finite number of expressions of the form $\bigoplus_{i =1}^\ell \lambda_i$, and only a finite number of such expressions are possible. 
\end{proof}

\begin{example}\label{tensorbasis}
If $(E, \v)$ and $(F, \w)$ are equipped with linear adapted bases $\B_1 \subset E$, $\B_2 \subset F$, one checks that $\B_1 \times \B_2 \cong \{b_1 \otimes b_2 \mid b_i \in \B_2\} \subset E \otimes_\k F$ gives a linear adapted basis as follows.  For any $\rho \in \Sigma \cap N$, $G^\rho_s(\v) \cap \B_1 \subset G^\rho_s(\v)$ and $G^\rho_t(\w) \cap \B_2 \subset G^\rho_t(\w)$ are bases, so $G^\rho_{s + t}(\v \star \w) \cap (\B_1 \times \B_2)$ is a basis of $G^\rho_{s + t}(\v \star \w)$.  This shows that $\B_1 \times \B_2$ is adapted to $\v \star \w$.  Moreover, $(\v\star \w)(b_1 \otimes b_2) = \v(b_1) + \w(b_2)$ for any $b_1 \times b_2 \in \B_1 \times \B_2$, so $\B_1 \times \B_2$ is linear. 
\end{example}

\begin{example}\label{schurbasis}
Example \ref{tensorbasis} allows us to find adapted bases of $S_\lambda(E, \v)$ for any Schur functor $S_\lambda$ and a linear adapted basis $\B$ for $(E, \v)$. For any partition $\lambda$ we obtain a basis $\B_\lambda$ of $S_\lambda(E)$ by applying the symmetrizers $s_\tau$, corresponding to semi-standard fillings $\tau$ of $\lambda$, to $\B$.  We let $b_\tau \in \B_\lambda \subset S_\lambda(E)$ be the tensor obtained by applying $s_\tau$ to the basis $\B$. It is straightforward to check that if $\v$ is adapted to $\B \subset E$, then every simple tensor appearing in $b_\tau$ has the same value. The simple tensors with entries in $\B$ are an adapted basis of $E^{\otimes |\lambda|}$, so we conclude that $s_\lambda(\v)(b_\tau)$ is linear if $\v(b)$ is linear for all $b \in \B$.   For any $\rho \in \sigma \cap N$, we have $s_\lambda(\v)_\rho(\sum C_\tau b_\tau) \geq \bigoplus s_\lambda(\v)_\rho(b_\tau)$.  Let $\sum_{\min} C_\tau b_\tau$ denote the sum of terms in $\sum C_\tau b_\tau$
which are minimum with respect to $s_\lambda(\v)_\rho$. If $\sum_{\min} C_\tau b_\tau$ is non-zero, then some simple tensor in the elements of $\B$ in this sum achieves the top value, so this is an equality. If  $\sum_{\min} C_\tau b_\tau = 0$, then $C_\tau = 0$ for all $\tau$, as the $\B_\lambda = \{b_\tau\}$ is a basis.  It follows that $\B_\lambda$ is a linear adapted basis of $S_\lambda(E, \v)$.   
\end{example}

\begin{example}\label{dual}
For a $T_N$-vector bundle $\E$, the dual $\E^*$ corresponds to the locally free sheaf with $\Gamma(\E^*, Y(\sigma)) = \bigoplus_{i =1}^r S_\sigma \, b_i^*$, where $\deg(b_i^*) = -\deg(b_i) = -\lambda_i$.  We have $\mathcal{L}(\E^*) = (E^*, \v^*)$, where locally $(E^*, \v^*\!\!\mid_\sigma) = \bigoplus_{i = 1}^r (\k, -\lambda_i)$. 
\end{example}

For a module $R \in \Mod_{S_\sigma}^M$ there are two $\k$-vector spaces associated to a point $\rho \in \sigma \cap N$.  First, the fiber over a point in the orbit corresponding to the face $\tau$ containing $\rho$ in its relative interior by taking the quotient $R/\m_\tau R$, where $\m_\tau = \langle\{ \chi_u -1, \chi_v \mid u \in \tau^{\perp}, \langle \rho, v \rangle < 0 \ \forall \rho \in \tau\}\rangle$.   Alternatively, we have the associated graded space $\gr_\rho(E_R) = \bigoplus_{r \in \Q} G^\rho_r(R)/G^\rho_{> r}(R)$.  These two spaces are naturally isomorphic if $R$ is flat.  In this way we connect the geometric notion of fiber with the algebraic notion of associated graded. 

\begin{proposition}\label{fiber}
The map $\hat{\phi}_\rho: R \to \gr_\rho(E_R)$ sending $R_m$ to $G^\rho_{\langle \rho, m \rangle}(R)/G^\rho_{> \langle \rho, m \rangle}(R)$ is surjective and factors through the surjection $R \to R/\m_\tau R$, giving a map $\phi_\rho: R/\m_\tau R \to gr_\rho(E_R)$.  If $R$ is flat, then $\phi_\rho$ is also injective. If $R$ is an algebra, $\phi_\rho$ is a map of algebras.  
\end{proposition}


Finally, with the following proposition we recover the fact that a vector bundle $\E$ is captured by its Klyachko spaces. For each ray $\varrho_i \in \Sigma(1)$ let $\rho_i \in \varrho_i$ denote the first integral point.   

\begin{proposition}\label{arrangementspaces}
Let $\F \in \Sh_{Y(\Sigma)}^{M, ev}$ be flat and let $\psi$ be a piecewise linear function on $\Sigma$, then:

\begin{equation}
F_\psi(\v) = \bigcap_{\varrho_i \in \Sigma(1)} G^{\rho_i}_{\psi(\rho_i)}(\v).
\end{equation}
\end{proposition}

\subsection{Classification of flat sheaves of graded algebras}
As an important application of the machinery set up in Section \ref{category}, we obtain a theorem classifying $T_N$-equivariant degenerations of $\k$-algebras using quasivaluations into $\hat{\O}_\Sigma$ (Theorem \ref{generalclassification}). 
In the next section (Section \ref{Khovanskii}) we prove a more refined classification theorem (Theorem \ref{thm-main-family}). It can be regarded as Theorem \ref{generalclassification} in the presence of an additional finiteness hypothesis.

Fix a $\k$-algebra $A$, and a direct sum decomposition $A = \bigoplus_{i \in I} A_i$ (not necessarily a grading) into finite dimensional $\k$-vector spaces. 

\begin{definition}[Homogeneous flat degeneration]
A \emph{homogeneous $Y(\Sigma)$-degeneration} of $A$ is defined to be a flat sheaf of algebras $\mathcal{A} \in \Alg_{S_\Sigma}^M$ which can be written $\mathcal{A} = \bigoplus_{i \in I} \mathcal{A}_i$ where each $\mathcal{A}_i$ is coherent, such that $E_{\mathcal{A}} = \bigoplus_{i \in I} E_{\mathcal{A}_i} = \bigoplus_{i \in I} A_i = A$.  A quasivaluation $\v: A \to \hat{\O}_\Sigma$ is said to be \emph{homogeneous} if $(A, \v) \cong \bigoplus_{i \in I} (A_i, \v\!\!\mid_{A_i})$ in $\Vect_\Sigma$. \end{definition}

\begin{theorem}\label{generalclassification}
Any homogeneous degeneration is eversive, $\mathcal{L}(\A) = (A, \v)$, where $\v: A \to \hat{\O}_\Sigma$ is homogeneous, and each $\v\!\!\mid_\sigma$ has a linear adapted basis. In particular, $\v$ takes values in $\O_\Sigma$.  Moreover, all homogeneous quasivaluations such that $(A, \v\!\mid_\sigma)$ has an adapted basis for all $\sigma \in \Sigma$ arise as $\L(\A)$ for a homogeneous degeneration $\A$. 
\end{theorem}

\begin{proof}
A degeneration $\mathcal{A} = \bigoplus_{i \in I} \mathcal{A}_i$ is flat if and only if each $\mathcal{A}_i$ is flat, and therefore locally projective, and therefore locally free. It follows that $\mathcal{A}$ is eversive and $(A, \v\!\!\mid_\sigma)$ has a linear adapted basis. Such a homogeneous quasivaluation $\v: A \to \O_\Sigma$ likewise corresponds under Theorem \ref{eversiveglobalequivalence} to a unique homogeneous degeneration.
\end{proof}

\section{Khovanskii bases}\label{Khovanskii}
In this section we define the notion of a Khovanskii basis for a valuation $\v: A \to \hat{\O}_\Sigma$, and we prove two of our main theorems, namely, Theorems \ref{thm-main-family} and \ref{thm-main-bundle}.  

We begin by recalling elements of the theory of Khovanskii bases for quasivaluations with values in $\overline{\Q}$. Let $\mathcal{B} = \{b_1, \ldots, b_d\} \subset A$ be a generating set. We let $\pi: \k[\bx] \to A$ be the surjection with $\pi(x_i) = b_i \in \mathcal{B}$ and $I = \ker(\pi)$. Any $w \in \Q^\mathcal{B}$ gives a valuation $\hat{\v}_w: \k[\bx] \to \overline{\Q}$ where $\hat{\v}_w(\bx^\alpha) = \langle w, \alpha \rangle$. We let $\v_w = \pi_*(\hat{\v}_w)$ and refer to $\v_w$ as the \emph{weight quasivaluation} on $A$ associated to $w \in \Q^\mathcal{B}$ (see  \cite[Definition 3.1]{Kaveh-Manon-NOK}). By construction, the quasivaluation $\v_w$ takes values in $\overline{\Z}$ when $w \in \Z^\mathcal{B}$.  We let $\gr_w(A)$ denote the associated graded algebra of the quasivaluation $\v_w$ (see Equation \eqref{equ-gr-E}). By \cite[Lemma 3.4]{Kaveh-Manon-NOK} we have $\gr_w(A) \cong \k[\bx]/\In_w(I)$ where $\In_w(I) \subset \k[\bx]$ is the initial ideal of $I$ corresponding to $w$. Moreover, $\pi(\bx) = \mathcal{B} \subset A$ is a Khovanskii basis for $\v_w$, that is, the image of $\mathcal{B}$ is a generating set for the associated graded algebra $\gr_w(A)$ (Definition \ref{khovanskiibasisdefinition}). 

Let $\GF(I) \subset \Q^\mathcal{B}$ denote the Gr\"obner fan of $I$ (see e.g. \cite[Proposition 2.4]{GBCP}). Recall that if $w, w' \in \Q^\mathcal{B}$ lie in the relative interior of the same face of $\GF(I)$ then $\In_w(I) = \In_{w'}(I)$. For a maximal cone $\tau \in \GF(I)$, we let $\B_\tau \subset A$ be the associated \emph{standard monomial basis} (see \cite[Proposition 1.1]{GBCP}).  We will need the following proposition from \cite{Kaveh-Manon-NOK}.

\begin{proposition}\label{GrobnerAdapted}\cite[Proposition 4.9]{Kaveh-Manon-NOK}
Let $\mathcal{B}=\{b_1, \ldots, b_d\} \subset A$ be a Khovanskii basis for $\v: A \to \overline{\Q}$ and put $w = (\v(b_1), \ldots, \v(b_d)) \in \GF(I)$. Then $\v = \v_w$. Moreover, if $\B_\tau \subset A$ is a standard monomial basis with $w \in \tau$, then $\B_\tau$ is adapted to $\v$ and $\v(b^\alpha) = \langle w, \alpha \rangle$ for any $b^\alpha \in \B_\tau$. 
\end{proposition}

It can happen that $\v_w = \v_{w'}$ for distinct $w, w' \in \Q^\mathcal{B}$.  To account for this,  the map $\iota: \Q^\mathcal{B} \to \Q^\mathcal{B}$ defined by $\iota(w) = (\v_w(b_1), \ldots, \v_w(b_d))$ was introduced in \cite[Section 3.2]{Kaveh-Manon-NOK}.  By \cite[Proposition 3.7]{Kaveh-Manon-NOK}, if $I$ is homogeneous with respect to a positive grading, then $\iota^2 = \iota$, and $\iota(w) = \iota(w')$ if and only if $\v_w = \v_{w'}$. It is not hard to show that $\iota$ is an integral piecewise linear map on each face of $\GF(I)$, and that the image of $\iota$ is the support of a subfan $\mathcal{K}(I) \subset \GF(I)$ which contains the tropical variety $\Trop(I)$.  Points $w \in \mathcal{K}(I)$ correspond precisely to the distinct quasivaluations with Khovanskii basis $\mathcal{B} \subset A$. 


\subsection{Khovanskii bases in $\Alg_\Sigma$}

In this section we extend the above definitions to quasivaluations into $\O_\Sigma$.  
It is more convenient to begin by considering valuations into $\hat{\O}_\Sigma$.  Recall that $\hat{\O}_\Sigma$ is the semialgebra of functions $\psi: |\Sigma| \to \overline{\Q}$ with the property that $\psi(\ell\rho) = \ell \psi(\rho)$ for all $\ell \in \Q_{\geq 0}$, and that $\Alg_\Sigma$ is the category whose objects are pairs $(A, \v)$, where $A$ is an algebra and $\v: A \to \hat{\O}_\Sigma$ is a quasivaluation (see Section \ref{category}).  

\begin{definition}\label{ONKhovanskiibasis}
A generating set $\mathcal{B} \subset A$ is said to be a \emph{Khovanskii basis} of $(A, \v) \in \Alg_\Sigma$ if it is a Khovanskii basis  of $\v_\rho: A \to \overline{\Q}$, for every $\rho \in |\Sigma|$, in the sense of Definition \ref{khovanskiibasisdefinition}. 
\end{definition}

For $(A, \v) \in \Alg_\Sigma$ and a subset $\mathcal{B} \subset A$ we define $\Phi_\mathcal{B}: |\Sigma| \to \overline{\Q}^\mathcal{B}$ to be the function $$\Phi_\mathcal{B}(\rho) = (\v_\rho(b_1), \ldots, \v_\rho(b_d)).$$  We can think of $\Phi_\mathcal{B}$ as a tuple $(\v(b_1), \ldots, \v(b_d)) \in \hat{\O}_\Sigma^d$. When it is clear from context we may drop the subscript $\mathcal{B}$.  

Recall that the tropical variety $\Trop_{\O_\Sigma}(I)$ over the semialgebra $\O_\Sigma$ can be viewed as the collection of tuples $(\phi_1, \ldots, \phi_d) \in \O_\Sigma^d$ which satisfy the \emph{tropical relations} coming from the ideal $I$ (see Section \ref{background}).  We note that a tuple $(\phi_1, \ldots, \phi_d) \in \O_\Sigma^d$ will satisfy the tropical relations coming from $I$ if and only if this is so for $(\phi_1(\rho), \ldots, \phi_d(\rho))$ for all $\rho \in |\Sigma|$. Hence, we can also view $\Trop_{\O_\Sigma}(I)$ as the set of integral piecewise linear maps $\Phi: |\Sigma| \to \Trop(I) \subset \Q^d$. 

We now define the analogues of the fans $\mathcal{K}(I)$ and $\GF(I)$ over the semialgebra $\mathcal{O}_\Sigma$.  
\begin{definition}\label{Onsets}
Let $\mathcal{K}_{\O_\Sigma}(I)$ and $\GF_{\O_\Sigma}(I)$ be the sets of integral piecewise linear maps on $|\Sigma|$  with values in the supports of fans $\mathcal{K}(I)$ and $\GF(I)$ respectively.  The sets $\mathcal{K}_{\hat{\O}_\Sigma}(I)$, and $\GF_{\hat{\O}_\Sigma}(I)$ are defined similarly. 
\end{definition}

 For any $\Phi = (\phi_1, \ldots, \phi_d) \in \hat{\O}^d_\Sigma$ there is a canonical valuation $\hat{\v}_{\Phi}: \k[\bx] \to \hat{\O}_\Sigma$ defined by:
$$\hat{\v}_\Phi(\sum_\alpha c_\alpha \bx^\alpha) = \min\{ \sum_i \alpha_i \phi_i \mid \alpha = (\alpha_1, \ldots, \alpha_d),~ c_\alpha \neq 0\}.$$ 
 In other words, we require that $\{\bx ^\alpha \mid \alpha \in \Z_{\geq 0}^d \}$ is an adapted basis for $\hat{\v}_\Phi$ and we set $\hat{\v}_\Phi(\bx^\alpha) = \sum_i \alpha_i \phi_i$, for all $\alpha=(\alpha_1, \ldots, \alpha_d) \in \Z_{\geq 0}^d$. The pushforward $\pi_*\hat{\v}_\Phi: A \to \hat{\O}_\Sigma$ is then a quasivaluation on $A$.  We use this construction to give an alternative characterization of Khovanskii bases. 

\begin{proposition}\label{Khovanskiipushforward}
Let $(A, \v) \in \Alg_\Sigma$, $\{b_1, \ldots, b_d\} = \mathcal{B} \subset A$, and suppose that $\Phi = (\v(b_1), \ldots, \v(b_d))$ lies in $\GF_{\hat{\O}_\Sigma}(I)$. Then $\mathcal{B}$ is a Khovanskii basis for $\v$ if and only if $\v = \pi_*\hat{\v}_{\Phi}$.  In this case $\Phi \in \mathcal{K}_{\hat{\O}_\Sigma}(I)$. 
\end{proposition}

\begin{proof}
By definition of pushforward, $(\pi_*\hat{\v}_{\Phi})_\rho$ for $\rho \in |\Sigma|$ is the weight quasivaluation $\v_{\Phi(\rho)}$. If $\mathcal{B} \subset A$ is a Khovanskii basis of $\v$, it is a Khovanskii basis of $\v_\rho$ in the sense of Definition \ref{khovanskiibasisdefinition}. By Proposition \ref{GrobnerAdapted},  $\v_\rho = \v_w$, where $w_i = \v(b_i)(\rho)$.  This means $\v_\rho = \v_{\Phi(\rho)}$ and $\Phi(\rho) \in \mathcal{K}(I)$ for all $\rho \in \Sigma \cap N$, so $\v = \pi_*\hat{\v}_{\Phi}$.  Conversely, if $\v = \pi_*\hat{\v}_{\Phi}$, then $\v_\rho = \v_{\Phi(\rho)}$. So each $\v_\rho$ has $\mathcal{B}$ as a Khovanskii basis. 
\end{proof}

Recall that if $I$ is homogeneous with respect to a positive grading on $\k[\bx]$ then the support of $\GF(I)$ is  all of $\Q^\mathcal{B}$, so the condition $\Phi \in \GF_{\hat{\O}_\Sigma}(I)$ is satisfied for any map $\Phi: |\Sigma| \to \overline{\Q}^\mathcal{B}$ whose components are in $\hat{\O}_\Sigma$.  

Now we deal with quasivaluations with values in $\O_\Sigma$.

\begin{proposition}\label{values}
Let $A$ be a positively graded $\k$-algebra with a Khovanskii basis $\mathcal{B}=\{b_1, \ldots, b_d\}$. For any quasivaluation $\v: A \to \O_\Sigma$ let $\Phi = (\v(b_1), \ldots, \v(b_d))$. We have $\Phi \in \mathcal{K}_{\O_\Sigma}(I)$, and moreover, if $\v$ is a valuation, $\Phi \in \Trop_{\O_\Sigma}(I)$. The set $\mathcal{K}_{\O_\Sigma}(I)$ is in bijection with the set of quasivaluations with Khovanskii basis $\mathcal{B}$. 
\end{proposition}

\begin{proof}
For any quasivaluation $\v: A \to \O_\Sigma$ with $\Phi$ as above and $\rho \in |\Sigma| \cap N$ we have $\Phi(\rho) \in \mathcal{K}(I)$.  Now let $\v, \v'$ be quasivaluations with Khovanskii basis $\mathcal{B}$ with tuples $\Phi$ and $\Phi'$, and let $\rho \in \Sigma \cap N$. If $\Phi = \Phi'$, then $\v_\rho(b) = \v'_\rho(b)$ for all $b \in \mathcal{B}$.  By Proposition \ref{GrobnerAdapted} we must have $\v_\rho = \v'_{\rho}$ for all $\rho \in |\Sigma|\cap N$, so $\v = \v'$. Moreover, the fibers of $\mathcal{A}$ coincide with the associated graded algebras of the $\v_\rho$ by Proposition \ref{fiber}.  The fibers of the flat family are domains if and only if the $\v_\rho$ are valuations, so we may use Proposition \ref{trop}. 
\end{proof}

By abuse of notation we let $\iota: \GF_{\O_\Sigma}(I) \to \GF_{\O_\Sigma}(I)$ denote the map $\Phi \to \iota \circ \Phi$. It is straightforward to check that $\iota^2 = \iota$ and that the image of $\iota$ is $\mathcal{K}_{\O_\Sigma}(I)$.

\subsection{Classification of toric flat families: proof of Theorem \ref{thm-main-family}}
In this section we complete the proof of Theorem \ref{thm-main-family}. We start with some preparation.
Let $\A$ be a flat sheaf of algebras over a toric variety $Y(\Sigma)$ with general fiber $A$. As usual $\A$ gives us a toric flat family $\pi: \mathcal{X} \to Y(\Sigma)$ constructed as the relative spectrum $\mathcal{X} = \Spec_{Y(\Sigma)}(\A)$. 

First we show that finite Khovanskii bases for $A$ appear when $\mathcal{X}$ is of finite type. We start with the case where the base is an affine toric variety. Recall that $S_\sigma = \k[\sigma^\vee \cap M]$ denotes the affine semigroup algebra associated to a cone $\sigma$.  

\begin{proposition}\label{finitetype}
Let $R$ be a finitely generated $S_\sigma$-algebra giving a sheaf of algebras $\mathcal{A}$ over $Y(\sigma)$. Let $A$ be the fiber of $\mathcal{A}$ over the distinguished point $x_0$ in $Y(\sigma)$ and let $\mathcal{L}(R) = (A, \v) \in \Alg_\sigma$ (see Definition \ref{def-functor-L}). Then $(A, \v)$ has a finite Khovanskii basis $\mathcal{B}$. Moreover, if $R$ is flat and positively graded then $R$ is a finitely generated $S_\sigma$-algebra if and only if the special fiber $R/\m_0 R$ is a finitely generated $\k$-algebra. 

In this latter case, $\mathcal{B} \subset A$ can be chosen so that $\Phi_\mathcal{B}$ is linear on $\sigma$, $\Phi_\mathcal{B}(\sigma)$ is contained in a face of $\mathcal{K}(I)$, and $\v$ has a linear adapted basis given by the standard monomials $\B_\tau$ for any maximal face $\tau \in \GF(I)$ which contains the image $\Phi_\mathcal{B}(\sigma)$.  
\end{proposition}

\begin{proof}
There is a presentation $\hat{\pi}: S_\sigma[\bx] \to R$, let $\hat{b}_i = \hat{\pi}(x_i) \in F_{\lambda_i}(R)$, and $b_i = \phi_R(\hat{b}_i) \in \phi_R(F_{\lambda_i}(R)) \subset A$.  By assumption, for each face $\eta \subset \sigma$, the images of the $\hat{b}_i$ in $R/\m_\eta R$ generate as a $\k$-algebra. Each map $\phi_\rho: R/\m_\eta R \to \gr_\rho(A)$ is a surjection of algebras (Proposition \ref{fiber}), hence the images of the $\hat{b}_i$ generate $\gr_\rho(A)$ as well. By definition, the image $\phi_\rho(\hat{b}_i)$ is in $F_{\lambda_i}(R) \subset G^\rho_{\langle \rho, \lambda_i \rangle}(R)/G^\rho_{> \langle \rho, \lambda_i \rangle}(R)$. As a consequence, for each $b_i \in A$, $\phi_\rho(\hat{b}_i)$ coincides with the equivalence class $\tilde{b}_i \in \gr_\rho(A)$.

Let $R$ be a flat, positively graded $S_\sigma$-algebra, $R = \bigoplus_{n \geq 0} R_n$, where each $R_n$ is a coherent $S_\sigma$ module. Then each $R_n$ is free as an $S_\sigma$-module, so it follows that $\mathcal{L}(R)$ has a linear adapted basis $\mathbb{B}$, and we may write $R = \bigoplus_{b \in \B} S_\sigma \otimes \k b$.  Let $\mathcal{B} \subset \B$ be such that the equivalence classes $\tilde{\mathcal{B}} \subset R/\m_0 R$ generate as a $\k$-algebra, then a graded version of Nakayama's lemma implies that $\mathcal{B}$ generates $R$.  

For any $\rho \in \sigma^\circ$, the relative interior of $\sigma$, the set $\tilde{\mathcal{B}} \subset \gr_\rho(A)$ is a generating set, so \cite[Algorithm 2.18]{Kaveh-Manon-NOK} implies that $\mathcal{B} \subset A$ is a Khovanskii basis of $\v_\rho$.  We can also apply this argument to the restriction of $\mathcal{L}(R)$ to a face $\eta \subset \sigma$, so it follows that $\mathcal{B} \subset A$ is a Khovanskii basis of $\mathcal{L}(R)$.  From above, the presentation of $R/\m_{\eta} R$ by the image of the Khovanskii basis can be obtained by composing the inclusion $\k[\bx] \subset S_\sigma[\bx]$ with $\hat{\pi}$ and the quotient map.  The kernel of this presentation is $\In_\rho(I)$.  It follows that if $\rho, \rho' \in \eta^\circ$ then $\In_\rho(I) = \In_{\rho'}(I)$, so $\Phi(\rho)$ and $\Phi(\rho)$ must lie in the same face of the Gr\"obner fan $\GF(I)$. This implies that $\rho$ and $\rho'$ are in the same face of the subfan $\mathcal{K}(I) \subset \GF(I)$.  By construction, the set $\tilde{\mathcal{B}} \subset A$ is $\k$-linearly independent, and each $\v(b)$ defines a linear function on $\sigma$.  Let $\tau \in \GF(I)$ be any maximal cone which contains $\Phi(\sigma)$, then for any $\rho \in \sigma$, the standard monomials for $\tau$ are an adapted basis of $\v_{\Phi(\rho)}$.  It follows that the corresponding monomials in $\mathcal{B} \subset R$ form an $S_\sigma$ basis of $R$. 
\end{proof}


In the next Proposition we give criteria for $\v:A \to \O_\Sigma$ to come from a flat sheaf of algebras on $Y(\Sigma)$.

\begin{proposition}\label{localbasis}
Suppose for each $\sigma \in \Sigma$ the restriction $(A, \v\!\!\mid_\sigma)$ has a finite Khovanskii basis $\mathcal{B}_\sigma$, and that the image $\Phi_\mathcal{B_\sigma}(\sigma \cap N) \subset \Z^\mathcal{B_\sigma}$ lies in a face of the Gr\"obner fan $\GF(I_\mathcal{B_\sigma})$. Then $(A, \v\!\!\mid_\sigma) \in \Alg_\sigma$ has an adapted basis.  Moreover, if $\v(b_i)\!\!\mid_\sigma \in M_\sigma \subset \O_\sigma$ for each $b_i \in \mathcal{B}_\sigma$ then $(A, \v) = \mathcal{L}(\mathcal{A})$ for a locally free family of algebras $\A$ on $Y(\Sigma)$. 
\end{proposition}
\begin{proof}
By Proposition \ref{GrobnerAdapted}, $\v_\rho = \v_{\Phi_{\mathcal{B}_\sigma}(\rho)}$.  It follows that $\B_\tau$ is an adapted basis of $\v_\rho$ for any $\rho \in \sigma \cap N$, where $\Phi_{\mathcal{B}_\sigma}(\sigma) \subset \tau \subset \GF(I_{\mathcal{B}_\sigma})$.  If $\v(b_i)\!\!\mid_\sigma \in M_\sigma$ for $b_i \in \mathcal{B}_\sigma$, then $\v(b^\alpha)\!\!\mid_\sigma \in M_\sigma$ for any $b^\alpha \in \B_\tau$ by Proposition \ref{GrobnerAdapted}, so $\B_\tau \subset A$ is a linear adapted basis for $(A, \v\!\!\mid_\sigma) \in \Vect_\sigma$. \\
\end{proof}


\begin{proof}[Proof of Theorem \ref{thm-main-family}]
Suppose $\pi: \mathcal{X} \to Y(\Sigma)$ is flat and of finite-type, where $\mathcal{X} = \Spec_{Y(\Sigma)}(\A)$, and $\A$ is positively graded. Let $R_\sigma$ be the coordinate ring of the induced family over the affine chart $Y(\sigma) \subset Y(\Sigma)$. Proposition \ref{finitetype} implies that $(A, \v\!\!\mid_\sigma)$ has a finite Khovanskii basis $\mathcal{B}_\sigma \subset A$ for each $\sigma \in \Sigma$, $\Phi_{\mathcal{B}_\sigma}$ is linear, and $\Phi_{\mathcal{B}_\sigma}(\sigma)$ lies in a face of $\mathcal{K}(I_{\mathcal{B}_\sigma})$. Taking the union $\mathcal{B} = \bigcup_{\sigma \in \Sigma} \mathcal{B}_\sigma$ gives a finite Khovanskii basis of $\mathcal{L}(\A)$.  Let $I$ be the ideal of relations among the elements of $\mathcal{B}$.  For any $\rho \in \sigma \cap N$, the point $\Phi(\rho) \in \mathcal{K}(I)$ is computed  by finding a standard monomial expression for each $b \in \mathcal{B}$ with respect to monomials in $\mathcal{B}_\sigma$: $b = \sum C_{\alpha} b^\alpha$, then taking the minimum $\min\{ \langle \Phi_{\mathcal{B}_\sigma}(\rho), \alpha \rangle \mid C_\alpha \neq 0 \}$.  Let $\prec_\sigma$ be any monomial order corresponding to a maximal face of $\Sigma(I_{\mathcal{B}_\sigma})$ which contains $\Phi_{\mathcal{B}_\sigma}(\sigma)$, and let $\prec$ be any extension of $\prec_\sigma$ to $I \subset \k[\bx]$ which makes elements of $\mathcal{B}_\sigma$ strictly larger than elements of $\mathcal{B}\setminus \mathcal{B}_\sigma$. By considering the $\mathcal{B}_\sigma$ standard monomial expression for $b \in \mathcal{B}$, we see that $b \in \mathcal{B}\setminus \mathcal{B}_\sigma$ and the monomials in $\In_{\prec_\sigma}(I_{\mathcal{B}_\sigma})$ generate the initial ideal $\In_\prec(I)$.   It follows that the standard monomials of $\prec_\sigma$ are the standard monomials of $\prec$, and that $\Phi(\sigma)$ lies in the closure of the  corresponding maximal face of $\GF(I)$. 

 By Proposition \ref{localbasis}, if $(A, \v) \in \Alg_\Sigma$, each $(A, \v\!\!\mid_\sigma)$ has a finite Khovanskii basis, $\Phi_{\mathcal{B}_\sigma}$ is linear, and $\Phi_{\mathcal{B}_\sigma}(\sigma)$ is contained in a face of $\mathcal{K}(I_{\mathcal{B}_\sigma})$, then $(A, \v) = \mathcal{L}(\A)$, where $\A$ is a locally free sheaf of algebras of finite type.  Finally, Theorem \ref{generalclassification} implies that $\A$ and $\mathcal{L}(\A)$ can be recovered from one another.  If $\Sigma$ is finite, then $\mathcal{B} = \bigcup_{\sigma \in \Sigma} \mathcal{B}_\sigma$ is a finite Khovanskii basis of $\mathcal{L}(\A)$.  Moreover, $\v$ is a valuation if and only if $\mathcal{X}$ has reduced, irreducible fibers; in this case $\Phi \in \Trop_{\O_\Sigma}(I)$. 
\end{proof}

\subsection{The diagram of a toric flat family}
Proposition \ref{values} implies that any quasivaluation $\v: A \to \O_\Sigma$ with Khovanskii basis $\mathcal{B}$ is determined by its values on $\mathcal{B}$. We represented this set of values by an element $\Phi \in \mathcal{K}_{\O_\Sigma}(I)$. For a cone $\sigma \in \Sigma$, Proposition \ref{finitetype} then shows that the points $\Phi(\rho_i) \in \Q^d$ for the ray generators $\rho_i$ of $\sigma$, must all lie in a common maximal face $\tau \in \GF(I)$.  Moreover, it is immediate that the restriction of $\v(b^\alpha)$ to $\sigma$ is a linear function for any member of the standard monomial basis associated to $\tau$.  These observations suggest the main idea of this section: \emph{the values of $\Phi$ on ray generators of $\Sigma$ are enough information to determine the flat family.} To be compatible with the fan $\Sigma$, these values must satisfy ``compatibility conditions'' with respect to certain subcomplexes of $\mathcal{K}(I)$. This is a far extension of the Klyachko data of a toric vector bundle.

\begin{definition}\label{def-apartment}
An \emph{apartment} $A_\tau \subset \mathcal{K}(I)$ is a subcomplex of the form $\mathcal{K}(I) \cap \tau$, where $\tau$ is a maximal face of the Gr\"obner fan $\GF(I)$. 
\end{definition}

Let $\B_\tau\subset A$ be the standard monomial basis associated to a facet $\tau \in \GF(I)$, and let $\pi_\tau: \mathcal{K}(I) \to \Q^{\mathcal{B}\cap \B_\tau}$ be the projection along the components associated to $\mathcal{B} \cap \B_\tau$.   For any monomial $\bx^\alpha \in \B_\tau$, $x_i \mid \bx^\alpha$ only if $x_i$ corresponds to an element of $\mathcal{B} \cap \B_\tau$.  It follows that $w \in A_\tau$ is determined by the projection $\pi_\tau(w)$.  Let $\phi_\tau: \Q^{\mathcal{B} \cap \B_\tau} \to \Q^\mathcal{B}$ be the map which takes the value $\min\{\langle w, \alpha\rangle \mid b = \sum C_\alpha b^\alpha, C_\alpha \neq 0\}$ on the component corresponding to $b$, where $b = \sum C_\alpha b^\alpha$ is the expression of $b$ in the basis $\B_\tau$. 

\begin{definition}[$\Sigma$-adapted map]
We say $\Phi: |\Sigma| \to \mathcal{K}(I)$ is \emph{adapted} to $\Sigma$ if for every $\sigma \in \Sigma$ there is an apartment $A_\tau$ such that $\Phi(\sigma) \subset A_\tau$ and $\pi_\tau\circ \Phi: \sigma \to \Q^{\mathcal{B} \cap \B_\tau}$ is a linear map.
\end{definition}

\begin{definition}[Diagram of a piecewise linear map]
Let $\Phi: |\Sigma| \to \mathcal{K}(I)$ be as above, then the diagram $D(\Phi)$ is defined to be the matrix whose $(j, i)$-th entry is $\v_{\Phi(\rho_j)}(b_i)$, where $\rho_j$ is the ray generator of $\varrho_i \in \Sigma(1)$ and $b_i \in \mathcal{B}$. 
\end{definition}

We can now give a proof Corollary \ref{cor-main-family} from the introduction.
\begin{proposition}[Corollary \ref{cor-main-family}] \label{prop-diagram}
If $\Phi$, $\Phi': |\Sigma| \to \mathcal{K}(I)$ and $D(\Phi) = D(\Phi')$ then $\Phi = \Phi'$.   Moreover, if $\Phi: N_\Q \to \mathcal{K}(I)$ is a $\Sigma$-adapted map for $\Sigma \subset N_\Q$, then $\Phi$ defines a toric flat family with general fiber $\Spec(A)$ over $Y(\Sigma)$.  
\end{proposition}
\begin{proof}
 By assumption, for any $\sigma \in \Sigma$, $\Phi(\sigma)$ is in some apartment $A_\tau \subset \mathcal{K}(I)$.   For any $\rho, \rho' \in \sigma$ and $b \in \mathcal{B}$ we have $\v_{\Phi(\rho + \rho')}(b) = \min\{ \langle \Phi(\rho + \rho'), \alpha \rangle \mid C_\alpha \neq 0\}$, where $b = \sum C_\alpha b^\alpha$ is the standard monomial expression for $b$ with respect to $\tau$.  We have $\v_{\Phi(\rho + \rho')}(b^\alpha) = \v_{\Phi(\rho) + \Phi(\rho')}(b^\alpha)$ on standard monomials, so $\v_{\Phi(\rho + \rho')}(b) = \v_{\iota(\Phi(\rho) + \Phi(\rho'))}(b)$.   It follows that $\Phi(\sigma)$ is determined by the image of the ray generators $\rho_i \in \varrho_i \in \sigma(1)$.   Now, if $\Phi: N_\Q \to \mathcal{K}(I)$ is as above, then for any $\sigma \in \Sigma$, the quasivaluations $\v_{\Phi(\rho)}$ have a common linear adapted basis given by the standard monomials associated to $\tau$, so we may use Theorem \ref{thm-main-family}. 
\end{proof}

\begin{remark}\label{rem-simplicialdiagrams}
Given a simplicial fan $\Sigma \subset N_\Q$ and a generating set $\mathcal{B} \subset A$, we can produce toric flat families with general fiber $\Spec(A)$ by choosing $w_i \in \mathcal{K}(I)$ in bijection with the rays of $\Sigma$ such that for each $\sigma \in \Sigma$, the $w_i$ corresponding to $\varrho_i \in \sigma(1)$ lie in a common apartment $A_\tau$.  Proposition \ref{prop-diagram} can then be used to define adapted piecewise linear maps $\Phi \!\! \mid_\sigma: \sigma \to \mathcal{K}(I)$, and glue them to form $\Phi: |\Sigma| \to \mathcal{K}(I)$.  To extend this construction to arbitrary fans, the $w_i$ must be chosen to satisfy any linear relations that hold among the ray generators $\rho_i$.
\end{remark}

\begin{remark}\label{rem-apartment}
The term ``apartment" is taken from the theory of buildings (see \cite{Kaveh-Manon-Building}).  If $L \subset \k[\bx]$ is a linear ideal, the maximal faces $\tau_\mathbb{B} \subset \Sigma(L)$ correspond to bases $\mathbb{B} \subset \mathcal{B}$.  The set $\tau_\mathbb{B} \cap \Trop(L)$ is then identified with the set of valuations on $\k[\bx]/L$ with adapted basis given by the monomials in $\B$.  The valuation $\v_\rho$ corresponding to a point $\rho \in \tau_\mathbb{B}$ is entirely determined by its values on $\mathbb{B}$, so $\tau_\mathbb{B} \cap \Trop(L) \cong \Q^m$, where $m = |\mathbb{B}|$.    Under the natural identification of the spherical building of $\GL_m(\k)$ with the valuations on $\k[\bx]/L$ framed by bases, $\phi_\mathbb{B}(\Q^m)$ is mapped onto the apartment of the spherical building corresponding to $\mathbb{B}$. 
\end{remark}

\begin{example}\label{ex-canonical}
Let $\Sigma$ be any fan which refines $\GF(I)$, then the map $\iota: \GF(I) \to \mathcal{K}(I)$
defines a canonical piecewise-linear map $\iota: |\Sigma| \to \mathcal{K}(I)$. Corollary \ref{cor-main-family} then implies that there is an associated toric flat family over $Y(\Sigma)$. The diagram of this family can be computed by applying $\iota$ to the integral generators of rays of $\Sigma$. 
\end{example}

\subsection{Toric vector bundles as tropical points: proof of Theorem \ref{thm-main-bundle}}
For any $T_N$-vector bundle $\mathcal{E}$ on $Y(\Sigma)$ there is a $\Z_{\geq 0}$-graded locally-free, $T_N$-sheaf $\mathcal{O}_\mathcal{E}= \Sym(\E^*)$.  If $(E, \v) = \mathcal{L}(\E)$, then by Theorem \ref{generalclassification}, this is the degeneration corresponding to the algebra $\Sym(E^*, \v^*) = (\Sym(E^*), s(\v^*))$.   In this case, $s(\v^*): \Sym(E^*) \to \O_\Sigma \subset \hat{\O}_\Sigma$ is a valuation.  

\begin{proof}[Proof of Theorem \ref{thm-main-bundle}]
 Given a toric vector bundle $\E$, we consider the prevalued vector space $(E^*, \v^*) = \mathcal{L}(\E^*)$ associated to the dual bundle (Example \ref{dual}).   Theorem \ref{eversiveglobalequivalence} implies that the restriction of $\v^*$ to each face $\sigma \in \Sigma$ has a linear adapted basis.  In turn, we obtain a linear adapted basis of monomials in each $(\Sym^n(E^*), s^n(\v^*)\!\!\mid_\sigma)$ (Example \ref{schurbasis}).  It follows that $\B_\sigma$ is a Khovanskii basis for each $s(\v^*)\!\!\mid_\sigma$ and $\mathcal{B} = \cup_{\sigma \in \Sigma} \B_\sigma$ is a Khovanskii basis for $s(\v^*)$.  By construction, $\E$ can be recovered from $(E^*, \v^*)$, so two distinct bundles yield distinct valuations on $\Sym(E^*)$.  Conversely, if $\v: \Sym(E^*) \to \O_\Sigma$ has Khovanskii basis $\mathcal{B}$ with linear adapted bases $\B_\sigma$, then $(E, \w)$ classifies a toric vector bundle, where $\w$ is the dual of the restriction of $\v$ to $E^*$.  Now, $s(\w^*)$ and $\v$ take the same values on a common Khovanskii basis $\mathcal{B}$, so they must coincide by Corollary \ref{values}. 
\end{proof}

Recall that a set of polynomials $\{f_1, \ldots, f_m\} \subset I$ is said to be a tropical basis of $I$ \cite[Definition 2.6.3]{MSt} if $\Trop(I) = \bigcap_{i =1}^m \Trop(\langle f_i \rangle)$.  

\begin{proposition}\label{tropistrop}


If $\{f_1, \ldots, f_m\} \subset I$ is a tropical basis for $\Trop(I)$ then it is also a tropical basis over $\O_N$.  
\end{proposition}

\begin{proof}
A function $\Phi \in \O_N^d$ is in the tropical variety over $\O_N$ as defined in Section \ref{background} if and only if equation \ref{bending} holds for all $f \in I$.  This happens if and only if for all $\rho \in N$, the point $\Phi(\rho)$ satisfies the \ref{bending} for all $f$ over $\overline{\Q}$, which in turn holds if and only if \ref{bending} is satisfied for each element of a tropical basis of $I$.  The latter equivalence implies that $\Trop_{\O_N}(I) = \bigcap_{i =1}^m \Trop_{\O_N}(\langle f_i \rangle )$ for $\{f_1, \ldots, f_m\} \subset I$ a tropical basis. 
\end{proof}

\begin{lemma}\label{iotalinear}
For any linear ideal $L$, $\mathcal{K}(L) = \Trop(L)$ and $\mathcal{K}_{\O_N}(L) = \Trop_{\O_N}(L)$. 
\end{lemma}

\begin{proof}
It is always the case that $\Trop(L) \subset \mathcal{K}(L)$.  Let $w \in \mathcal{K}(L)$ and $\sum C_i x_i \in L$ be a circuit.  It is known that the circuits of a linear ideal form a tropical basis. Suppose for some $j$, $w_j < w_k$ for all $k \neq j$, $C_k \neq 0$.  We obtain a contradiction: $\v_w(x_j) \geq \max\{ w_k \mid k \neq j, C_k \neq 0\} > w_j = \v_w(x_j)$.  As a consequence $w \in \Trop(L)$. 
\end{proof}

If $(A, \v) \in \Alg_\Sigma$ has a finite Khovanskii basis $\mathcal{B}$ and $\v$ is a valuation, the image of $\Phi: |\Sigma| \to \Trop(I)$ must lie in the locus $\Trop^p(I)$ of points $w \in \GF(I)$ for which $\In_w(I)$ is a prime ideal.  This condition depends on the initial ideal, so $\Trop^p(I)$ is a union of relatively open faces of $\mathcal{K}(I)$ (in fact, if $I$ is positively graded, the closure of an open face in $\Trop^p(I)$ is in $\Trop^p(I)$).   We say that an ideal $I$ is \emph{well-poised} (see \cite{Ilten-Manon}) if $\Trop^p(I) = \Trop(I)$.  

\begin{proposition}\label{bundletropicalpoint}
Let $\mathcal{B} \subset E^*$ and let $L \subset \k[\bx]$ be the corresponding linear ideal.  Valuations $\v: \Sym(E^*) \to \O_\Sigma$ with Khovanskii basis $\mathcal{B}$ are in bijection with points $\Phi \in \Trop_{\O_\Sigma}(L)$.  Moreover, for any $\Phi \in \Trop_{\O_\Sigma}(L)$ there is a refinement $\pi: \hat{\Sigma} \to \Sigma$ such that the corresponding valuation corresponds to a vector bundle $\E$ over $Y(\hat{\Sigma})$. 
\end{proposition}

\begin{proof}
Valuations $\v: \Sym(E^*) \to \O_\Sigma$ are classified by their values $\Phi \in \mathcal{K}_{\O_\Sigma}(L) = \Trop_{\O_\Sigma}(I)$.  Moreover, any point $\Phi \in \Trop_{\O_\Sigma}(L)$ determines a quasi-valuation $\v_\Phi$ with Khovanskii basis $\mathcal{B}$ which must be a valuation since $L$ is well-poised.  By construction, the restriction of $\v$ to $E^*$ is computed on $f \in E^*$ by taking $\max$ of expressions of the form $\bigoplus_{i =1}^\ell \phi_i$ where $\Phi = (\phi_1, \ldots, \phi_d)$. There are only a finite number of functions of this type, so we conclude that $\v$ takes a finite number of values on $E^*$.  By Corollary \ref{finiteequivalence}, there is a fan $\Sigma$ such that $\mathcal{B}$ contains a linear adapted basis for each restriction $\v\!\!\mid_\sigma$, $\sigma \in \Sigma$.   
 \end{proof}


\begin{remark}\label{rem-tensorsumdiagram}
The direct sum and tensor product operations on toric vector bundles are straightforward to compute with diagrams.  If $\E$ is defined by the ideal $L \subset \k[y_1, \ldots, y_r]$ and diagram $D_1$, and $\F$ is defined by $K \subset \k[w_1, \ldots, w_s]$ and daigram $D_2$, then $\E \oplus \F$ is defined by the ideal $\langle L, K \rangle \subset \k[y_1, \ldots, y_r, w_1, \ldots, w_s]$ and the diagram $[D_1, D_2]$.  The tensor product $\E \otimes \F$ is defined by the ideal in $\k[\ldots, y_i\otimes w_j, \ldots]$ generated by forms $\sum_{i =1}^r C_i[y_i\otimes w_j]$, $\sum_{j = 1}^s K_jy_i \otimes w_j$, where $\sum C_iy_i \in L$ and $\sum K_jw_j \in K$. The the column of the diagram $D_1\odot D_2$ for $\E \otimes \F$ corresponding to $y_j\otimes w_j$ is obtained by adding the $i$-th column of $D_1$ to the $j$-th column of $D_2$.  In particular, if $\F$ is the $T_N$-linearized line bundle corresponding to integers $r_1, \ldots, r_n$, then $D_1\odot D_2$ is computed by adding $r_i$ to every entry in the the $i$-th row of $D_1$. 
\end{remark}

\subsection{Connection with the work of Di Rocco, Jabbusch, and Smith}\label{sec-DJS}

In \cite{DJS}, Di Rocco, Jabbusch, and Smith define a matroid $\mathcal{M}(\E)$ which captures many properties of the toric vector bundle $\E$. We finish this section by relating $\mathcal{M}(\E)$ to the notion of Khovanskii basis. 


\begin{proposition}  \label{matroidproperties}
With notation as above, let $\mathcal{B} \subset E$ be a representation of the matroid $\mathcal{M}(\E)$ as constructed in \cite[Algorithm 3.2]{DJS}, then $\mathcal{B}$ is a Khovanskii basis of $\v$. 
\end{proposition}
\begin{proof}

In \cite[Section 3]{DJS} it is observed that the Klyachko spaces $G^{\rho_i}_r(\v)$ for $\Z_{\geq 0}\rho_i = \varrho_i \cap N, \varrho_i \in \sigma(1)$, generate a distributive lattice under intersection, so there is a ``compatible basis" $\B_\sigma \subset \mathcal{B}$. This means that $\B_\sigma \cap G^{\rho_i}_r(\v)$ is a basis for each Klyachko space. By Proposition \ref{arrangementspaces}, $\B_\sigma \cap F_m(\v\!\!\mid_\sigma)$ is also a basis for $m \in M$ and $\sigma \in \Sigma$, and by Proposition \ref{free},  $\B_\sigma \cap G^{\rho}_r(\v)$ is a basis for all $\rho \in \sigma \cap N$. Take $b \in \B_\sigma$ and write it as a linear combination $\sum_{i =1}^\ell C_i b_i'$ of members of a linear adapted basis $\B'$ of $\v\!\!\mid_\sigma$ with $\deg(b_i') = \lambda_i$.  It follows that $b \in \sum_{i =1}^\ell F_{\lambda_i}(\v)$. The set $\B_i = \B_\sigma \cap F_{\lambda_i}(\v)$ is a basis, and $b \in \bigcup_{i =1}^\ell \B_i$, so for some $j$, $b\in F_{\lambda_j}(\v)$.  As $\B'$ is adapted, it follows that each $b_i' \in F_{\lambda_j}(\v)$ as well, so $\lambda_i \leq \lambda_j$ for each $i$.  As a consequence, $\v(b) = \bigoplus_{i =1}^\ell \lambda_i = \lambda_j$. It follows that $\B_\sigma$ is a linear adapted basis.  
\end{proof}


\section{Strong Khovanskii bases}\label{StrongKhovanskii}

Let $\Sigma$ be a smooth fan with toric variety $Y(\Sigma)$.  The $T_N$-linearized line bundles on $Y(\Sigma)$ correspond to functions $\psi: |\Sigma| \to \Q$ whose restriction to each cone $\sigma \in \Sigma$ is linear and has integer values on lattice points.  In turn, such functions are determined by their values $\psi(\rho_i)$ for $\rho_i \in \varrho_i \in \Sigma(1)$.  Thus, letting $|\Sigma(1)| = n$, we identify the $T_N$-linearized line bundles on $Y(\Sigma)$ with the free group $\Z^n$.   

Let $\A$ be a toric flat sheaf of algebras on $Y(\Sigma)$ corresponding to a valuation $\v: A \to \O_\Sigma$ and let $L$ be a line bundle on $Y(\Sigma)$. The space of maps from $L$ to $\A$ decomposes into a direct sum of the $T_N$-equivariant maps from the $T_N$-linearized line bundles whose underlying line bundle is $L$:

\[\Hom_{Y(\Sigma)}(L, \A) \cong \bigoplus_{\O(\psi) \cong L} \Hom_{\Sh_{Y(\Sigma)}^M}(\O(\psi), \A).\]
By Corollary \ref{cor-invertiblehom}, each space $\Hom_{\Sh_{Y(\Sigma)}^M}(\O(\psi), \A)$ can be identified with the Klyachko space $F_\psi(\v) \subset A$.  The total section ring $\mathcal{R}(A, \hat{\v})$ of $\A$ is defined as the following direct sum:

\[\mathcal{R}(A, \hat{\v}) = \bigoplus_{L \in \Pic(Y(\Sigma))} \Hom_{Y(\Sigma)}(L, \A) \cong \bigoplus_{\psi \in \Z^n} F_\psi(\v).\]
Let $S_\Sigma$ be the Cox ring of $Y(\Sigma)$; recall that $S_\Sigma$ is isomorphic to the polynomial ring $\k[X_\varrho \mid \varrho \in \Sigma(1)]$.   The monomial in $S_\Sigma$ corresponding to $\phi \in \Z_{\leq 0}^n$ acts on $\mathcal{R}(\A, \hat{\v})$ by shifting a graded component $F_\psi(\v)$ into $F_{\psi + \phi}(\v)$.   By Proposition \ref{thm-adjoint}, the total section ring defines a quasivaluation $\hat{\v}: A \to \O_{C_\Sigma}$, where $C_\Sigma = \Q_{\leq 0}^n$; whence the notation $\mathcal{R}(A, \hat{\v})$. 

When $\A = \Sym(\E)$ for a toric vector bundle $\E$, the total section ring (denoted $\mathcal{R}(\E)$) is isomorphic to the Cox ring of the projectivized vector bundle $\mathbb{P}\E$ of rank $1$ quotients of $\E$, see \cite{GHPS}.  The main result of this section is Theorem \ref{thm-main-StrongKhovanskii} which characterizes the finite generation of $\mathcal{R}(A, \hat{\v})$ in terms of the map $\Phi_\v$, or equivalently the diagram $D(\Phi_\v)$.

\subsection{Total section rings as Rees algebras}

Corollary \ref{cor-main-family} implies that the total section ring $\mathcal{R}(A, \hat{\v})$ is determined by the valuations $\v_{\rho_i}: A \to \overline{\Z}$, or equivalently the rows of the diagram $D(\Phi_\v)$.  With this in mind, we broaden our perspective and consider finite generation properties of general \emph{multi-Rees algebras}. 

\begin{definition}\label{def-multiRees}
Let $A$ be an algebra of finite type over $\k$, and let $\v_i: A \to \overline{\Z}$, $1 \leq i \leq n$ be quasivaluations with a common Khovanskii basis $\mathcal{B}$.  For $\br = (r_1, \ldots, r_n) \in \Z^n$, let $F_\br = \bigcap_{i = 1}^n G_{r_i}(\v_i)$.  The multi-Rees algebra is the direct sum:

\[\mathcal{R}(A, \v_1, \ldots, \v_n) = \bigoplus_{\br \in \Z^n} F_\br.\]
\end{definition}

For a toric flat algebra $\A$, the total section ring $\mathcal{R}(A, \hat{\v})$ is the multi-Rees algebra defined by $\v_i = \v_{\rho_i}$.  Moreover, if we let $\Sigma(n) \subset \Q^n$ be the fan composed only of the rays $\Q_{\geq 0}{\bf e}_i$ for $1 \leq i \leq n$, any multi-Rees algebra $\mathcal{R}(A, \v_1, \ldots, \v_n)$ can be realized as a total section ring by defining $\Phi: \Sigma(n) \to \Trop(I_\mathcal{B})$ to be the map which sends ${\bf e}_i$ to the weight of $\v_i$ in $\Trop(I_\mathcal{B})$.  Total section rings for families over $Y(\Sigma)$ are then those multi-Rees algebras whose valuations $\v_i$ satisfy the  apartment compatibility conditions of Remark \ref{rem-simplicialdiagrams}.   From this perspective, Cox rings $\mathcal{R}(\E)$ of projectivized toric vector bundles correspond to those multi-Rees algebras where $A$ is a polynomial ring, and the $\v_i$ are weight valuations emerging from points on a tropicalized linear space.   

The next lemma determines properties of multi-Rees algebras, and sharpens the description of a total section algebra. For $\v: A \to \O_\Sigma$, $j \leq n$ and $\br \in \Z^j$, let $F^j_\br(\v) = \bigcap_{1 \leq i \leq j} G_{r_i}(\v_i)$.  We define the partial multi-Rees algebras $R_j = \bigoplus_{\br \in \Z^j} F^j_\br(\v)$, and let $S_j = \k[X_1, \ldots, X_j]$,  and $T_j$ be the Laurent polynomial ring $\k[t_1^\pm, \ldots t_j^\pm]$.  Let $\m_j = \langle X_1 -1, \ldots, X_j -1\rangle \subset S_j$. 

\begin{lemma}\label{lem-iteratedRees}
For each $j$, $R_j$ is an algebra over $S_j$, and is naturally identified with a subalgebra of $A \otimes_\k T_j$.  For each $j$, $\v_{\rho_{j+1}}$ defines a $\Z^j$-homogeneous quasivaluation on $R_j$, and the Rees algebra of this valuation is $R_{j+1}$.  In particular, $R_j \cong R_{j+1}/\langle X_{j+1} - 1\rangle$, $gr_{\v_{\rho_{j+1}}}(R_j) \cong R_{j+1}/ \langle X_{j+1} \rangle$, and $\mathcal{R}(A, \hat{\v}) = R_n$. Moreover, $\mathcal{R}(A, \hat{\v})$ is eversive, and if $w \in C_\sigma \subset C_\Sigma$ for $\sigma \in \Sigma$, then $\hat{\v}_w = \v_w$.  
\end{lemma}

\begin{proof}
The algebra $R_1$ is clearly a Rees algebra of $A$, and a subalgebra of $A\otimes_\k \k[t_1^{\pm}]$, where we view $X_1$ as $t_1^{-1}$.  If this statement holds for $j$, then we use $R_j \subset A \otimes_\k T_j$ and the fact that $\v_{\rho_{j+1}}$ defines a $\Z^j$-homogeneous quasivaluation on $A \otimes_\k T_j$.  It is then straightforward to check that the $r_{j+1}-st$ filtration level of $F^j_{r_1, \ldots, r_j}(\v)$ with respect to this quasivaluation is $F_{r_1, \ldots, r_j, r_{j+1}}^{j+1}(\v)$.

It is straightforward to check that the graded components of $R_j/\langle X_j-1\rangle$ are the spaces $\bigcap_{i =1}^{j-1} G^{\rho_i}_{r_i}$. Consequently, $R_j/\langle X_j -1\rangle \cong R_{j-1}$, and the algebra $\mathcal{R}(A, \hat{\v})/\m_n \mathcal{R}(A, \hat{\v})$ is isomorphic to $R_1/ \langle X_1 -1 \rangle \cong A$.  It follows that $\mathcal{L}(\mathcal{R}(A, \hat{\v})) = (A, \hat{\v})$ for some quasivaluation $\hat{\v}: A \to \O_{C_\Sigma}$. Moreover, if $w\in C_\sigma \subset C_\Sigma$, where $w = \sum_{\rho_i \in \sigma(1)} w_i \hat{\rho}_i$ for ray generators $\hat{\rho}_i \in C_\sigma$, we have:

\[G^w_r(\hat{\v}) = \sum_{\sum_{\rho_i \in \sigma(1)} w_i\psi(\rho_i) \geq r } F_\psi(\v) = G^{\pi(w)}_\mathcal{R}(\v).\]
In particular, this implies that $\hat{\v}_{\hat{\rho}_i} = \v_{\rho_i}$ as quasivaluations on $A$. Now, by construction, the quasivaluation $\hat{\v}$ is concave with respect to the rays of $C_\Sigma$, so we have:

\[\bigcap_{\tilde{\rho} \in C_ \Sigma(1)} G^{\tilde{\rho}}_{\psi(\rho)}(\hat{\v}) = \bigcap_{\rho \in \Sigma(1)} G^\rho_{\psi(\rho)}(\v)  = F_\psi(\v) \subset  \bigcap_{w \in C_\Sigma} G^{w}_{\sum w_i\psi(\rho_i)}(\hat{\v}) \subset \bigcap_{\tilde{\rho} \in C_ \Sigma(1)} G^{\tilde{\rho}}_{\psi(\rho)}(\hat{\v}).\]
The middle term $F_\psi(\v)$ is a graded component of $\mathcal{R}(A, \hat{\v})$, and the inclusion into the space  $\bigcap_{w \in C_\Sigma} G^{w}_{\sum w_i\psi(\rho_i)}(\hat{\v})$ is a consequence of the definition of the functor $\L$.  The equality on the left, and the inclusion on the right then shows that $\mathcal{R}(A, \hat{\v})$ is eversive. 
\end{proof}

\subsection{Strong Khovanskii bases and proof of Theorem \ref{thm-main-StrongKhovanskii}}

By Proposition \ref{finitetype} and Lemma \ref{lem-iteratedRees} above, if $\mathcal{R}(A, \hat{\v})$ is finitely generated as a $\k$-algebra, then the image $\mathcal{B} \subset A$ of this set of generators is a finite Khovanskii basis for both $\hat{\v}$ and $\v$.  

\begin{definition}
We say a set $\mathcal{B} \subset A$ is a strong Khovanskii basis for $\v: A \to \O_\Sigma$ if it is the image of a generating set of $\mathcal{R}(A, \hat{\v})$.  
\end{definition}

The following proposition gives the general picture of what is going on when there is a strong Khovanskii basis.  It has also been discovered independently by N\o dland, \cite{Nodland} in the case of toric vector bundles. 

\begin{proposition}\label{prop-StrongKhovanskiispan}
Let $(A, \v)$ be as above, then $\mathcal{B} \subset A$ is a strong Khovanskii basis if and only if each space $F_\psi(\v)$ is spanned by monomials in $\mathcal{B}$. 
\end{proposition}

\begin{proof}
Any strong Khovanskii basis must have the stated property.  For sufficiency, we note that $\mathcal{B}$ will be a Khovanskii basis for $\v$, indeed, $\mathcal{B}$ monomials span each $G^\rho_r(\v)$ space.  If we let $\pi: \k[\bx] \to A \to 0$ be the presentation by $\mathcal{B}$, it follows that $\v = \pi_*\w$, where $\w(x_i) = \v(b_i)$, $b_i \in \mathcal{B}$.  We get also get an induced map $\pi_\psi: F_\psi(\w) \to F_\psi(\v)$ for each $\psi \in \Div(\Sigma)$.  By assumption, each $\pi_\psi$ is surjective.   The algebra $\mathcal{R}(\k[\bx], \hat{\w}) \subset \k[\bx][t_1^\pm, \ldots, t_n^\pm]$ has a basis of those monomials in $\bx, t_1, \ldots, t_n$ which satisfy certain integral inequalities coming from the rows of $D(\Phi_\v)$. It follows that $\mathcal{R}(\k[\bx], \hat{\w})$ is a saturated affine semigroup algebra, and is therefore finitely generated.  
\end{proof}

Unfortunately, it is not enough to show that $\hat{\v}: A \to \O_{C_\Sigma}$ has a finite Khovanskii basis.  Recall that this means that the associated graded $gr_\rho(A)$ is finitely generated for all $\rho \in C_\Sigma \cap \Z^n$.  Finite generation of the toric fixed point fiber $\mathcal{R}(A, \hat{\v})/\m_{C_\Sigma} \mathcal{R}(A, \hat{\v})$ is the obstruction to finite generation of $\mathcal{R}(A, \hat{\v})$, however the maps $\hat{\phi}_\rho: \mathcal{R}(A, \hat{\v})/\m_{C_\Sigma} \mathcal{R}(A, \hat{\v}) \to \gr_\rho(A)$ from Proposition \ref{finitetype} are only surjections in general. It could be the case that the associated graded algebras $\gr_\rho(A)$ are finitely generated when $\mathcal{R}(A, \hat{\v})$ is not.  In particular, $\mathcal{R}(A, \hat{\v})$ may not be flat over $S_\Sigma$.  Moreover, if $\mathcal{R}(A, \hat{\v})$ is flat as an $S_\Sigma$ module and $A$ is positively graded then $\mathcal{R}(A, \hat{\v})$ must be free.  In this case, $\v$ has a global adapted basis, and $\A$ splits as a direct sum of line bundles.  This special circumstance holds in the cluster algebra example in Section \ref{clusterexample}.  The next Proposition shows that we can guarantee this circumstance holds if we assume that the image of $\Phi$ is ``small."

\begin{proposition}\label{prop-singleGrobnercone}
Let $\v: A \to \O_\Sigma$ have a finite Khovanskii basis $\mathcal{B} \subset A$, and suppose that the image of the induced map $\Phi: \Sigma \to \mathcal{K}(I)$ lies in a single apartment, then the associated sheaf $\mathcal{A}$ splits as a direct sum of line bundles, and $\mathcal{B}$ is a strong Khovanskii basis. 
\end{proposition}

\begin{proof}
By assumption, $\v$ is the pushforward of a valuation $\w: \k[\bx] \to \O_\Sigma$ determined by $\w(\bx^\alpha) = \sum \alpha_i \phi_i$, where $\phi_i = \v(b_i)$ for $b_i \in \mathcal{B}$.   For each $\psi \in \Div(\Sigma)$ there is a map $\pi_\psi: F_\psi(\w) \to F_\psi(\v)$.  The space $F_\psi(\w)$ has as a basis those monomials $\bx^\alpha$ such that $\w(\bx^\alpha) \geq \psi$.   Let $\mathbb{B} \subset A$ be the image of the set of standard monomials for the face $\tau$, then $\mathcal{B}$ is a global adapted basis of $\v$, so $\mathcal{A}$ splits as a sum of line bundles.  Moreover, for each $\psi$ we must have $F_\psi(\v) \cap \mathbb{B} \subset F_\psi(\v)$ is a basis.  It follows that $\pi_\psi$ is always surjective, so we may use Proposition \ref{prop-StrongKhovanskiispan}. 
\end{proof}

Now we show that it is always possible to construct a strong Khovanskii basis if one exists by adapting \cite[Algorithm 2.18]{Kaveh-Manon-NOK}.  Let $\v: A \to \overline{\Z}$ be a quasivaluation with associated graded algebra $\gr_\v(A)$.  Let $\mathcal{B} = \{b_1, \ldots, b_d\} \subset A$ be a generating set with ideal $I_\mathcal{B}$, and let $J_\mathcal{B}$ be the ideal of polynomials which vanish on the image $\overline{\mathcal{B}} \subset \gr_\v(A)$.  We have $\In_\v(I_\mathcal{B}) \subset J_\mathcal{B}$, and these ideals are equal if and only if $\mathcal{B}$ is a Khovanskii basis of $\v$ \cite[Proposition 2.17]{Kaveh-Manon-NOK}.  Let $J_\mathcal{B} = \langle G \rangle \subset \k[\bx]$.  

\begin{enumerate}
\item For an element $g \in G$, compute $h = g(b_1, \ldots, b_d) \in A$ and $\overline{h} \in \gr_\v(A)$.
\item Check if $\overline{h}$ can be written as a polynomial in $\overline{\mathcal{B}} \subset \gr_\v(A)$. 
\item If $\overline{h}$ cannot be written as a polynomial in $\overline{\mathcal{B}}$ then add $h$ to $\mathcal{B}$. 
\item Repeat steps $(1) - (3)$ for each element of $G$. 
\item If $\mathcal{B}$ is unchanged then $J_\mathcal{B} = \In_\v(I_\mathcal{B})$ so return $\mathcal{B}$, otherwise go back to step $(1)$. 
\end{enumerate}

The Rees algebra $\mathcal{R}(A, \v) = \bigoplus_{r \in \Z} F_r(\v)$ is naturally a subalgebra of $A \otimes_\k \k[t^\pm]$.  A generating set $\overline{\mathcal{B}} \subset \gr_\v(A)$ then lifts to a $\k[t^{-1}]$-algebra generating set $\hat{\mathcal{B}} \subset \mathcal{R}(A, \v)$, where $\hat{b}_i = b_i \otimes t^{\v(b_i)}$.  In particular $\mathcal{R}(A, \v)$ is generated by $\hat{\mathcal{B}} \cup \{t^{-1}\}$ as a $\k$-algebra.  As a consequence, one knows that $R(A, \v)$ has a finite generating set if and only if $(A, \v)$ has a finite Khovanskii basis.  We use this perspective to prove finiteness properties for $\mathcal{R}(A, \hat{\v})$.

\begin{algorithm}\label{alg-StrongKhovanskii}
Let $R_0 = A$, and $\mathcal{B}_0 \subset R_0$ be a finite generating set. 

\begin{enumerate}
\item Starting with a generating set $\mathcal{B}_j \subset R_j$, use \cite[Algorithm 2.18]{Kaveh-Manon-NOK} to construct a Khovanskii basis $\mathcal{B}_j'$ of $\v_{\rho_{j+1}}$. 
\item Set $\mathcal{B}_{j+1} = \hat{\mathcal{B}_j'} \cup \{t_{j+1}^{-1}\} \subset R_{j+1} \subset A \otimes_\k T_{j+1}$, and go to $(1)$.  
\end{enumerate}

\end{algorithm}

\begin{proposition}\label{prop-StrongKhovanskii}
A pair $(A, \v)$ has a strong Khovanskii basis if and only if \ref{alg-StrongKhovanskii} stops. 
\end{proposition}

\begin{proof}
If \ref{alg-StrongKhovanskii} stops, $(A, \v)$ has a strong Khovanskii basis.  Conversely, if $\mathcal{B} \subset R$ is a finite strong Khovanskii basis, then its specialization in $R_j$ is a generating set.  It follows that the $j$-th instance of \cite[Algorithm 2.18]{Kaveh-Manon-NOK} terminates for each $j$. 
\end{proof}


Now we give an algebraic criterion for Algorithm \ref{alg-StrongKhovanskii} to stop. This will be a key step in the proof of Theorem \ref{thm-main-StrongKhovanskii}. 

\begin{proposition}\label{prop-ideal-StrongKhovanskii}
Let the pair $(A, \v)$ have Khovanskii basis $\mathcal{B} \subset A$ with $\v$ a valuation, and let $\I_{\mathcal{B}}$ be the kernel of the map $\k[\overline{X}, \overline{Y}]\to A \otimes_\k \k[t_1^\pm, \ldots,t_n^\pm]$, where $X_i \to 1 \otimes t_i^{-1}$ and $Y_j \to b_j \otimes t^{\hat{\v}(b_j)}$,  then $\mathcal{B}$ is a strong Khovanskii basis if and only if the ideal $\langle \I_{\mathcal{B}}, X_i \rangle$ is prime for all $i \in \{1, \ldots, n\}$. 
\end{proposition}

\begin{proof} If $\mathcal{B}$ is a strong Khovanskii basis, then we can reorder the variables $X_1, \ldots, X_n$ so that $\k[\overline{X}, \overline{Y}]/\langle \I_\mathcal{B}, X_i\rangle$ $\cong \gr_{\v_{\rho_i}}(R_{n-1})$ as in Algorithm \ref{alg-StrongKhovanskii}.  The latter algebra is a domain since $\v_{\rho_i}$ is a valuation, so Lemma \ref{lem-iteratedRees} implies that $\langle I, X_i\rangle$ is prime for all $i \in \{1, \ldots, n\}$. 

We prove the converse by induction. We let $R_m \subset A \otimes_\k T_m$ be as in Algorithm \ref{alg-StrongKhovanskii}, and we show that $R_m$ is finitely generated by the images of $X_1, \ldots, X_m$ and $\overline{Y}$ for each $m \in \{1, \ldots, n\}$.  The base case $R_0 = A$ is handled by the fact that $\mathcal{B}$ is a weak Khovanskii basis.  Assume this holds for $R_m$, and let $\I_m$ denote the kernel of $\phi_m: \k[X_1, \ldots, X_m, \overline{Y}] \to R_m$.  Let $\In(\I_m)$  denote the initial ideal of $\I_m$ with respect to the term weighting induced by the valuation $\v_{\rho_{m+1}}: A \otimes_\k T_{m+1} \to \overline{\Z}$, and let $J_m \subset \k[X_1, \ldots, X_m, \overline{Y}]$ denote the kernel of the map $\psi_m: \k[X_1, \ldots, X_m, \overline{Y}] \to gr_{\v_{\rho_{m+1}}}(R_m) \subset gr_{\v_{\rho_{m+1}}}(A) \otimes_\k T_m$.  Observe that $\v_{\rho_{m+1}}$ is homogeneous with respect to the $\Z^{m+1}$ grading on $R_m$.   We have $\In(\I_m) \subset J_m$ as above, and moreover, the images of $X_1, \ldots, X_m, \overline{Y}$ are a Khovanskii basis for $\v_{\rho_{m+1}}$ if and only if this is an equality.  In this case, the Rees algebra $R_{m+1}$ is finitely generated by the images of $X_1, \ldots, X_{m+1}, \overline{Y}$ as in the statement of the theorem above, so to prove the induction step, it suffices to show that $\In(\I_m) = J_m$.  The quotient algebra of $J_m$ is a domain, and the algebra obtained from this quotient by inverting the $X_1, \ldots, X_m$ is $gr_{\v_{\rho_{m+1}}}(A)\otimes_\k T_m$.   Therefore, it suffices to show that the quotient algebra of $\In(\I_m)$ is also a domain of the same dimension.


We prove the required statement with two general observations about algebras with gradings.  Let $S = \k[x, \overline{y}]/J$, where $J$ is prime and homogeneous with respect to a $\Z$-grading, and $deg(x) = -1$, then the kernel $J_0$ of the map $\k[\overline{y}] \to S/\langle x-1 \rangle$ is prime.  To see this, let $\phi$ be an isomorphism between a polynomial rings $\k[\ldots, \frac{y}{x^{deg(y_i)}}, \ldots]$ and $\k[\overline{y}]$ obtained by setting $x$ to $1$. Let $(\frac{1}{x}J)_0 \subset \k[\ldots, \frac{y}{x^{deg(y_i)}}, \ldots]$ be the prime ideal of degree $0$ forms.  Each element of $(\frac{1}{x}J)_0$ is mapped into $J_0$ by $\phi$, and both ideals have the same height, so it follows that the image $\phi((\frac{1}{x}J)_0)$ is $J_0$.  Moreover, the dimension of $S/\langle x-1 \rangle$ is one less than the dimension of $S$.   By repeatedly applying this observation to the prime ideal $\langle \I, X_{m+1} \rangle$, we see that the kernel $\I_m'$ of the map $\k[X_1, \ldots, X_m, \overline{Y}] \to \k[\overline{X}, \overline{Y}]/\langle X_{m+1}, X_{m+2}-1, \ldots, X_n -1 \rangle$ is prime, and the image algebra has the same dimension as $gr_{\v_{\rho_{m+1}}}(A)\otimes_\k T_m$.  
  
Next we consider the general case of an algebra $S \cong \k[x, \overline{y}]/J$ as above.  If $S = \bigoplus_{m \in \Z} S_m$, then the images $F_m$ of the spaces $S_m$ in $S/\langle x-1 \rangle$ form an algebra filtration, and $xF_m \subset F_{m-1}$.  The image $\mathcal{B}$ of the set $\overline{y}$ in $S/\langle x-1\rangle$ generates, and the same holds for the image $\overline{\mathcal{B}} \subset S/\langle x \rangle \cong gr_F(S)$.  It follows that $\mathcal{B}$ is a Khovanskii basis for the filtration $F$.  If $J_0$ denotes the kernel of the presentation of $S/\langle x-1\rangle$ and $J_0'$ denotes the kernel of the presentation of $S/\langle x \rangle$, we conclude that $J_0' = \In_F(J_0)$.   We apply this observation to the ideal $J = \I_{m+1}$, the kernel of the map $\phi_{m+1}: \k[X_1, \ldots, X_{m+1}, \overline{Y}] \to \k[\overline{X}, \overline{Y}]/\langle I, X_{m+2}-1, \ldots, X_n -1\rangle$ and $x = X_{m+1}$.   In this case, $J_0 = \I_m$, and $J_0' = \I_m'$ from above.  It follows that $\I_m' = \In(\I_m) \subset J_m$.  But we have seen that the quotient algebra of $\I_m' = \In(\I_m)$ has the same dimension as the quotient algebra of $J_m$.  This proves the induction step. \end{proof}

The localization of $R' = \k[\overline{X}, \overline{Y}]/\I_\mathcal{B}$ obtained by inverting $\phi(X_i), 1 \leq i \leq n$ is isomorphic to $A \otimes_\k T_n$.  As a consequence, there is an isomorphism $\tilde{\pi}: \Trop(\I_\mathcal{B}) \to \Trop(I_\mathcal{B}) \times \Q^n$, where $\tilde{\pi}(w_1, \ldots, w_d, m_1, \ldots, m_n) = (\ldots,w_j + \sum \v_{\rho_i}(b_j)m_i , \ldots, -m_1, \ldots, -m_n)$.  We let $\pi: \Trop(\I_\mathcal{B}) \to \Trop(I_\mathcal{B})$ be the composition of $\tilde{\pi}$ with the projection to the $\Trop(I_\mathcal{B})$ component, and we let $s: \Trop(I_\mathcal{B}) \to \Trop(\I_\mathcal{B})$ be the natural section $s({\bf w}) = ({\bf w}, 0, \ldots, 0)$. 

 The ideal $\I_\mathcal{B}$ is homogeneous with respect to the $\Z^n$ grading on $\k[\overline{X}, \overline{Y}]$, so for any $\m = (0, \ldots, 0, m_1, \ldots, m_n) \in \Q^{d + n}$ and ${\bf u} \in \Trop(\I_\mathcal{B})$ we have $\In_{{\bf u} + \m}(\I_\mathcal{B}) = \In_{{\bf u}}(\I_\mathcal{B})$. In particular, the fan structure on $\Trop(\I_\mathcal{B})$ induced from the Gr\"obner fan of $\I_\mathcal{B}$ is stable under addition by $\m \in \{(0, \ldots, 0)\} \times \Q^n \subset \Q^{d+n}$.  Moreover, if $\uu, \uu' \in \Trop(\I_\mathcal{B})$ have the same initial ideal, then the same holds for $\pi(\uu)$ and $\pi(\uu')$.   We give give $\Trop(I_\mathcal{B})$ the coarsest fan structure such that $\pi(\uu)$ and $\pi(\uu')$ share the same face whenever $\In_\uu(\I_\mathcal{B}) = \In_{\uu'}(\I_\mathcal{B})$.  This is a refinement of the Gr\"obner fan structure on $\Trop(I_\mathcal{B})$.   Faces $\tilde{C} \in \Trop(\I_\mathcal{B})$ are of the form $\tilde{C} \cong C \times \Q^n$ for $C \in \Trop(I_\mathcal{B})$.   We say that $\tilde{C} \in \Trop(\I_\mathcal{B})$ is the \emph{lift} of $C \in \Trop(I_\mathcal{B})$. 

We recall the notion of \emph{prime cone} from \cite{Kaveh-Manon-NOK}.  A face $C \in \Trop(I_\mathcal{B})$ is said to be a prime cone if its corresponding initial ideal is a prime ideal.  If $\Psi: |\Sigma| \to \Trop(I_\mathcal{B})$ corresponds to a valuation, then we must have that the face $C_i \in \Trop(I_\mathcal{B})$ containing $\Psi(\rho_i)$ is a prime cone for $1 \leq i \leq n$.  We let $\tilde{C}_i \in \Trop(\I_\mathcal{B})$ be the corresponding lift.

\begin{theorem}[Theorem \ref{thm-main-StrongKhovanskii}]\label{thm-primepointlift}
The set $\mathcal{B} \subset A$ is a strong Khovanskii basis if and only if $\tilde{w}_i = s\circ\Psi(\rho_i) \in \tilde{C}_i \subset \Trop(\I_\mathcal{B})$ is a prime point for each $\varrho_i \in \Sigma(1)$. 
\end{theorem}

\begin{proof}
For each ray $\rho_i$, the valuation $\v_{\rho_i}: A \to \overline{\Z}$ defines a valuation on $A \otimes_\k T_n$, which in turn induces a valuation on $R'$.  We let $\w_{\rho_i}: R' \to \overline{\Z}$ denote the weight quasivaluation induced by the values of $\v_{\rho_i}$ on the generating set $\{\phi(\overline{X}), \phi(\overline{Y})\} \subset R'$.   In particular, we have $\w_{\rho_i}(\phi(X_j)) = 0$ and $\w_{\rho_i}(\phi(Y_j)) = \v_{\rho_i}(b_j)$, so the weight of $\w_{\rho_i}$ is $\tilde{\Psi}(\rho_i) \in \tilde{C}_i$.   

There is an inclusion of algebras $R' \subset \mathcal{R}(A, \hat{\v})$.  We let $\textup{ord}_i : R' \to \overline{\Z}$ be the quasivaluation obtained by the $\langle \phi(X_i) \rangle$-adic filtration.  Moreover, we let $\textup{deg}_i: R' \to \overline{\Z}$ be the valuation obtained by reporting the homogeneous degree in the $i$-th direction.  We claim that $\textup{ord}_i + \textup{deg}_i = \w_{\rho_i}$ as homogeneous quasivaluations on $R'$.  This is a consequence of the fact that both sides of this equation take the same values on a shared Khovanskii basis $\{\phi(\overline{X}), \phi(\overline{Y})\}$.

Now, if $\mathcal{B}$ is a strong Khovanskii basis, then $R' = \mathcal{R}(A, \hat{\v})$, and $\langle \phi(X_i) \rangle \subset R'$ is a prime ideal.  It follows that $\textup{ord}_i$ is a valuation.  Since $\textup{deg}_i$ is a homogeneous grading function, this implies that $\w_{\rho_i}$ is a valuation as well, so $\tilde{C}_i$ is a prime cone.   Similarly, if $\tilde{C}_i$ is a prime cone, then $\w_{\rho_i}$ is a valuation, and $\textup{ord}_i$ must also be a valuation; this implies that $\langle \phi(X_i) \rangle$ is prime.  Proposition \ref{prop-ideal-StrongKhovanskii} completes the proof. 
\end{proof}


\section{Cox rings of projectivized toric vector bundles}\label{sec-Cox-TVB}

We apply our results from Section \ref{StrongKhovanskii} to the Cox rings of several classes of projectivized toric vector bundles. If $\E$ and $\E'$ are toric vector bundles such that $\E' = \E\otimes_{Y(\Sigma)} L$ for some $T_N$-linearized line bundle $L$, then the projectivizations agree: $\mathbb{P}\E' \cong \mathbb{P}\E$, and $\mathcal{R}(\E) \cong \mathcal{R}(\E')$. Because of this observation and Remark \ref{rem-tensorsumdiagram}, we assume that the minimum entry in each row of a diagram $D$ is $0$. 


We start by classifying those diagrams $D$ whose associated bundle $\E$ has $\mathcal{R}(\E)$ presented by the ``expected" relations. We call such a bundle \emph{complete intersection bundle}. More precisely, suppose that the linear ideal $L_\mathcal{B}  \subset \k[\overline{y}]$ is generated by linear forms $\overline{\ell} = \{\ell_1, \ldots, \ell_d\}$, where $\ell_i = \sum_{j=1}^r c_{i,j}y_j$, and let $M$ be the matrix with $(i, j)$-th entry equal to $c_{i,j}$.  We let $\alpha_i$ be the entry-wise minimum of the values $\hat{\v}(b_j)$ taken over the $y_j$ in the support of $\ell_i$.  Let $p_i = \frac{1}{X^{\alpha_i}}\sum_{j = 1}^r c_{i,j} X^{\hat{\v}(b_j)}Y_j \in \k[\overline{X}, \overline{Y}]$. It is straightforward to verify that $\langle p_1, \ldots, p_d \rangle \subseteq \I_\mathcal{B}$. 

\begin{definition}\label{def-formal}
Let $\E$ be a toric vector bundle corresponding to a $\v: \Sym(E) \to \O_\Sigma$ with Khovanskii basis $\mathcal{B} \subset E$. We say that $\E$ is a \emph{complete intersection bundle} if $\mathcal{B}$ is a strong Khovanskii basis and $\I_\mathcal{B} = \langle p_1, \ldots, p_d \rangle$, for $p_i$ $1 \leq i \leq d$ associated to $\ell_i \in L_\mathcal{B}$ as above. 
\end{definition}


By Theorem \ref{thm-main-StrongKhovanskii}, the projectivization $\mathbb{P}\E$ of any complete intersection bundle $\E$ is a Mori dream space.  Moreover, in this case the Cox ring $\R(\E)$ can be presented as a complete intersection.  For a subset $A \subset \{1, \ldots, n\}$ let $M_A$ be the matrix whose $j$-th row is given by the nonzero entries of the $j$-th row of $M$ where the rows $w_i$ in the diagram have a common minimal entry.   We let $m_A = \textup{rank}(M_A)$.



\begin{proposition}\label{prop-formal}
The image of the map $\Phi_\mathcal{B}: \k[\overline{X}, \overline{Y}] \to \Sym(E)\otimes_\k T_n$ is $\mathcal{R}(\E)$ with kernel $\I_\mathcal{B} = \langle p_1, \ldots, p_d\rangle$ if and only if for every $A \subset \{1, \ldots, n\}$ we have $d < |A| + m_A$, and for every $B \subset \{1, \ldots, n\}$ and $i \in B$ we have $1 + m_{\{i\}} < |B| + m_B$. 
\end{proposition}


\begin{proof}
Suppose the stated inequalities hold. We show that $\langle p_1, \ldots, p_d\rangle$ is a prime complete intersection. For dimension reasons, this implies that the $p_i$ generate $\I_\mathcal{B}$. Let $V \subset \mathbb{A}^n \times_\k \mathbb{A}^r$ be the scheme defined by $\langle p_1, \ldots, p_d\rangle$.  After inverting $X_1, \ldots, X_n$, the $p_i$ present $\Sym(E)\otimes_\k T_n$, which has dimension $n + r - d$.  For $A \subset \{1, \ldots, n\}$, let $\mathbb{A}^A \subset \mathbb{A}^n$ be the coordinate subspace defined by $X_i \mid i \in A$, and let $\mathbb{T}^A \subset \mathbb{A}^A$ be the torus orbit. The polynomials $p_{i, A} = p_i\!\!\mid_{x_i = 0, i \in A} \in \k[X_i \mid i \notin A, \overline{Y}]$ define $V \cap \mathbb{A}^A\times_\k \mathbb{A}^r$, and $V \cap \mathbb{T}^A\times_\k \mathbb{A}^r$ is a trivial vector bundle over $\mathbb{T}^A$ of rank $r - m_A$ with dimension $n-|A| + r - m_A$.  By assumption, $n - |A| + r - m_A < n + r - d$, so the dimension of $V$ is $n + r -d$, and $V$ is a generically reduced complete intersection.  The latter implies that $V$ is unmixed and reduced, so $\langle p_1, \ldots, p_d \rangle$ is prime.   A similar argument using the inequality $1 + m_{\{i\}} < |B| + m_B$ shows that $\langle \I_\mathcal{B}, X_i \rangle$ is prime for each $i$.  Conversely, we suppose $\I_\mathcal{B} = \langle p_1, \ldots, p_d \rangle$ and $\langle \I_\mathcal{B}, X_i \rangle$ are prime.  It follows that the height of $\I_\mathcal{B}$ is strictly smaller than the height of $\langle \I_\mathcal{B}, X_i\rangle$.  Moreover, $\langle \I_\mathcal{B}, X_i\rangle$ must have height strictly smaller than that of $\langle \I_\mathcal{B}, X_i, X_j \rangle$.  These two observations imply the inequalities  $d < |A| + m_A$ and $1 + m_{\{i\}} < |B| + m_B$. \end{proof}


\begin{definition}\label{def-uniform}
We say a bundle $\E$ is \emph{uniform} if we can find a Khovanskii basis $\mathcal{B} \subset E$ such that the associated matrix $M$ has all nonzero minors.
\end{definition}

\begin{remark}
Note that our condition defining uniform bundles is stronger than requiring $\mathcal{B}$ be a representation of a uniform matroid.  The $\mathcal{B}$ with $M$ having all nonzero minors form a dense, open subset in the representations of the uniform matroid. 
\end{remark}


\begin{corollary}\label{cor-uniform}
Let $\E$ be a uniform bundle with diagram $D$ and Khovanskii basis $\mathcal{B} \subset E$, then $\Phi_\mathcal{B}$ has image $\mathcal{R}(\E)$ with kernel $\I_\mathcal{B} = \langle p_1, \ldots, p_d \rangle$ if and only if for any $A \subset \{1, \ldots, n\}$ the nonzero entries of the corresponding rows of $D$ are contained in $r + |A| -2$ columns.
\end{corollary}

\begin{proof}
We have assumed that $M$ has all nonzero minors, so $m_A$ is equal to the minimum of $d$ and the number $s_A$ of columns of $D$ where the rows from $A$ have a common zero entry. Now we use Proposition \ref{prop-formal}.  Observe that $m - s_{\{i\}} < r = m - d$, so $s_{\{i\}} > d$ for any $i \in \{1, \ldots, n\}$, and $m_{\{i\}} = d$.   Now the presentation statement is equivalent to $1 + d < |A| + \min\{s_A, d\}$ for all $A$, which is in turn equivalent to the requirement that all nonzero entries of the $A$-rows of $D$ be contained in $r + |A| - 2$ columns.  
\end{proof}




\noindent
In the case that $d = 1$ ( the \emph{hypersurface} case),  this condition states that any two rows of such a diagram $D$ must have a common zero entry.   Moreover, if we work with bundles over $\mathbb{P}^{n-1}$, then any collection of $n-1$ rows in $D$ must share a common apartment.  It follows that these rows share at least $d+1$ zero entries.  
\begin{corollary}
Let $\E$ be a uniform bundle over $\mathbb{P}^{n-1}$ as above with $d < n-1$, then $\mathbb{P}\E$ is a Mori dream space. 
\end{corollary}


We say a bundle $\E$ is \emph{sparse} if we can find a Khovanskii basis $\mathcal{B} \subset E$ such that $D$ has at most one non-zero entry in each row.   Without loss of generality, we assume that the minimal circuits in $L_\mathcal{B}$ have size at least $3$. In this case $m_{\{i\}} = d$.   

\begin{corollary}\label{cor-sparse}
Let $\E$ be a sparse bundle with diagram $D$ and Khovanskii basis $\mathcal{B}$, then $\Phi_{\mathcal{B}}$ has image $R(\E)$ with kernel $\I_\mathcal{B} = \langle p_1, \ldots, p_d\rangle$.
\end{corollary}

\begin{proof}
Let $M = [T_A\mid M_A]$, then $|A|$ is at least equal to the number of columns in $T_A$. Without loss of generality we assume that $|A|$ equals the number of columns in $T_A$.  We must show that $d+1 < |A| + m_A$.   We consider three cases.  First, assume that the rank of $T_A$ is less than or equal to $|A| -2$, then $m_A \geq d- |A| +2$.   Next, assume that the rank of $T_A$ is $|A| - 1$.  We select an independent set $I_1$ from the columns of $T_A$ and complete it to a basis with columns $I_2$ from $M_A$.  We rearrange $M$ so that $I_1 \cup I_2$ come first, and row-reduce.  As the minimal circuits of $L_\mathcal{B}$ are of size at least $3$, the first $|I_1|$ rows of the resulting matrix must contain a non-zero entry in the columns past the first $d\times d$ block.  The row containing such an entry along with the last $I_2$ rows form a linearly independent set of size $d-|A| + 2$, so $m_A \geq d-|A|+2$.  Similarly, if $T_A$ has rank $|A|$, then $|I_2| = d-|A|$.  We consider the vectors formed by the first $I_1$ entries in the remaining columns in $M_A$.  Two of these vectors must be linearly independent, else $L$ contains a linear binomial in the members of $\mathcal{B}$ corresponding to the columns of $T_A$.  Taking the columns of $M_A$ in the positions of the two linearly independent vectors, plus the columns in $I_2$ then gives a linearly independent set of size $d-|A| + 2$. In all cases, we conclude that $d+1 < |A| + m_A$.    
\end{proof}

We can apply Corollary \ref{cor-sparse} to rank $2$ vector bundles, which were considered in \cite{Gonzalez}, and \cite{Nodland}.  If $D$ is the diagram of a rank $2$ vector bundle, then there are at most two distinct entries in each row of $D$.  To see this, note that the entries of the $i$-th row of $D$ are the values of a valuation $v_i: \Sym(E) \to \overline{\Z}$ on a spanning set $\mathcal{B} \subset E$.  It is an elementary consequence of the valuation axioms that members of $\mathcal{B}$ with distinct values must be linearly independent.  It follows that after an appropriate transformation, $D$ is the diagram of a sparse bundle.  The bundles considered in \cite[Section 6.2]{GHPS} and \cite[Proposition 5.2]{Nodland} are sparse by a similar argument.  

Now we have several examples illustrating Theorem \ref{thm-main-StrongKhovanskii} and Algorithm \ref{alg-StrongKhovanskii}.

\begin{example}[An example of Corollary \ref{cor-uniform}]
We consider the hypersurface bundle $\E$ over $\mathbb{P}^2$ defined by the ideal $L = \langle \sum_{i =1}^6 y_i \rangle \subset \k[\by]$ and the following diagram:  

\begin{equation}
D = \begin{bmatrix}
4 & 0 & 0 & 3 & 2 & 1\\
0 & 4 & 0 & 2 & 1 & 3\\
0 & 0 & 4 & 1 & 3 & 2\\
\end{bmatrix}
\end{equation}\\

Observe that $\E$ is uniform, but not sparse. Let $\mathcal{B} = \{b_1, \ldots, b_6\} \subset E$ be the corresponding arrangement. If $\psi_i$ is the $i$-th column, then $F_{\psi_i}(\v) = \langle b_i \rangle \subset E$.  The rays $\varrho_0, \varrho_1, \varrho_2$ of the fan of $\mathbb{P}^2$ correspond to the first, second, and third rows of $D$, respectively.  Each pair of rows lives in a common apartment of $\Trop(L)$. In particular, we have adapted bases over each face:  $\mathcal{B} \setminus \{b_1\}$ is adapted over $\sigma_0 = \R_{\geq 0}\{\varrho_1, \varrho_2\}$, $\mathcal{B} \setminus \{b_2\}$ is adapted over $\sigma_1 = \R_{\geq 0}\{\varrho_0, \varrho_2\}$, and $\mathcal{B} \setminus \{b_3\}$ is adapted over $\sigma_2 = \R_{ \geq 0}\{\varrho_0, \varrho_1\}$.  This diagram satisfies the condition in Corollary \ref{cor-uniform}, so $\mathcal{R}(\E)$ is presented by the following polynomial:

\[Y_1X_1^4 + Y_2X_2^4 + Y_3X_3^4 + Y_4X_1^3X_2^2X_3 + Y_5X_1^2X_2X_3^3 + Y_6X_1X_2^3X_3^2.\]

\end{example}

Next we see a hypersurface bundle over $\mathbb{P}^1 \times_\k \mathbb{P}^1$ whose Khovanskii basis is not strong. 

\begin{example}[An application of Algorithm \ref{alg-StrongKhovanskii}]\label{ex-Khovanskiiextend}
Let $\E$ be the rank $3$ bundle over $\mathbb{P}^1 \times_\k \mathbb{P}^1$ defined by the ideal $L = \langle \sum_{i=1}^4 y_1 \rangle \subset \k[\by]$ and the following diagram. 

\begin{equation}\label{diagram}
D = \begin{bmatrix}
2 & 1 & 0 & 0\\
0 & 2 & 1 & 0\\
0 & 0 & 2 & 1\\
1 & 0 & 0 & 2\\
\end{bmatrix}
\end{equation}\\

Let $\mathcal{B} \subset E$ be the corresponding arrangement. The ideal $\I_\mathcal{B}$ is generated by the polynomial $Y_1X_1^2X_4 + Y_2X_1X_2^2 + Y_3X_2X_3^2 + Y_4X_3X_4^2$.  The ideal $\langle \I_\mathcal{B}, X_1\rangle$ is not prime, so $\mathcal{B}$ is not a strong Khovanskii basis by Theorem \ref{thm-main-StrongKhovanskii}.  Algorithm \ref{alg-StrongKhovanskii} creates a strong Khovanskii basis from $\mathcal{B}$ by adding two new elements: $\mathcal{B}' = \mathcal{B} \cup\{ c_1, c_2\}$, where $c_1 = b_3 + b_4$ and $c_2 = b_1 + b_4$.   The diagram with respect to $\mathcal{B}'$ is obtained by adding two new columns obtained by applying the weight valuations represented by the rows of $D$ to $c_1$ and $c_2$:

\begin{equation}\label{diagram2}
D' = \begin{bmatrix}
2 & 1 & 0 & 0 & 1 & 0\\
0 & 2 & 1 & 0 & 0 & 1\\
0 & 0 & 2 & 1 & 1 & 0\\
1 & 0 & 0 & 2 & 0 & 1\\
\end{bmatrix}
\end{equation}\\

The ideal $\I_{\mathcal{B}'} \subset \k[X_1, X_2, X_3, X_4, Y_1, Y_2, Y_3, Y_4, Z_1, Z_2]$ is generated by $5$ polynomials:

\[ X_1X_3Y_3 +X_4^2Y_4 +X_1Z_1, \ \ X_1X_2Y_2+X_3^2Y_3-X_4Z_2,\] \[X_1X_4Y_1+X_2^2Y_2-X_3Z_1, \ \ X_3X_4Y_4+X_1^2Y_1+X_2Z_2,\] \[X_1X_3Y_1Y_3 - X_2X_4Y_2Y_4 -Z_1Z_2.\]

\bigskip

We conclude that $\mathcal{R}(\E) \cong  \k[\overline{X}, \overline{Y}, \overline{Z}]/\I_{\mathcal{B}'}$.

\end{example}

\begin{example}[A non-example]\label{ex-NonExample}

 We take $\E$ to be the rank $3$ bundle defined by the diagram:

\[D= \begin{bmatrix} 0 & 0 & 0 & 2 & 1 \\ 3 & 1 & 0 & 1 & 0\\ 0 & 1 & 3 & 0 & 1\end{bmatrix} \in \Delta(\mathbb{P}^2, \Trop(\langle y_4 - y_1 - y_2, y_5-y_2-y_3 \rangle)) \]\\

	\begin{figure}
		\centering

		\begin{minipage}{.5\textwidth}
		\begin{tikzpicture}[scale = .4]
		
		\draw[ultra thin] (0,0) coordinate (a_1) -- (0,12) coordinate (a_2);
		\draw[ultra thin] (a_2) -- (12, 0) coordinate (a_3);
		\draw[ultra thin] (a_3) -- (a_1);
		
		\filldraw[black] (0,11) coordinate (a_4) circle (4pt) node [anchor = east] {$\omega_1\omega_2^{3n-1}$};
		\filldraw[black] (8,0) coordinate (a_5) circle (4pt) node [anchor = north] {$\omega_1^n\omega_2^{2n}$};
		\filldraw[black] (8,3) coordinate (a_6) circle (4pt) node [anchor = west] {$\omega_1\omega_2^{n-1}\omega_3^{2n}$};
		\filldraw [black] (a_1) circle (4pt) node [anchor=north] {$\omega_1^{3n}$};
		\filldraw [black] (a_2) circle (4pt) node [anchor=east] {$\omega_2^{3n}$};
		
		\filldraw[color=red, fill=red!5] (1, 11) coordinate (a_7) circle (5pt);		
		\filldraw[color=red, fill=red!5] (2, 10) coordinate (a_8) circle (5pt);	
		 \filldraw[black] (2, 11) node [anchor = west] {$(2\omega_2^2\omega_3 + \omega_2\omega_3^2)\omega_2^{3(n-1)}$};	
		\draw[thick, red] (a_7) .. controls +(right:7mm) and +(up:7mm) .. (a_8);
		
		\filldraw[color=purple, fill=purple!5](6, 6)coordinate (a_9) circle (5pt);
		\filldraw[black] (6, 7) node [anchor = west] {$(2\omega_2^2\omega_3 + \omega_2\omega_3^2)^{n-1}\omega_2^3$};	
		\draw[thick, purple] (5,7.5) .. controls +(right:7mm) and +(up:7mm) .. (a_9);
			
		\filldraw[color=blue, fill=blue!5](8, 4)coordinate (a_10) circle (5pt);
		\filldraw[black] (8, 5) node [anchor = west] {$(2\omega_2^2\omega_3 + \omega_2\omega_3^2)^{n}$};	
		\draw[thick, blue] (7,5.5) .. controls +(right:7mm) and +(up:7mm) .. (a_10);

		\filldraw[color=black, fill=black!10,  ultra thick] (a_1) -- (a_4) -- (a_6) -- (a_5) -- cycle;
		
		\end{tikzpicture}
		\end{minipage}

		\begin{minipage}{.5\textwidth}
		\begin{tikzpicture}[scale = .4]
		
		\draw[ultra thin] (0,0) coordinate (a_1) -- (0,11) coordinate (a_2);
		\draw[ultra thin] (a_2) -- (11, 0) coordinate (a_3);
		\draw[ultra thin] (a_3) -- (a_1);
		
		\filldraw[black] (0,10) coordinate (a_4) circle (4pt) node [anchor = east] {$\omega_1\omega_2^{3n-2}$};
   	\filldraw[black] (7,0) coordinate (a_5) circle (4pt) node [anchor = north] {$\omega_1^n\omega_2^{2n-1}$};
		\filldraw[black] (7,3) coordinate (a_6) circle (4pt) node [anchor = west] {$\omega_1\omega_2^{n-1}\omega_3^{2n-1}$};
		\filldraw [black] (a_1) circle (4pt) node [anchor=north] {$\omega_1^{3n-1}$};
		\filldraw [black] (a_2) circle (4pt) node [anchor=east] {$\omega_2^{3n-1}$};
		
		\filldraw[color=red, fill=red!5] (1, 10) coordinate (a_7) circle (5pt);		
		\filldraw[color=red, fill=red!5] (2, 9) coordinate (a_8) circle (5pt);	
		 \filldraw[black] (2, 10) node [anchor = west] {$(2\omega_2^2\omega_3 + \omega_2\omega_3^2)\omega_2^{3(n-1)-1}$};	
		\draw[thick, red] (a_7) .. controls +(right:7mm) and +(up:7mm) .. (a_8);
		
		\filldraw[color=purple, fill=purple!5](6, 5)coordinate (a_9) circle (5pt);
		\filldraw[black] (6, 6) node [anchor = west] {$(2\omega_2^2\omega_3 + \omega_2\omega_3^2)^{n-1}\omega_2^2$};	
		\draw[thick, purple] (5,6.5) .. controls +(right:7mm) and +(up:7mm) .. (a_9);

		\filldraw[color=black, fill=black!10,  ultra thick] (a_1) -- (a_4) -- (a_6) -- (a_5) -- cycle;
		
		\end{tikzpicture}
		\end{minipage}

		\begin{minipage}{.5\textwidth}
		\begin{tikzpicture}[scale = .4]
		
		\draw[ultra thin] (0,0) coordinate (a_1) -- (0,10) coordinate (a_2);
		\draw[ultra thin] (a_2) -- (10, 0) coordinate (a_3);
		\draw[ultra thin] (a_3) -- (a_1);
		
		\filldraw[black] (0,9) coordinate (a_4) circle (4pt) node [anchor = east] {$\omega_1\omega_2^{3n-3}$};
		\filldraw[black] (6,0) coordinate (a_5) circle (4pt) node [anchor = north] {$\omega_1^n\omega_2^{2(n-1)}$};
		\filldraw[black] (6,3) coordinate (a_6) circle (4pt) node [anchor = west] {$\omega_1\omega_2^{n-1}\omega_3^{2(n-1)}$};
		\filldraw [black] (a_1) circle (4pt) node [anchor=north] {$\omega_1^{3n-2}$};
		\filldraw [black] (a_2) circle (4pt) node [anchor=east] {$\omega_2^{3n-2}$};
		
		\filldraw[color=red, fill=red!5] (1, 9) coordinate (a_7) circle (5pt);		
		\filldraw[color=red, fill=red!5] (2, 8) coordinate (a_8) circle (5pt);	
		 \filldraw[black] (2, 9) node [anchor = west] {$(2\omega_2^2\omega_3 + \omega_2\omega_3^2)\omega_2^{3(n-1)-2}$};	
		\draw[thick, red] (a_7) .. controls +(right:7mm) and +(up:7mm) .. (a_8);
		
		\filldraw[color=purple, fill=purple!5](6, 4)coordinate (a_9) circle (5pt);
		\filldraw[black] (6, 5) node [anchor = west] {$(2\omega_2^2\omega_3 + \omega_2\omega_3^2)^{n-1}\omega_2$};	
		\draw[thick, purple] (5,5.5) .. controls +(right:7mm) and +(up:7mm) .. (a_9);
		
		\filldraw[color=black, fill=black!10,  ultra thick] (a_1) -- (a_4) -- (a_6) -- (a_5) -- cycle;
		
		\end{tikzpicture}
		\end{minipage}

		\caption{Graded components $F_{0,3n; 3n}$ (top), $F_{0, 3n-1; 3n-1}$ (middle) and $F_{0, 3n-2; 3n-2}$ (bottom) of $\Gamma$.}
		\label{fig-NonExample}
	\end{figure}
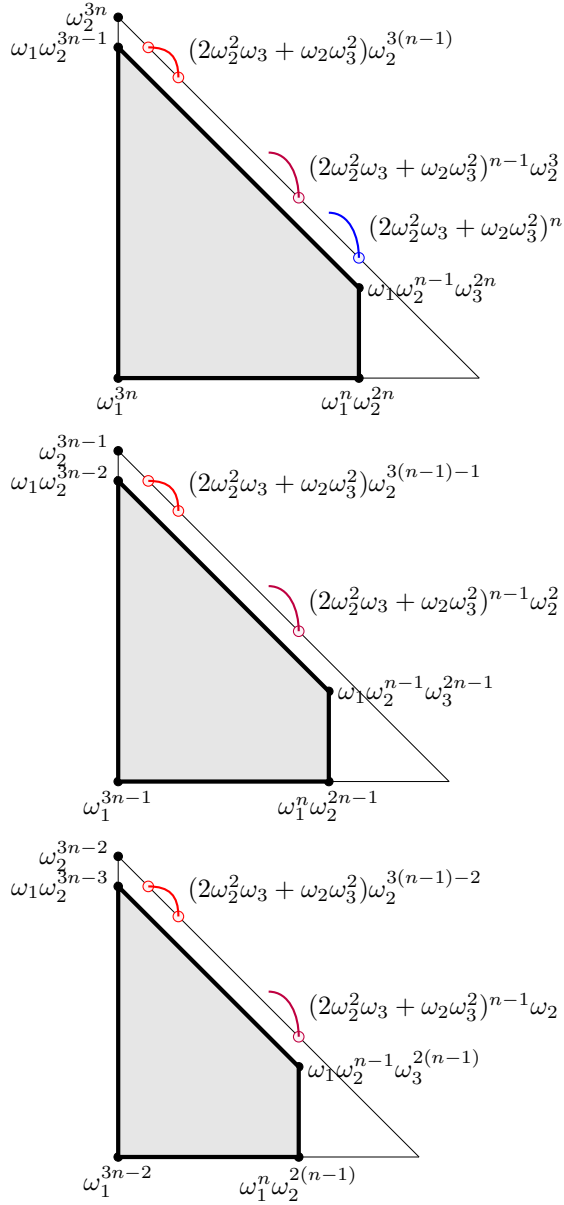

Let $\mathcal{R}_2 \subset \Sym(E)\otimes_\k \k[t_1^\pm, t_2^\pm]$ be the $2$nd Rees algebra in Algorithm \ref{alg-StrongKhovanskii}. This algebra is presented as a quotient of $\k[Y_1\ldots Y_5, X_1, X_2]$ by $\langle Y_5X_1 - Y_2X_2 -Y_3, Y_4X_1^2 -Y_1X_2^2 -Y_2\rangle$.  This ideal is the kernel of the map determined by sending $Y_1 \to \omega_1t_2^2$, $Y_2 \to \omega_2t_2$, $Y_3 \to \omega_3$, $Y_4 \to (\omega_1 + \omega_2)t_1^2t_2$, $Y_5 \to  (\omega_2 + \omega_3)t_1$, where $\omega_1, \omega_2, \omega_3 \in \Sym(E)$ are a set of polynomial generators.   

We show that the associated graded algebra $\gr_3(\mathcal{R}_2)$ of $\mathcal{R}_2$ with respect to the weight valuation $\v_{\rho_3}$ defined by the third row of $D$ is not finitely generated.  Let $F_{0,n; n} \subset \mathcal{R}_2$ be the subspace $F_{n, 0} \cap \Sym^n(E)$, and let $\Gamma = \bigoplus_{n \geq 0} F_{0, n; n}$.   As $\gr_3(\Gamma)$ is an invariant subring of $\gr_3(\mathcal{R}_2)$ by the action of a torus, its failure to be finitely generated implies that $\gr_3(\mathcal{R}_2)$ cannot be finitely generated, and so neither can the Rees algebra $\mathcal{R}_3= \mathcal{R}(\E)$.  An invariant theory computation shows that $\Gamma$ is generated by $F_{0, n;n}$ for $n = 1, 2, 3$.  These components have the following elements as basis members:

\[\{\omega_1, \omega_2\} \subset F_{0, 1;1} \ \ \ \{\omega_1^2, \omega_1\omega_2, \omega_2^2, \omega_1\omega_3\} \subset F_{0, 2;2} \]  \[\{\omega_1^3, \omega_1^2\omega_2, \omega_1\omega_2^2, \omega_2^3, \omega_1^2\omega_3, \omega_1\omega_2\omega_3, \omega_1\omega_3^2, (2\omega_2^2\omega_3 + \omega_2\omega_3^2) \} \subset F_{0, 3;3} \]

\bigskip

It can be shown by induction that the components $F_{0, 3n-2; 3n-2}, F_{0, 3n-1; 3n-1}, F_{0, 3n;3n}$ have bases depicted in Figure \ref{fig-NonExample}.  All monomials in the quadrillateral shaded region are included, as is $\omega_2^{3n}$ (resp. $\omega_2^{3n-1}, \omega_2^{3n-2}$).   Also included are the polynomials $(2\omega_2^2\omega_3 + \omega_2\omega_3^2)^k\omega_2^p$, where $p + 3k = 3n$ (resp. $3n-1$, $3n-2$).   This basis can be shown to be adapted to the valuation $\v_{\rho_3}$.   The components of the associated graded algebra $\gr_3(\Gamma) \subset \gr_3(\mathcal{R}_2)$ are similarly described, except the polynomials $(2\omega_2^2\omega_3 + \omega_2\omega_3^2)^k\omega_2^p$ are replaced with the monomials $(2\omega_2^2\omega_3)^k\omega_2^p$. This is the source of the infinite generation.   In the component $F_{0, 3n;3n}$, the corner element $\omega_1\omega_2^{n-1}\omega_3^{2n}$ is only reachable as a difference of $(2\omega_2^2\omega_3 + \omega_2\omega_3^2)^{n-1}\omega_1\omega_2^2 \in F_{0, 3(n-1); 3(n-1)}F_{0,3;3}$ and monomials obtained from the product $F_{0, 3n-1; 3n-1}F_{0,1;1}$.   However, after taking associated graded, no terms exist in lower degrees with monomials supported in the degrees necessary to reach this point.   It follows that for every $n$, $\omega_1\omega_2^{n-1}\omega_3^{2n} \in F_{0, 3n;3n}$ is an essential generator of $\gr_3(\Gamma)$. 

 \end{example}

\subsection{Diagram alterations and the fan $\Delta(\Sigma, \Trop_\star(I_\mathcal{B}))$}\label{sec-diagrams}

We finish this section by showing that the language of prime cones also provides a mechanism to make new Mori dream space bundles from old.   In principle, these results can be formulated for any multi-Rees algebra, but we focus on the projectivized toric vector bundle case.    To prepare, we make several definitions.   Let $\mathcal{B} \subset \Sym(E)$ be a (not necessarily linear) Khovanskii basis for a toric vector bundle bundle $\E$ defined by $\Psi: |\Sigma| \to \Trop(I_\mathcal{B})$. By construction, the rows of the corresponding diagram $D$ are in the subset $\Trop_\star(I_\mathcal{B}) \subset \Trop(I_\mathcal{B})$ of those points which come from the tropical variety $\Trop(L_\mathcal{B})$, where $L_\mathcal{B}$ is the ideal which vanishes on the linear part of $\mathcal{B}$. In particular, $w \in \Trop(L_\mathcal{B})$ defines a weight valuation $\v_w: \Sym(E) \to \overline{\Z}$, which in turn defines the point $(\v_w(b_1), \ldots, \v_w(b_r)) \in \Trop_\star(I_\mathcal{B})$. 

\begin{definition}\label{def-diagramfan}
Let $\Delta(\Sigma, \Trop_*(I_\mathcal{B}))$ be the subset of $\Trop_*(I_\mathcal{B})^n\subset \Q^{n \times r}$ consisting those matrices $D$ which are adapted to $\Sigma$. 
\end{definition}

By Proposition \ref{prop-diagram}, any $D \in \Delta(\Sigma, \Trop_\star(I_\mathcal{B}))$ defines a toric vector bundle over $Y(\Sigma)$.

\begin{proposition}\label{prop-diagramfan}
The set $\Delta(\Sigma, \Trop_\star(I_\mathcal{B})) \subset \Q^{n \times r}$ is the support of a polyhedral fan.
\end{proposition}

\begin{proof}
Give $\Trop_\star(I_\mathcal{B})$ the Gr\"obner fan structure.  If $(w_1, \ldots, w_n) \in \Delta(\Sigma, \Trop_\star(I_\mathcal{B}))$, and $w_i \in C_i$, where $C_i$ is the face of $\Trop_\star(I_\mathcal{B})$ containing $w_i$, then $C_1 \times \ldots \times C_n \subset \Delta(\Sigma, \Trop_\star(I_\mathcal{B}))$ as well.  
\end{proof}

Next we describe alterations to the rows of a diagram $D \in \Delta(\Sigma, \Trop_\star(I_\mathcal{B}))$ which preserve the finite generation of the corresponding total section ring.   

\begin{proposition}\label{prop-add-subtract}
Let $\E$ and $D$ be as above, with $\mathcal{R}(\E)$ finitely generated.  Let $w \in \Trop(\I_\mathcal{B})$ be a prime point with $\pi(w) \in \Trop_\star(I_\mathcal{B})$, and let $D'$ be obtained by appending $\pi(w)$ to $D$.  If $\Sigma'$ is a fan such that $D' \in \Delta(\Sigma', \Trop_\star(I_\mathcal{B}))$, then the associated bundle $\E'$ has $\mathcal{R}(\E')$ finitely generated.   If $D''$ is obtained from $D$ by deleting a row, and $D'' \in \Delta(\Sigma'', \Trop_\star(I_\mathcal{B}))$, then the associated bundle $\E''$ has $\mathcal{R}(\E'')$ finitely generated. 
\end{proposition}

\begin{proof}
By the construction in Lemma \ref{lem-iteratedRees}, the algebra $\mathcal{R}(\E')$ is the Rees algebra of $\mathcal{R}(\E)$ and the valuation defined by $\pi(w)$. As $\pi(w)$ is the image of a prime point, $\mathcal{R}(\E')$ is finitely generated.  In the second case, $\mathcal{R}(\E'')$ is a quotient of $\mathcal{R}(\E)$.   
\end{proof}

\noindent
Proposition \ref{prop-add-subtract} can be used to decide when the pullback of a Mori dream space bundle along a single toric blow-up is a Mori dream space bundle. 

\begin{corollary}\label{cor-toricblowup}
Let $\E$ be a toric vector bundle over a toric variety $Y(\Sigma)$ with diagram $D(\Phi) \in \Delta(\Sigma, \Trop_\star(I_\mathcal{B}))$, let $\beta:Y(\Sigma') \to Y(\Sigma)$ be a blow-up of a toric orbit of $Y(\Sigma)$ corresponding to a new ray $\rho = \Q_{\geq 0}p \subset |\Sigma|$, and let $\beta^*\E$ be the pullback bundle. If $\Phi(p)$ is contained in the image of the union of the set of prime cones of $\Trop(\I_\mathcal{B})$ under $\pi: \Trop(\I_\mathcal{B}) \to \Trop(I_\mathcal{B})$, then $\mathcal{R}(\beta^*\E)$ is finitely generated. 
\end{corollary}

\noindent
By Corollary \ref{cor-toricblowup}, if $\I_\mathcal{B}$ is well-poised, then $\mathcal{R}(\beta^*\E)$ is finitely generated for any such blow-up $\beta: Y(\Sigma') \to Y(\Sigma)$.  In particular, this holds for any sparse bundle.

By Theorem \ref{thm-main-StrongKhovanskii}, each row $w_i$ of $D$ is in the image $\pi(\tilde{C}_i) \subset \Trop(I_\mathcal{B})$ of a prime cone $\tilde{C}_i \subset \Trop(\I_\mathcal{B})$.  Let $K_i$ be the polyhedral complex $\pi(\tilde{C}_i) \cap \Trop_\star(I_\mathcal{B})$.  By Proposition \ref{prop-add-subtract}, we can append $w_i' \in K_i$ to $D$, then delete $w_i$ while maintaining a diagram adapted to $\Sigma$ and preserving finite generation of the associated total section ring.  Next we show that given $D \in \Delta(\Sigma, \Trop_\star(I_\mathcal{B}))$, with finitely generated total section ring, there is a polyhedral family of diagrams containing $D$ all of whose total section rings are also finitely generated.  To prepare for this construction, let $w_1, \ldots, w_n$ be the rows of $D \in \Delta(\Sigma, \Trop_\star(I_\mathcal{B}))$. We show the existence of a generating set $\GG \subset \I_\mathcal{B}$ with well-behaved initial forms. 


\begin{lemma}\label{lem-goodbasis}
There exists a generating set $\{f_1, \ldots, f_\ell\} = \GG \subset \I_\mathcal{B}$ with the property that for each $w_i$, the initial forms of $\GG$ generate $\In_{s(w_i)}(\I_\mathcal{B})$, and $\In_{s(w_i)}(f_j) = f_j\!\!\mid_{X_i = 0}$. 
\end{lemma}

\begin{proof}
We let $G_i \subset \I_\mathcal{B}$ be a Gr\"obner basis with respect to $w_i$.  We assume without loss of generality that the elements of $G_i$ are irreducible and homogeneous with respect to the $\Z^{n+1}$ grading on $\I_\mathcal{B}$.  We let $\GG = \bigcup_{i =1}^n G_i$.  It remains to show that $\In_{s(w_i)}(f) = f\!\!\mid_{X_j = 0}$ for each $f \in \F$.  This is a consequence of the homogeneity and irreducibility of $f$. 
\end{proof}

We define $C_\GG \subset \Delta(\Sigma, \Trop_\star(I_\mathcal{B}))$ to be the set of those diagrams $D'$ with the property that $\In_{s(w_i')}(f) = \In_{s(w_i)}(f)$ for each $f \in \F$, where $w_i'$ is the $i$-th row of of $D'$. 

\begin{proposition}\label{prop-coneStrongKhovanskii}
Let $\E'$ be the bundle defined by $D' \in C_\GG$, then $\mathcal{R}(\E')$ is finitely generated. 
\end{proposition}

\begin{proof}
We will make a diagram $\hat{D}$ for $\Sigma(2n)$ whose rows are $w_1, w_1', \ldots, w_n, w_n'$ and show that the associated total section ring is finitely generated.  The proposition then follows from the fact that $R(\E')$ is a specialization of this section ring. 

Starting with $\I_0 = \I_\mathcal{B} \subset \k[\overline{X}, \overline{Y}]$, we build a sequence of ideals $\I_i \subset \k[\overline{X}, x_1', \ldots, x_i', \overline{Y}]$ each satisfying Theorem \ref{thm-main-StrongKhovanskii}. By construction, $\In_{s(w_1')}(\I_0) = \langle f_1\!\!\mid_{X_1 = 0}, \ldots, f_\ell\!\!\mid_{X_1 = 0}\rangle$ is prime, it follows that the Rees algebra of $\v_{s(w_1')}$ and $\mathcal{R}(\E)$ can be presented as a quotient of $\k[\overline{X}, x_1', \overline{Y}]$ by the ideal $\I_1 = \langle \tilde{f}_1, \ldots, \tilde{f}_\ell\rangle$, where the $\tilde{f}_i$ are obtained as follows. The form $\tilde{f}$ is constructed from $f = \sum C_{\alpha, \beta}X^\alpha Y^\beta$ by appending a power $(x_1')^{p(\alpha, \beta)}$, where $p(\alpha, \beta)$ is the difference between $\v_{w_1'}(b^\beta)$ and the minimum such value that appears among the monomials of $f$.  By construction, $\tilde{f}$ is in the kernel of the presentation of the Rees algebra, and the initial forms $\In_{s(w_1')}(\tilde{f}) = \In_{s(w_1)}(f)$ generate a prime ideal of height equal to the height of $\I_1$.  It follows that the lifts $\tilde{f}$ generate $\I_1$, and constitute a set $\tilde{\GG} \subset \I_1$ satisfying the properties of Lemma \ref{lem-goodbasis}.  We may now repeat this procedure with $w_2'$, and so on.  
\end{proof}

Proposition \ref{prop-coneStrongKhovanskii} shows that the locus of diagrams in $\Delta(\Sigma, \Trop_\star(I_\mathcal{B}))$ which correspond to bundles $\E$ with $\mathcal{R}(\E)$ finitely generated is a union of polyhedral cones.  It is natural to ask if these cones are faces of a fan structure on $\Delta(\Sigma, \Trop_\star(\I_\mathcal{B}))$ (Conjecture \ref{conj-fan}).   The next proposition is a weakened version of such a result.  


\begin{proposition}\label{prop-weakfan}
There is a finite fan $\F$ supported on $\Delta(\Sigma, L_\mathcal{B})$ such that the set of diagrams $D \in \Delta(\Sigma, L_\mathcal{B})$ which satisfy Proposition \ref{prop-formal} for some minimal generating set $\{\ell_1, \ldots, \ell_d\} \subset L_\mathcal{B}$ is the union of faces of $\F$. 
\end{proposition}

\begin{proof}
Fix a minimal generating set $\overline\ell = \{\ell_1, \ldots, \ell_d\} \subset L_\mathcal{B}$. We define $F_{\overline\ell}$ to be the fan supported on $\Trop(L_\mathcal{B})$ by first taking the Gr\"obner fan of $L_\mathcal{B}$ and then refinining by the domains of linearity for the piecewise linear functions obtained by tropicalizing the $\ell_i$.  In particular there is a specific choice of initial forms of the $\ell_i$ corresponding to each face of $F_{\overline\ell}$.  We let $\F_{\overline\ell}$ be the fan obtained by taking the common refinement of $\Delta(\Sigma, L_\mathcal{B})$ with $F_{\overline\ell} \times \cdots \times \F_{\overline\ell} \subset \Trop(L_\mathcal{B})^n$. 

Any face $\sigma \in \F_{\overline\ell}$ determines initial forms of $\ell_i$ for each of $n$ choices of points from $\Trop(L_\mathcal{B})$, and this information in turn determines each matrix $M_A$.  It follows that for a given choice of ranks $\mathcal{M} = \{ m_A \mid A \subset \Sigma(1)\}$, there is a union of faces $\F_{\mathcal{M},\overline\ell} \subset \F_{\overline\ell}$ such that each diagram $D \in \sigma$ with $\sigma \subset \F_{\mathcal{M},\overline\ell}$ realizes $\mathcal{M}$.   As a consequence, the set of diagrams which satisfy the inequalities in Proposition \ref{prop-formal} are the points contained in union of the appropriate $\F_{\mathcal{M},\overline{\ell}}$.  Now we observe that $F_{\overline\ell}$ only depends on the supports of the $\ell_i \in \overline\ell$, and that there is a finite set of possible selections of supports.  We take the common refinement of the possible $F_{\overline\ell} \subset \Trop(L_\mathcal{B})$; this defines a fan structure $\F$ on $\Delta(\Sigma, L_\mathcal{B})$. 

\end{proof}


\begin{example}
We apply Propositions \ref{prop-coneStrongKhovanskii} and \ref{prop-add-subtract} to the bundle from Example \ref{ex-Khovanskiiextend}.  To find an appropriate set $\GG \subset \I_\mathcal{B}$ we must add two polynomials to the generating set:

\[ X_1X_3Y_3 +X_4^2Y_4 +X_1Z_1, \ \ X_1X_2Y_2+X_3^2Y_3-X_4Z_2,\] \[X_1X_4Y_1+X_2^2Y_2-X_3Z_1, \ \ X_3X_4Y_4+X_1^2Y_1+X_2Z_2,\] \[X_1X_3Y_1Y_3 - X_2X_4Y_2Y_4 -Z_1Z_2,\] \[X_2^3Y_2^2+X_4^2Y_1Z_1-X_2X_3Y_2Z_1-X_3^2X_4Y_1Y_3,\] \[X_3^3Y_3^2-X_1^2Y_2Z_1-X_3X_4Y_3Z_2-X_1X_4^2Y_2Y_4\]

\bigskip

In this case $\mathcal{B}'$ is composed of linear forms, so $I_{\mathcal{B}'} = L_{\mathcal{B}'}$. The cone $C_\GG \subset \Delta(\Sigma, \Trop(L_{\mathcal{B}'}))$ is composed of diagrams of the form:

\begin{equation}\label{diagram3}
D = \begin{bmatrix}
\ell_1 +2a & \ell_1 + a & \ell_1 & \ell_1 & \ell_1+a & \ell_1\\
\ell_2 & \ell_2 +2b & \ell_2 + b & \ell_2 & \ell_2 & \ell_2+b\\
\ell_3 & \ell_3 & \ell_3 + c + d & \ell_3 + c & \ell_3 + c& \ell_3\\
\ell_4 + e & \ell_4 & \ell_4 & \ell_4 + e + f & \ell_4 & \ell_4+ e\\
\end{bmatrix}
\end{equation}\\

\noindent
Here $\ell_1,\ldots, \ell_4 \in \Z$ correspond to addition by elements of the lineality space of $\Trop(L_{\mathcal{B}'})$ (equivalently tensoring $\E$ by $T_N$-linearized line bundles), and $a, b, c, d, e, f \in \Z_{\geq 0}$. 

The tropical variety $\Trop(\I_{\mathcal{B}'})$ contains one prime cone $\tilde{C}_0$ in addition to the $4$ prime cones containing the rows of the matrix $D'$ from Example \ref{ex-Khovanskiiextend}.  The image $\pi(\tilde{C}_0)$ is $C_0 = \Q(1,1,1,1,1,1) + \Q_{\geq 0}(0, 0, 0, 0, 1, 0) + \Q_{\geq 0}(0, 0, 0, 0, 0, 1)$.  Let $\Sigma'$ be the fan of the blow-up $\mathcal{BL}_p(\mathbb{P}^1\times_\k \mathbb{P}^1)$, where $p$ is the toric fixed point corresponding to the cone spanned by $\{(1,0), (0, -1)\}$ in the fan of $\mathbb{P}^1\times_\k \mathbb{P}^1$.  It can be shown that the following diagram is adapted to $\Sigma'$, where the bottom row is associated to the new ray in $\Sigma'$. 

 \begin{equation}\label{diagram4}
D'' = \begin{bmatrix}
2 & 1 & 0 & 0 & 1 & 0\\
0 & 2 & 1 & 0 & 0 & 1\\
0 & 0 & 2 & 1 & 1 & 0\\
1 & 0 & 0 & 2 & 0 & 1\\
0 & 0 & 0 & 0 & 1 & 1\\
\end{bmatrix}
\end{equation}\\
By Proposition \ref{prop-add-subtract}, the algebra $\mathcal{R}(\E'')$ corresponding to the bundle $\E''$ defined by this diagram is finitely generated. 
\end{example}


\section{Examples from flag varieties and cluster algebras}\label{examples}
In this section we discuss several examples of $\O_N$ valuations on the coordinate rings of Grassmannian varieties, and more general cluster varieties. We also present general methods to construct toric vector bundles.   

\subsection{Vector bundles from tropicalization}

By Theorem \ref{thm-main-bundle}, one can construct a toric vector bundle by finding a point $\Phi \in \Trop_{\O_N}(L)$, where $L$ is a linear ideal.  The following proposition is a consequence of the fact that the circuits of a linear ideal are a tropical basis. 

\begin{proposition}\label{constructbundle}
Let $\Psi = \overline{\psi}$ be an $n$-tuple of piecewise linear functions, and suppose that for each circuit $\sum_{i \in J} a_i\bx_i \in L$,  $\bigoplus_{i \in J} \psi_i = \bigoplus_{i \in J \setminus \{j\}} \psi_i$ for each $j \in J$, then $\overline{\psi} \in \Trop_{\O_N}(L)$.  Moreover, suppose $\overline{P}$ is a collection of $n$ integral polyhedra in $M_\Q$ such that $P_j$ is in the convex hull of the union $\bigcup_{i \in J \setminus \{j\}} P_i$ for each circuit as above, then the support functions of the polyhedra $\overline{P}$ define a point in $\Trop_{\O_N}(L)$. 
\end{proposition}

\begin{proposition}\label{linearconstructbundle}
Let $\overline{\lambda}(\bt) = (\lambda_1(\bt), \ldots, \lambda_n(\bt))$ with $\lambda_i(\bt) \in \k(T_N)$ be in the kernel of the matrix $M_I$ whose rows are given by the coefficients of the minimal circuits of $L$, then $(\w_N(\lambda_1), \ldots, \w_N(\lambda_n)) \in \Trop_{\O_N}(L)$.  Moreover, if $\lambda_i(\bt) \in S_0 \subset \k(T_N)$ for each $1 \leq i \leq n$, then each $\w_n(\lambda_i)$ is the support function of a polytope as in Lemma \ref{constructbundle}. 
\end{proposition}

\begin{proof}
Proposition \ref{trop} immediately implies that  $(\w_N(\lambda_1), \ldots, \w_N(\lambda_n)) \in \Trop_{\O_N}(L)$.  
\end{proof}

\subsection{$\O_N$ tropical points on Grassmannian varieties}

Let $I_{2, n}$ be the Pl\"ucker ideal which cuts out the affine cone $X_{2,n} \subset \bigwedge^2(\k^n)$ over the Grassmannian variety $Gr_2(\k^n)$ of $2$-planes in $n$-space.  In general, the ideal $I_{2, n}$ is well-poised. It follows that, similar to tropicalized linear spaces, flat $T_N$ families $\pi: \mathcal{G} \to Y(\Sigma)$ with reduced, irreducible fibers, general fiber $X_{2, n}$, and a $T_N$-equivariant embedding into the trivial bundle $\bigwedge^2(\k^n) \times Y(\Sigma)$ correspond to piecewise linear maps $\Psi: \Sigma \to \Trop(I_{2, n})$, or equivalently a point $\overline{\psi} \in \Trop_{\O_N}(I_{2, n})$.  Moreover, the Pl\"ucker generators $p_{ij}p_{k\ell} - p_{ik}p_{j\ell} + p_{i\ell4}p_{jk}$ for $1 \leq i < j < k < \ell \leq n$ form a tropical basis of $I_{2, n}$. 

Let us consider the Pl\"ucker ideal $I_{2, 4} = \langle p_{12}p_{34} - p_{13}p_{24} + p_{14}p_{23} \rangle$.  We choose four points $\m = \{m_1, m_2, m_3, m_4\} \subset M$ and we let $\Delta(\m)$ be the convex hull of the set $\m$. Finally, we let $P_{ij} = conv\{m_i, m_j\}$. The following is straightforward to verify by hand.

\begin{lemma}\label{24case}
The support functions $\psi_{ij}$ of the edges $P_{ij}$ define a point on $\Trop_{\O_N}(I_{2, 4})$. 
\end{lemma}

\begin{proposition}\label{2ncase}
Let $\m = \{m_1, \ldots, m_n\} \subset M$, with $\Delta(\m)$ and $P_{ij}$ defined as above, then the support functions $\psi_{ij}$ of $P_{ij}$ define a point on $\Trop_{\O_N}(I_{2,n})$. 
\end{proposition}

\begin{proof}
Pl\"ucker relations are a tropical basis, and by Lemma \ref{24case} we have $\overline{\psi} = (\ldots \psi_{ij}, \ldots) \in \Trop_{\O_N}(\langle p_{ij}p_{k\ell} - p_{ik}p_{j\ell} + p_{i\ell4}p_{jk}\rangle)$ for each $1 \leq i < j < k < \ell \leq n$.  
\end{proof}

From Proposition \ref{2ncase} we conclude that \emph{the 1-skeleton of an $n$-simplex is a point on the $\O_N$ tropical $(2, n)$ Grassmannian}.  The corresponding statement for $m$-skeleton of a simplex and the $(m, n)$ Grassmannian also holds. Let $A_{m, n}$ denote the $(m, n)$ Pl\"ucker algebra. 

\begin{theorem}\label{mncase}
Let $\m$ and $\Delta(\m)$ be as above, and let $P_{I} = conv\{m_i \mid i \in I\}$ for $I \subset \{1, \ldots, n\}$, $\mid I \mid = m$, then the tuple $\overline{\psi} = (\ldots \psi_I, \ldots)$ defines a point on $\Trop_{\O_N}(I_{m,n})$, where $\psi_I$ is the support function of $P_I$. 
\end{theorem}

\begin{proof}
We use a technique similar to the proof of Proposition \ref{linearconstructbundle}. Let $L(\m, \bt)$ be the following matrix with entries in $\k(T_N)$:\\

\[ 
\begin{bmatrix}
\bt^{m_1} & \cdots & \bt^{m_n}\\
a_{2,1} & \cdots & a_{2, n}\\
\vdots & \ddots & \vdots\\
a_{m,1} & \cdots & a_{m, n}
\end{bmatrix}
\]\\
where the $a_{ij}$ are chosen generically in $\k$.  By sending the Pl\"ucker generator $p_I$ to the $I$-minor $\delta_I$ of $L(\m, \bt)$, we obtain a ring map $\phi_\m: A_{m, n} \to \k(T_N)$, and a valuation $\w_N\circ \phi_\m: A_{m, n} \to \O_N$. One checks that $\delta_I(L(\m, \bt)) = \sum_{i \in I} d_i\bt^{m_i}$, where $d_i$ is a non-zero $m-1 \times m-1$ determinant in the $a_{ij}$. It follows that $\w_N\circ\phi_{\m}(p_I)$ is the support function of $P_I$.  The theorem is then a consequence of Proposition \ref{trop}. 
\end{proof}

\begin{example}
We return to the case $X_{2,4} \subset \bigwedge^2(\k^4)$, and its corresponding Pl\"ucker ideal $\langle p_{12}p_{34} - p_{13}p_{24} + p_{14}p_{23}\rangle$.    Let $m_1 = (0,0), m_2 = (1, 0), m_3 = (0, 1),$ and $m_4 = (0, 0)$. We view the convex hull of the $m_i$ as a degenerate tetrahedron in $\Q^2$.  Following Theorem \ref{mncase}, there is an associated ring map $A_{2, 4} \to Q_{(\Z^2)^*}$ given by the minors of a matrix:

\[
\begin{bmatrix}
1 & t^{10} & t^{01} & 1\\
a_{2, 1} & a_{2,2} & a_{2,3} & a_{2,4}\\
\end{bmatrix}
\] 

with generic $a_{ij}$.  We obtain an associated valuation $\w: A_{2, 4} \to \O_{\Z^2}$.  Letting $x: \Z^2 \to \Z$ and $y: \Z^2 \to \Z$ be projections onto the $x$ and $y$ axes, respectively, we have:

\[\w(p_{12}) = \min\{x, 0\} \ \ \w(p_{13}) = \min\{0,y\} \ \ \w(p_{14}) = 0\]  
\[\w(p_{23}) = \min\{x, y\} \ \ \w(p_{24}) = \min\{x, 0\} \ \ \w(p_{34}) = \min\{y, 0\}\]

\bigskip
The common refinement of the domains of linearity of these generators is the fan $\Sigma \subset \Q^2$.

\begin{figure}[!htbp]
\begin{tikzpicture}
\draw[->] (0, 0) -- (2,0);
\draw[->] (0, 0) -- (1.5,1.5);
\draw[->] (0,0) -- (0,2);
\draw[->] (0, 0) -- (-2,0);
\draw[->] (0,0) -- (-1.5,-1.5);
\draw[->] (0, 0) -- (0, -2);
\node at (1, 0) {$\bullet$};
\node at (1, 1) {$\bullet$};
\node at (0, 1) {$\bullet$};
\node at (-1, 0) {$\bullet$};
\node at (-1, -1) {$\bullet$};
\node at (0, -1) {$\bullet$};
\node at (1.5, .5) {$\sigma_1$};
\node at (.5, 1.5) {$\sigma_2$};
\node at (-1, 1) {$\sigma_3$};
\node at (-1.5, -.5) {$\sigma_4$};
\node at (-.5, -1.5) {$\sigma_5$};
\node at (1, -1) {$\sigma_6$};
\end{tikzpicture}
\caption{The fan $\Sigma$}
\end{figure}
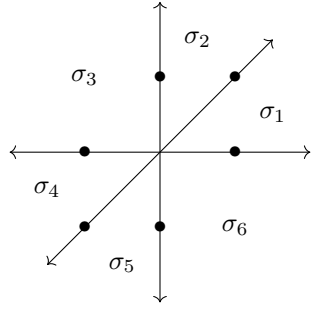

There is a corresponding piecewise linear map $\Phi: \Sigma \to \Trop(I_{2, 4})$.  Recall that $\Trop(I_{2, 4})$ can be identified with the space of metric trees with four leaves.  When used with the $\min$ convention the lengths are the non-leaf edges of these trees become negative, which is somewhat counterintuitive.  To compensate we record the negatives of the values of $\Phi$. 

\[\raisebox{5pt}{$-\Phi(1, 0) =$} \begin{tikzpicture}
\draw[-] (0, 0) -- (.25, .25);
\draw[-] (0, 0) -- (.25, -.25);
\draw[-] (0,0) -- (-.25, .25);
\draw[-] (0, 0) -- (-.25,-.25);
\node at (.3, .125) {{\tiny $0$}};
\node at (.3, -.125) {{\tiny $0$}};
\node at (-.3, .125) {{\tiny $0$}};
\node at (-.3, -.125) {{\tiny $0$}};
\end{tikzpicture} \ \
\raisebox{5pt}{$-\Phi(1, 1) =$} \begin{tikzpicture}
\draw[-] (0,0) -- (.25,0);
\draw[-] (.25, 0) -- (.5, .25);
\draw[-] (.25, 0) -- (.5, -.25);
\draw[-] (0,0) -- (-.25, .25);
\draw[-] (0, 0) -- (-.25,-.25);
\node at (.125, .125) {{\tiny $.5$}};
\node at (.61, .125) {{\tiny $-.5$}};
\node at (.61, -.125) {{\tiny $-.5$}};
\node at (-.3, .125) {{\tiny $0$}};
\node at (-.3, -.125) {{\tiny $0$}};
\end{tikzpicture} \ \
\raisebox{5pt}{$-\Phi(0, 1) =$}\begin{tikzpicture}
\draw[-] (0, 0) -- (.25, .25);
\draw[-] (0, 0) -- (.25, -.25);
\draw[-] (0,0) -- (-.25, .25);
\draw[-] (0, 0) -- (-.25,-.25);
\node at (.3, .125) {{\tiny $0$}};
\node at (.3, -.125) {{\tiny $0$}};
\node at (-.3, .125) {{\tiny $0$}};
\node at (-.3, -.125) {{\tiny $0$}};
\end{tikzpicture} \ \
 \]

\[\raisebox{5pt}{$-\Phi(-1,0) =$} \begin{tikzpicture}
\draw[-] (0, 0) -- (.25, .25);
\draw[-] (0, 0) -- (.25, -.25);
\draw[-] (0,0) -- (-.25, .25);
\draw[-] (0, 0) -- (-.25,-.25);
\node at (.3, .125) {{\tiny $1$}};
\node at (.3, -.125) {{\tiny $0$}};
\node at (-.3, .125) {{\tiny $0$}};
\node at (-.3, -.125) {{\tiny $0$}};
\end{tikzpicture} \ \
\raisebox{5pt}{$-\Phi(-1,-1) =$} \begin{tikzpicture}
\draw[-] (0,0) -- (.25,0);
\draw[-] (.25, 0) -- (.5, .25);
\draw[-] (.25, 0) -- (.5, -.25);
\draw[-] (0,0) -- (-.25, .25);
\draw[-] (0, 0) -- (-.25,-.25);
\node at (.125, .125) {{\tiny $.5$}};
\node at (.55, .125) {{\tiny $.5$}};
\node at (.55, -.125) {{\tiny $.5$}};
\node at (-.3, .125) {{\tiny $0$}};
\node at (-.3, -.125) {{\tiny $0$}};
\end{tikzpicture} \ \
\raisebox{5pt}{$-\Phi(0, -1) =$} \begin{tikzpicture}
\draw[-] (0, 0) -- (.25, .25);
\draw[-] (0, 0) -- (.25, -.25);
\draw[-] (0,0) -- (-.25, .25);
\draw[-] (0, 0) -- (-.25,-.25);
\node at (.3, .125) {{\tiny $0$}};
\node at (.3, -.125) {{\tiny $1$}};
\node at (-.3, .125) {{\tiny $0$}};
\node at (-.3, -.125) {{\tiny $0$}};
\end{tikzpicture}\]

\bigskip

The map $\Phi$ maps the positive $x$ and $y$ axes to the origin, and maps the ray through $(1, 1)$ and $\sigma_4$ and $\sigma_5$ to the face $\tau \subset \Trop(I_{2, 4})$ corresponding to the initial ideal $\langle p_{12}p_{34} - p_{13}p_{24}\rangle$.  The family $\mathcal{G}(\Phi) \to Y(\Sigma)$ has fibers $X_{2,4}$, except over the divisors $D(-1,-1), D(1,1) \subset Y(\Sigma)$, where the fiber is the toric variety $V(\langle p_{12}p_{34} - p_{13}p_{24}\rangle) \subset \bigwedge^2(\k^4)$. 
\end{example}

\begin{example}
In the previous example we see a simple toric family defined by a map $\Phi: \Sigma \to \Trop(I_{2,4})$, which hits only one non-trivial face of the tropical variety.  Now we consider a map $\Psi: |\Sigma| \to \Trop(I_{2,4})$ which folds $\Sigma$ into all three non-trivial faces of $\Trop(I_{2,4})$:

\def\DA{\begin{tikzpicture}
\draw[-] (0, 0) -- (.25, .25);
\draw[-] (0, 0) -- (.25, -.25);
\draw[-] (0,0) -- (-.25, .25);
\draw[-] (0, 0) -- (-.25,-.25);
\node at (.3, .125) {{\tiny $1$}};
\node at (.3, -.125) {{\tiny $1$}};
\node at (-.3, .125) {{\tiny $1$}};
\node at (-.3, -.125) {{\tiny $1$}};
\end{tikzpicture}}

\def\DB{
\begin{tikzpicture}
\draw[-] (0,0) -- (.25,0);
\draw[-] (.25, 0) -- (.5, .25);
\draw[-] (.25, 0) -- (.5, -.25);
\draw[-] (0,0) -- (-.25, .25);
\draw[-] (0, 0) -- (-.25,-.25);
\node at (.125, .125) {{\tiny $1$}};
\node at (.55, .125) {{\tiny $0$}};
\node at (.55, -.125) {{\tiny $0$}};
\node at (-.3, .125) {{\tiny $0$}};
\node at (-.3, -.125) {{\tiny $0$}};
\end{tikzpicture}}

\def\DC{\begin{tikzpicture}
\draw[-] (0,0) -- (.25,0);
\draw[-] (.25, 0) -- (.5, .25);
\draw[-] (0, 0) -- (.5, -.25);
\draw[-] (0,0) -- (-.25, .25);
\draw[-] (.25, 0) -- (-.25,-.25);
\node at (.125, .125) {{\tiny $1$}};
\node at (.55, .125) {{\tiny $0$}};
\node at (.55, -.125) {{\tiny $0$}};
\node at (-.3, .125) {{\tiny $0$}};
\node at (-.3, -.125) {{\tiny $0$}};
\end{tikzpicture}}

\def\DD{\begin{tikzpicture}
\draw[-] (0, 0) -- (0, .25);
\draw[-] (0, .25) -- (.25, .5);
\draw[-] (0, 0) -- (.25, -.25);
\draw[-] (0,.25) -- (-.25, .5);
\draw[-] (0, 0) -- (-.25,-.25);
\node at (.125,.125 ) {{\tiny $1$}};
\node at (.3, .375) {{\tiny $0$}};
\node at (.3, -.125) {{\tiny $0$}};
\node at (-.3, .375) {{\tiny $0$}};
\node at (-.3, -.125) {{\tiny $0$}};
\end{tikzpicture}}

\[ \raisebox{5pt}{$-\Psi_{\sigma_1 \cup \sigma_2}(x,y) = \max\{x,y\}$}\DA \raisebox{5pt}{$+ \max\{x-y, 0\}$} \raisebox{-3pt}{$\DD$}  \raisebox{5pt}{$+\max\{y-x, 0\}$} \DB \]

\[\raisebox{5pt}{$-\Psi_{\sigma_3 \cup \sigma_4}(x,y) =(-x + \max\{-y, y\})$}\DA \raisebox{5pt}{$+ \max\{y, 0\}$} \DB \raisebox{5pt}{$+ \max\{-y, 0\}$} \DC. \]

\[\raisebox{5pt}{$ -\Psi_{\sigma_5 \cup \sigma_6}(x,y) = (-y + \max\{-x, x\})$}\DA \raisebox{5pt}{$+ \max\{x, 0\}$} \raisebox{-3pt}{$\DD$} \raisebox{5pt}{$+ \max\{-x, 0\}$} \DC \]

\bigskip
On the ray generators of $\Sigma$ we have:

\[\raisebox{5pt}{$-\Psi(1, 0) =$} \DA \raisebox{5pt}{$+$}\raisebox{-3pt}{$\DD$} \ \ \raisebox{5pt}{$-\Psi(1, 1) =$} \DA \ \ \raisebox{5pt}{$-\Psi(0, 1) =$} \DA \raisebox{5pt}{$+$} \DB \]

\[\raisebox{5pt}{$-\Psi(-1, 0) =$} \DA \ \ \raisebox{5pt}{$-\Psi(-1,-1) = 2$}\DA \raisebox{5pt}{$+$} \DC \ \ \raisebox{5pt}{$-\Psi(0, -1) =$} \DA \]

\bigskip
We let $\w: A_{2, 4} \to \O_\Sigma$ denote the valuation associated to $\Psi$. 
By Theorem \ref{thm-main-StrongKhovanskii}, we compute the map $\psi: \k[X_1, \ldots, X_6, Y_{12}, \ldots, Y_{34}] \to \mathcal{R}(A_{2, 4}, \hat{\w})$, where $\psi(X_i) = t_i^{-1} \in F_{-\e_i}(\w)$, and:

\[\psi(Y_{12}) = p_{12} \in F_{-2,- 2, -3, -2, -5, -2}(\w) \ \ \ \psi(Y_{13}) = p_{13} \in F_{-3, -2, -3, -2, -4, -2}(\w) \]
\[\psi(Y_{14}) = p_{14} \in F_{-3, -2, -2, -2, -5, -2}(\w) \ \ \ \psi(Y_{23}) = p_{23} \in F_{-3, -2, -2, -2, -5, -2}(\w) \]
\[\psi(Y_{24}) = p_{24} \in F_{-3, -2, -3, -2, -4, -2}(\w) \ \ \ \psi(Y_{34}) = p_{34} \in F_{-2, -2, -3, -2, -5, -2}(\w) \]

\bigskip
Here $p_{ij} \in F_{r_1, \ldots, r_6}(\w)$ means that the value of the valuation associated to the $k$-th ray is $r_k$ on $p_{ij}$; in particular the length of the unique path from $i$ to $j$ in the metric tree defined by the image of the $k$-th ray is $-r_k$.  We have $ker(\psi) = \langle X_1^2Y_{12}Y_{34} - X_5^2Y_{13}Y_{24} + X_3^2Y_{14}Y_{23}\rangle$.  The $p_{ij}$ are a Khovanskii basis for $\w$, and the ideal $\langle ker(\psi), X_i\rangle$ is prime for all $i \in [6]$, so 

\[\mathcal{R}(A_{2, 4}, \hat{\w}) \cong \k[X_1, \ldots, X_6, Y_{12}, \ldots, Y_{34}]/\langle X_1^2Y_{12}Y_{34} - X_5^2Y_{13}Y_{24} + X_3^2Y_{14}Y_{23}\rangle.\]

\bigskip
In particular, the Pl\"ucker generators are a strong Khovanskii basis of $\w$.
\end{example}

\subsection{Prime cones}
In \cite{Kaveh-Manon-NOK} the authors show that a full rank discrete valuation can be constructed on a domain $A$ with  Khovanskii basis $\mathcal{B} \subset A$ by choosing a linearly independent set ${\bf u} = \{u_1, \ldots, u_d\}$ in a \emph{prime cone} $C \subset \Trop(I)$.   These points are then arranged into the rows of a $d \times n$ matrix $M$, where $n = |\mathcal{B}|$.  The valuation $\v_{\bf u}$ then sends $b_i$ to the $i$-th column of $M$.  The collection ${\bf u}$ defines a linear map $\Phi_{\bf u}: C_d \to C \subset \Trop(I)$, where $C_d$ is a smooth cone with $d$ rays.  In this case, the matrix $M$ is the diagram of the corresponding flat family over $\mathbb{A}^d(\k)$.  This family is described in \cite[Section 6]{Kaveh-Manon-NOK}. 

\subsection{Examples from cluster algebras}\label{clusterexample}

As we have seen in this section, $\O_N$ valuations appear on the coordinate rings of rational varieties.  In particular, any map $\phi: A \to \k(T_N)$ can be composed with the canonical valuation $\w_N: \k(T_N) \to \O_N$.  Cluster varieties are equipped with large combinatorial families of such maps, in particular there is one for each cluster chart.  Following \cite{GHKK}, we consider an $A$-type cluster variety $\A$ associated to a non-degenerate skew-symmetric, bilinear form $\{ \ ,\ \}$ on a lattice $N$.   Let $P^*: N \to M$ be the map $P^*(n) = \{n, \cdot \}$ with dual map $-P^* = P^{**}: N \to M$.  We make technical assumptions that $\A$ has no frozen variables, and that $\Theta = \A_{prin}(\Z^T)$ (see \cite[Section 7]{GHKK}). 

  A seed $\bs = \{e_1, \ldots, e_n\} \subset N$ determines a cone $\Q_{\geq 0}\bs = \sigma^\vee_\bs \subset N_\Q $ and a dual cone $\sigma_\bs \subset M_\Q$ called a \emph{chamber}.  The monoid $N_\bs^- = \sigma^\vee_\bs \cap N$ defines a polynomial ring $\k[N_\bs^-]$ equipped with an action by the torus $T_N$, in particular the $T_N$ character corresponding to a monomial $z^n \in \k[N_\bs^-]$ is $P^*(n) \in M$.   Each seed $\bs$ determines a chart $i_\bs: T_N \to \A$.  A mutation of seeds $\mu_k: \bs \to \bs'$ along the $k$-th element corresponds to a birational map $\mu_k: T_N \dashrightarrow T_N$ and a piecewise linear map $\mu_k: N \to N$. There is also a corresponding dual piecewise linear map $T_k: M \to M$.  Let $H_k = \{ m \mid \langle e_k, m \rangle = 0\} \subset M$, with $H_k^+$ and $H_k^-$ the two corresponding half-spaces.  On $H_k^-$, $T_k$ is the identity map, and on $H_k^+$, $T_k(m) = m + \langle e_k, m \rangle P^*(e_k)$.   The mutation maps $\mu_k: N \to N$ and $T_k: M \to M$ commute with $P^*: N \to M$.  

 One of the central constructions of \cite{GHKK} is the canonical basis $\Theta \subset \k[\A]$.  The elements of $\Theta$ are linearly independent, and the vector space $\can(\A) = \bigoplus \k \theta \subset \k[\A]$ is shown to be a subalgebra of $\k[\A]$ under the conditions we have assumed here. A choice of seed $\bs$ defines for each $\theta \in \Theta$ an element $g_\bs(\theta) \in M$ called the $g$-vector.  For a mutation $\mu_k: \bs \to \bs'$ we have $g_{\bs'}(\theta) = T_k(g_\bs(\theta))$.  There are distinguished subspaces of $\can(\A)$ each with a basis inherited from $\Theta$.  For each $m \in M$ we have $F_m(\bs) = \langle \{ \theta \mid m - g_\bs(\theta) \in P^*(\sigma^\vee_\bs) \} \rangle$. We let $S_{\bs} = \bigoplus_{m \in M} F_m(\bs)$ be the Rees space of this system of filtrations. 

\begin{proposition}\label{cluster-local-valuation}
The space $S_{\bs}$ is a free, $M$-graded $\k[N_\bs^-]$ module with generators given by the canonical basis $\Theta$. Furthermore, $S_\bs$ is an algebra, and $\mathcal{L}(S_\bs) = (\can(\A), \v_\bs)$, where $\v_\bs: \can(\A) \to \O_{\sigma_\bs}$ has linear adapted basis $\Theta$. 
\end{proposition}

\begin{proof}
Once again, the $M$-grading on $S_\bs$ is given by assigning $z^n\theta$ for $n \in N_\bs^-$ the weight $P^*(n) + g_\bs(\theta)$. 
The fact that $S_\bs$ is an algebra and $\v_\bs$ is a valuation are both deep results from \cite[Section 9]{GHKK}.  The freeness property is a consequence of Proposition \ref{free}. 
\end{proof}

Let $\hat{\sigma}_\bs \subset N_\Q$ be the cone that maps to $\sigma_\bs \subset M_\Q$ under $P^*$. For $\rho \in \sigma_\bs \cap M$ with $\rho = P^*(\rho')$ we have $\v_\bs(\theta)(\rho) = \langle \rho', g_\bs(\theta) \rangle$. 

Once again, we fix a seed $\bs$. The $g$-fan $\Delta_\bs \subset M_\Q$ for the seed $\bs$ is a simplicial fan composed of cones $\sigma_{\bs, \bs'}$, where $\bs'$ runs over all other seeds in the cluster complex and $\sigma_{\bs, \bs} = \sigma_\bs$.  Here one thinks of $\sigma_\bs$ has the positive orthant.  For a seed $\bs'$ connected to $\bs$ through a composite series of mutations $\mu_{\bf k}$, $\sigma_{\bs, \bs'} \subset M_\Q$ is defined by the property $T_{\bf k}(\sigma_{\bs, \bs'}) = \sigma_{\bs'}$.  Remarkably, whenever $\mu: \bs \to \bs'$ is a simple mutation, $\sigma_{\bs, \bs'} \cap \sigma_{\bs, \bs}$ share the common facet $H_k \cap \sigma_{\bs}$.  It is shown in \cite{GHKK} that a mutation $\mu_k: \bs \to \bs'$ takes $\Delta_\bs$ to $\Delta_{\bs'}$, and each restriction $T_k: \sigma_{\bs, \bs''} \to \sigma_{\bs', \bs''}$ is linear.  For $\sigma_{\bs, \bs'} \in \Delta_{\bs}$, we use the map $T_{\bf k}: \sigma_{\bs, \bs'} \to \sigma_{\bs}$ to define a valuation $\v_{\bs, \bs'}: \can(\A) \to \O_{\sigma_{\bs, \bs'}}$, where $\v_{\bs, \bs'}(f)(\rho) = \v_{\bs'}(f)(T_{\bf k}(\rho))$. 

\begin{proposition}\label{cluster-bundle-global-valuation}
The valuations $\v_{\bs, \bs'}: \can(\A) \to \O_{\sigma_{\bs, \bs'}}$ fit together to define a valuation $\w_\bs: \can(\A) \to \hat{\O}_{\Delta_\bs}$.  Furthermore, for each $\theta \in \Theta$, there is a line bundle $D_\bs(\theta)$ on the toric variety $Y_{\Delta_\bs}$, and $\mathcal{L}(\bigoplus_{\theta \in \Theta} D_\bs(\theta)) = (\can(\A), \w_\bs)$. 
\end{proposition}

\begin{proof}
We show that for any $\sigma_{\bs, \bs'}$ and $\sigma_{\bs, \bs''}$ meeting at a facet in $\Delta_\bs$, the restrictions of $\v_{\bs, \bs'}$ and $\v_{\bs, \bs''}$ agree. Since $\Theta$ is an adapted basis of all $\v_{\bs, \bs'}$, it suffices to show the restrictions of $\v_{\bs, \bs'}(\theta)$ and $\v_{\bs, \bs''}(\theta)$ coincide.  If this already holds when $\bs' = \bs$ (ie there is a simple mutation $\mu: \bs \to \bs''$), then for $\rho \in \sigma_{\bs, \bs'} \cap \sigma_{\bs, \bs''}$ 
we have $\v_{\bs, \bs'}(f)(\rho) = \v_{\bs'}(f)(T_{\bf k}(\rho)) = \v_{\bs', \bs''}(f)(T_{\bf k}(\rho)) = \v_{\bs, \bs''}(f)(\rho)$, where $\mu_{\bf k}: \bs \to \bs'$.   So it remains to check the case when $\bs'$ is adjacent to $\bs$.  Let $T_k: \sigma_{\bs, \bs'} \to \sigma_{\bs'}$; then be definition $T$ is identity on the shared face of $\sigma_{\bs, \bs'}$ and $\sigma_{\bs}$, as this is contained in $H_k$. Now let $\rho$ be in this face, then $\v_{\bs, \bs'}(\theta)(\rho) = \v_{\bs'}(\theta)(\rho) = \langle P^*(\rho'), g_{\bs'}(\theta)\rangle$, for $\rho' \in N_\Q$ in the hyperplane which maps to $H_k$ under $P^*$. This is equal to $\langle \rho', T_k(g_\bs(\theta))\rangle = \langle \rho', g_\bs(\theta) + \langle e_k, g_\bs(\theta) \rangle P^*(e_k) \rangle$.  But by assumption $\langle \rho', P^*(e_k) \rangle = \{\rho', e_k \} = - \{e_k, \rho'\} = - \langle e_k, P^*(\rho') \rangle = 0$, since $P^*(\rho') \in H_k$.  This implies that the $\v_{\bs, \bs'}$ fit together into a valuation $\w_\bs: \can(\A) \to \hat{\O}_{\Delta_\bs}$ with adapted basis $\Theta$.  For each $\theta \in \Theta$, the restriction of $\w_{\bs}(\theta)$ to a face of $\Delta_\bs$ is linear; it follows that there is a toric line bundle $D_\bs(\theta)$ on $Y_{\Delta_\bs}$ corresponding to the piecewise linear function $\w_\bs(\theta)$.  Theorem \ref{eversiveglobalequivalence} implies the rest. 
\end{proof}

 Observe that while $\w_\bs$ depends on the choice of seed $\bs$, the ring of global coordinates $\mathcal{R}(\can(\A), \w_\bs)$ does not, as the cluster fans $\Delta_\bs$ and the valuations $\w_\bs$ are all related by piecewise linear maps. 

\section{Proofs for Section \ref{category}}\label{proofs}
In this section we give proofs for the main results in Section \ref{category}.


\subsection*{Proof of Theorem \ref{thm-adjoint}}
First we show that $\L: \Mod_{S_\sigma}^M \to \Vect_\sigma$ and $\mathcal{R}: \Vect_\sigma \to \Mod_{S_\sigma}^M$ are functors. For $R \in \Mod_{S_\sigma}^M$, we check that $\v_R(f) \in \hat{\O}_\sigma$, it is then immediate that $\v_R$ is a prevaluation.  Fix $\ell \in \Z_{> 0}$, and suppose $\v_R(f)(\rho) = r$, so  $f \in G^\rho_r(R) \setminus G^\rho_{r-1}(R)$. By definition, $f \in G^{\ell\rho}_{\ell r}(R)$, and if $f \in G^{\ell \rho}_{\ell r -1}(R)$, then for some $n_1, \ldots, n_d$ we have $f \in F_{n_1}(R) + \cdots + F_{n_d}(R)$ where $\langle \ell \rho, n_i \rangle > \ell r$. But then $\langle \rho, n_i \rangle > r$, which contradicts $\v_R(f)(\rho) = r$. 

For an $\Mod_{S_\sigma}^M$-morphism $\psi: R_1 \to R_2$ there is an induced map $\mathcal{L}(\psi): E_{R_1} \to E_{R_2}$, and for any $m \in M$ the $m$-component $\psi_m$ satisfies $\phi_R \circ \psi_m = \mathcal{L}(\psi) \circ \phi_R$. It follows that $\mathcal{L}(\psi)$ maps $G^\rho_r(R_1)$ into $G^\rho_r(R_2)$, so $\mathcal{L}(\psi)$ is a morphism in $\Vect_\sigma$. 

For $(E, \v) \in \Vect_\sigma$, it is clear that $\mathcal{R}(E, \v) \in \Mod_{S_\sigma}^M$, so it remains to define the morphism $\mathcal{R}(\psi): \mathcal{R}(E_1, \v_1) \to \mathcal{R}(E_2, \v_2)$ associated to $\psi: (E_1, \v_1) \to (E_2, \v_2)$. We have $\v_1(f) \geq m$ implies $\v_2(\psi(f)) \geq \v_1(f) \geq m$, so $\psi(F_m(\v_1)) \subset F_m(\v_2)$.  For any $u \in \sigma^\vee \cap M$ the action of $\chi_{m}$ commutes with the inclusions $\psi$, so $\mathcal{R}(\psi)$ is a morphism in $\Mod_{S_\sigma}^M$. It is straightforward to show $\mathcal{R}(\psi\circ \psi') = \mathcal{R}(\psi) \circ \mathcal{R}(\psi')$.

Now we prove the adjunction statement.  Recall the universal property of colimit: if we are given a system of maps $\psi_m: R_m \to E$ which commute with $\chi_u$ for $u \in \sigma^\vee \cap M$, then there is a unique map $\ell(\psi) = \colim \psi_m : E_R \to E$ such that $\psi_m = (\ell(\psi) \circ \phi_R)_m$.  Let $\phi: \mathcal{L}(R) \to (E, \v)$ be a morphism in $\Vect_\sigma$.  The image $F_m(R)$ is a subspace of $F_m(\v_R)$, so we let $r(\phi): R \to \mathcal{R}(E, \v)$ be defined by the maps $r(\phi)_m = \phi_m \circ \phi_R: R_m \to F_m(\v)$.  The following diagram commutes:

$$
\begin{CD}
R_m @>\phi_R>> F_m(\v_R) @>\phi_m>> F_m(\v)\\
@V\chi_uVV @VVV @VVV\\
R_{m + u} @>\phi_R>> F_{m+u}(\v_R) @>\phi_{m+u}>> F_{m+u}(\v),
\end{CD}
$$\\
as the right two vertical arrows are inclusions, with the middle occuring in the colimit $E_R$. As a consequence, $r(\phi) \in \Hom_{\Mod_{S_\sigma}^M}(R, \mathcal{R}(A, \v))$.    

Given $\psi: R \to \mathcal{R}(E, \v)$, we have a family of maps $\psi_m: R_m \to F_m(\v)$ which commute with each $\chi_u$. We get $\ell(\psi): E_R \to E$, where $\psi_m = \ell(\psi) \circ \phi_R$ on $R_m$.  Furthermore, $G^\rho_r(R)$ is mapped into $\sum_{\langle \rho, m \rangle \geq r} F_m(\v) \subset G^\rho_r(\v)$  so $\ell(\psi) \in \Hom_{\Vect_{\sigma}}(\mathcal{L}(R), (E, \v))$. 

Now we show that $\ell$ and $r$ are inverses of each other. Given $\phi: \mathcal{L}(R) \to (E, \v)$ and $f \in E_R$, pick $\hat{f} \in R_m$ so that $\phi_R(\hat{f}) = f$.  Then $r(\phi)(\hat{f}) = \phi \circ \phi_R(\hat{f}) = \phi(f) \in F_m(\v)$.  It follows that $\ell(r(\phi))(f) = \phi(f)$.  If $\psi: R \to \mathcal{R}(E, \v)$, $\psi_m = (\ell(\psi)\circ \phi_R)_m$, so $r(\ell(\psi))_m = (\ell(\psi) \circ \phi_R)_m = \psi_m$.   We omit the proof that $r$ and $\ell$ are natural, as it is straightforward. \\

\subsection*{Proof of Theorem \ref{eversiveglobalequivalence}}
Theorem \ref{thm-adjoint} implies that the fiber $E_\F$ over $id \in T_N \subset Y(\Sigma)$ carries a prevaluation $\v_{\F, \sigma}: E_\F \to \hat{\O}_\sigma$ for each face $\sigma \in \Sigma$. We must check that for any $f \in E_\F$, and $\tau = \sigma_1 \cap \sigma_2$, the restrictions $\v_{\F, \sigma_1}(f) \!\! \mid_\tau$ and $\v_{\F, \sigma_2} (f) \!\! \mid_\tau$ coincide.  More generally, let $\tau \subset \N'$ and $\sigma \subset \N$ be pointed polyhedral cones, and let $\iota: \tau \to \sigma$ be a linear map induced by a map of lattices $\iota: N' \to N$.  There is an associated map on dual lattices $\iota^*: M \to M'$ and semigroup algebras $\iota^*: S_\sigma \to S_\tau$.   The map $\iota^*$ gives an extension functor $-\otimes_{S_\sigma} S_\tau: \Mod_{S_\sigma}^M \to \Mod_{S_\tau}^{M'}$.   We also have a functor $\iota^{\dagger}: \Vect_\sigma \to \Vect_\tau$ obtained by composing $\v: E \to \hat{\O}_\sigma$ with the map on semialgebras $\iota^{\sharp}: \hat{\O}_\sigma \to \hat{\O}_\tau$ given by precomposition with $\iota$. We show that these two functors coincide under $\mathcal{L}$.  As a consequence, the restriction of $\mathcal{L}(R) = (E_R, \v_R)$ to a facet $\tau \subset \sigma$ is $\mathcal{L}(R\otimes_{S_\sigma} S_\tau)$. 

\begin{proposition}\label{localization}
The following diagram commutes. 

\[
\begin{CD}
\Mod_{S_\sigma}^M @>-\otimes_{S_\sigma} S_\tau>> \Mod_{S_\tau}^{M'}\\
@V\mathcal{L}_\sigma VV @V \mathcal{L}_\tau VV\\
\Vect_\sigma @>\iota^{\dagger}>> \Vect_\tau\\
\end{CD}
\]

\end{proposition}

\begin{proof}
Let $R' = R \otimes_{S_\sigma} S_\tau$ and $R_0 = R\otimes_{S_\sigma} \k[M]$. We have $R' \otimes_{S_\tau} \k[M'] \cong R_0\otimes_{\k[M]}\k[M']$. By setting $\chi_{m'} = 1$ for each $m' \in M'$ we obtain $E_R \cong E_{R'}$, so $\phi_{R'} \circ (\iota^* \otimes 1) = \phi_R$.  Now let $\rho \in \tau \cap N'$, and consider $G^{\iota(\rho)}_r(R) = \sum_{\langle \iota(\rho), m \rangle \geq r} F_m(R)$. This space is the image of $\bigoplus_{\langle \iota(\rho), m \rangle \geq r} R_m \subset R$ under $\phi_R$.  Similarly, $G^\rho_r(R')$ is the image of $\bigoplus_{\langle \rho, m \rangle \geq r} R'_m$ under $\phi_{R'}$ where 

\begin{equation}
R'_m = (\bigoplus_{n \in M, u \in \tau^\vee \cap M', \iota^*(n) + u = m} R_n \otimes_{\k} \k\chi_u)/\sim.
\end{equation}
We have $\langle \iota(\rho), m \rangle \geq r$ if and only if $\langle \rho, \iota^*(m) \rangle \geq r$, so $F_m(R) \subset F_{\iota^*(m)}(R')$ and $G^{\iota(\rho)}_r (R) \subset G^\rho_r(R')$ in $E_R \cong E_{R'}$.  Similarly, if $f \in G^\rho_r(R')$ then we can write $f = \sum \phi_R(g_i)$, for $g_i \in F_{n_i}(R) = \phi_{R'}((\iota^* \otimes 1)(R_{n_i} \otimes_\k \k \chi_u))$, where $\iota^*(n_i) \odot u = m_i$ with $\langle \rho, m_i \rangle \geq r$.   But then $\langle \rho, \iota^*(n_i) \rangle \geq r$, so that $\langle \iota(\rho), n_i \rangle \geq r$.  It follows that $f \in G^{\iota(\rho)}_r(R)$ and $G^{\iota(\rho)}_r(R) = G^\rho_r(R')$.   Now fix $\rho \in \tau \cap N'$ and $f \in A_R$, then $\iota^{\sharp}\v_R(f)(\rho) = \v_R(f)(\iota(\rho))$, which is equal to $\v_{R'}(f)(\rho)$ by the above calculation.  This shows that $\iota^{\dagger} \circ \mathcal{L}_\sigma (R) = (E_R, \iota^{\sharp}\v_R)$ is isomorphic to $(E_{R'}, \v_{R'}) = \mathcal{L}_\tau \circ (R \otimes_{S_\sigma} S_\tau)$ by the map which identifies $E_R$ with $E_{R'}$. 
\end{proof}

Proposition \ref{localization} implies that $\L(\F) = (E_\F, \v_\F) \in \Vect_\Sigma$.  By Proposition \ref{thm-adjoint} a morphism $\phi: \mathcal{F} \to \mathcal{G}$ gives a map $\mathcal{L}(\phi)\!\!\mid_\sigma: (E_\mathcal{F}, \v_{\mathcal{F}}\!\! \mid_\sigma) \to (E_\mathcal{G}, \v_{\mathcal{G}}\!\! \mid_\sigma)$, so we have a functor $\mathcal{L}: \Sh^M_{Y(\Sigma)} \to \Vect_\Sigma$.  We observe (\cite[Lemma 17.15.4]{stacks}) that for any $\sigma \in \Sigma$, and $\F, \E \in \Sh_{Y(\Sigma)}^M$, $\mathcal{L}(\F \otimes_{Y(\Sigma)} \E)$ restricted to $\sigma$ is $\mathcal{L}\circ \Gamma(Y(\sigma), \E\!\mid_{Y(\sigma)} \otimes_{Y(\sigma)} \F\!\mid_{Y(\sigma)})$ $=\mathcal{L}(\Gamma(Y(\sigma), \E)\otimes_{S_{\sigma}} \Gamma(Y(\sigma), \F))$ since $\E, \F$ are quasi-coherent and $Y(\sigma)$ is affine.  The latter is $\mathcal{L}(\Gamma(Y(\sigma), \E)) \otimes \mathcal{L}(\Gamma(Y(\sigma), \F))$, so it follows that $\mathcal{L}(\E \otimes_{Y(\Sigma)} \F) = \mathcal{L}(\E) \otimes \mathcal{L}(\F)$.

Now let $\F, \mathcal{G} \in \Sh^{M, ev}_{Y(\Sigma)}$, and consider a map $\psi: (E_\mathcal{F}, \v_{\mathcal{F}}) \to (E_{\mathcal{G}}, \v_{\mathcal{G}})$.  For any $\sigma \in \Sigma$ we obtain a map on modules $\mathcal{R}(\psi \!\!\mid_\sigma): \mathcal{R}(E_\F, \v_\F \!\!\mid_\sigma) \to \mathcal{R}(E_\mathcal{G}, \v_\mathcal{G} \!\!\mid_\sigma)$, which in turn induces a map on quasi-coherent $Y(\sigma)$ sheaves: $\tilde{\mathcal{R}}(\psi \!\!\mid_\sigma): \tilde{\mathcal{R}}(E_\F, \v_\F \!\!\mid_\sigma) \to \tilde{\mathcal{R}}(E_\mathcal{G}, \v_\mathcal{G} \!\!\mid_\sigma)$.  The sheaves $\F, \mathcal{G}$ were chosen in $\Sh^{M, ev}_{Y(\Sigma)}$, so $\tilde{\mathcal{R}}(E_\F, \v_\F \!\!\mid_\sigma)$ and $\tilde{\mathcal{R}}(E_\mathcal{G}, \v_\mathcal{G} \!\!\mid_\sigma)$ coincide with the restrictions $\F \!\mid_{Y(\sigma)}$ and $\mathcal{G} \!\mid_{Y(\sigma)}$, and we obtain an induced map on \emph{descent data} for the open cover of $Y(\Sigma)$ by the $Y(\sigma)$ for $\sigma \in \Sigma$: $\tilde{\mathcal{R}}(\psi): \{ (Y(\sigma), \F \!\mid_{Y(\sigma)}), \sigma \in \Sigma\} \to \{ (Y(\sigma), \mathcal{G} \!\mid_{Y(\sigma)}), \sigma \in \Sigma\}$ (see e.g. \cite[Lemma 68.3.4]{stacks}).  This map induces a unique map between $\F$ and $\mathcal{G}$.  It's straightforward to check that this construction is inverse to $\mathcal{L}$ on $\Sh^{M, ev}_{Y(\Sigma)}$, in particular $\mathcal{L}(\phi): (E_\mathcal{F}, \v_{\mathcal{F}}) \to (E_\mathcal{G}, \v_{\mathcal{G}})$ is an isomorphism if and only if $\phi: \mathcal{F} \to \mathcal{G}$ is an isomorphism. \\

\subsection*{Proof of Proposition \ref{fiber}} 
We make use of Lazard's Theorem  (\cite[Theorem 10.80.4]{stacks}).   Any $f\in G^\rho_r(R)$ is a sum of elements $f_i \in F_{m_i}(R)$ with $\langle \rho, m_i \rangle \geq r$, and all $F_{m}(R)$ with $\langle \rho, m \rangle = r$ are in the image of $\hat{\phi}_\rho$.    If $g \in \m_\tau R$ then we can write $g = \sum \chi_{v_j}g_j + \sum (\chi_{u_i} -1)h_i$ for $\langle \rho, v_j \rangle < 0$ and $\langle \rho, u_i \rangle = 0$, where $g_j$ and $h_i$ are homogeneous.  Let $deg(g) = m-v$ so that $\chi_vg \in R_m$, then $\hat{\phi}_\rho(\chi_vg) \in F_{m-v}(R) \subset G^\rho_{\langle \rho, m \rangle}(R)$. It follows that $\hat{\phi}_\rho(\chi_v g) = 0$.  Similarly, since $\langle \rho, u \rangle = 0$, $\hat{\phi}_\rho(\chi_u h) = \hat{\phi}_\rho(h)$.   If $R \in \Alg_{S_\sigma}^M$ then $\mathcal{L}(R) \in \Alg_\sigma$; it is straightforward to check that $\gr_\rho(E_R)$ is a graded algebra, and $\phi_\rho$ is a map of algebras. 

Now we show that $\phi_\rho$ is a injective for a free module $P$. Let $p_i \in F_{m_i}(P)$ with $\langle \rho, m_i \rangle = r$, and suppose that $\phi_\rho(\sum_{i =1}^n p_i) = 0$. Let $\mathbb{B} \subset E_P$ be a linear adapted basis with $deg(b_j) = \lambda_j$ for $b_j \in \B$.  Then $p_i = \sum_{j =1}^k c_{ij} \chi_{v_{ij}} b_j$ with $v_{ij} + \lambda_j = m_i$.   If $\langle \rho, v_{ij} \rangle < 0$ then $c_{ij}\chi_{v_{ij}}b_j \in \m_\tau P$, so without loss of generality we assume that $\langle \rho, v_{ij} \rangle = 0$. This implies that $\langle \rho, \lambda_j \rangle = r$ for each $b_j$.  Since the $b_j$ with $\langle \rho, \lambda_j \rangle = r$ form a basis of $G^\rho_r(P)/G^\rho_{> r}(P)$, we conclude that $\sum_{i =1}^n c_{ij} = 0$ for each $j$. Then $\sum_{i = 1}^n p_i = \sum_{j =1}^k (\sum_{i =1}^n c_{ij}\chi_{v_{ij}})b_j = \sum_{j =1}^k (\sum_{i =1}^n c_{ij}(\chi_{v_{ij}} -1))b_j \in \m_\tau P$.  Finally, observe that $\gr_\rho: \Vect_\sigma \to \Vect_\k$ is a functor which commutes with colimits (it is a direct sum of cokernels).  The functor $\mathcal{L}$ is a left adjoint, so it also commutes with colimits.  For any flat module $R$, we can write $\colim P_i = R$ for $P_i$ free. It follows that $R/\m_\tau R \cong \colim P_i/\m_\tau P_i \cong \colim \gr_\rho(\mathcal{L}(P_i)) \cong \gr_\rho(\colim \mathcal{L}(P_i)) \cong \gr_\rho(\mathcal{L}(R)) = \gr_\rho(E_R)$. \\

\subsection*{Proof of Proposition \ref{arrangementspaces}}
By definition, $F_\psi(\v) \subset \bigcap_{\varrho_i \in \Sigma(1)} G^{\rho_i}_{\psi(\rho_i)}(\v)$. We show that if $\v(f)(\rho_i) \geq \langle \rho_i, m_\sigma \rangle$ for each $\varrho_i \in \sigma(1)$ then $\v(f)\!\!\mid_\sigma \geq m_\sigma$, so $F_{m_\sigma}(\v\!\!\mid_\sigma) = \bigcap_{\varrho_i \in \sigma(1)} G^{\rho_i}_{\psi(\rho_i)}(\v)$.  A function $\phi \in \hat{\O}_\sigma$ is said to be \emph{concave} if for any $\rho_1, \ldots, \rho_\ell \in \sigma$ we have $\phi(\sum_{i =1}^\ell \rho_i) \geq \sum_{i =1}^\ell \phi(\rho_i)$.  We show that $\v_R: E_R \to \hat{\O}_\sigma$ takes concave values when $R$ is flat. 

If $(E, \v) = \mathcal{L}(R)$, where $R \in \Mod_{S_\sigma}^M$, let $\rho, \rho' \in \sigma \cap N$, and consider $f \in G^{\rho + \rho'}_r(R)$ for some $r \in \Z$.   We can write $f = \sum_{i =1}^\ell f_i$ for $f_i \in F_{m_i}(R)$ with $\langle \rho + \rho', m_i \rangle \geq r$.   We let $s_i = \langle \rho, m_i \rangle$ and $t_i = \langle \rho', m_i\rangle$ with $s_i + t_i = r$, so that  $f_i \in G^{\rho}_{s_i}(R) \cap G^{\rho'}_{t_i}(R)$.  It follows that $f \in \sum_{s + t \geq r} G^\rho_s(R) \cap G^{\rho'}_t(R)$, and $G^{\rho +\rho'}_r(R) \subset \sum_{s + t \geq r} G^\rho_s(R) \cap G^{\rho'}_t(R)$.  We show that containment holds in the other direction when $R$ is flat.   Let $f \in G^\rho_s(R) \cap G^{\rho'}_t(R)$ for $s + t \leq r$, then $f = \sum g_i = \sum h_j$ for $g_i \in F_{m_i}(R)$ and $h_j \in F_{n_j}(R)$ with $\langle \rho, m_i \rangle \geq s$ and $\langle \rho', n_j \rangle \geq t$.   Since $R$ is flat it is torsion-free; we choose an appropriate $m$ and identify $f \in R_m$ with $f = \sum \chi_{u_i} g_i$ $= \sum \chi_{v_j}h_j$ for $u_i + m_i = v_j + n_j = m$.  This is always possible as $\sigma^\vee$ is a full dimensional cone in $\M$.   Now we choose a presentation $\hat{\pi}: P \to R$ with $p_i \in P_{m_i}$ and $q_j \in P_{n_j}$ such that $\hat{\pi}(p_i) = g_i$, $\hat{\pi}(h_j) = q_j$.  We can find $\phi: P \to P'$, $\pi: P' \to R$ such that $\pi \circ \phi = \hat{\pi}$ with $P'$ free and $\pi(\sum \chi_{u_i} p_i)$ $= \pi(\sum \chi_{v_j}q_j) = p \in G^\rho_s(P') \cap G^{\rho'}_t(P') \subset G^{\rho+ \rho'}_r(P')$. We have $\mathcal{L}(\pi)(G^{\rho + \rho'}_r(P')) \subset G^{\rho + \rho'}_r(R)$, so $\mathcal{L}(\pi)(p) = f \in G^{\rho + \rho'}_r(R)$.  In conclusion, if $R$ is flat then:
\begin{equation}
G^{\sum_{i =1}^\ell \rho_i}_r(R) = \sum_{\sum_{i =1}^\ell t_i \geq r} G^{\rho_1}_{t_1}(R) \cap \cdots \cap G^{\rho_\ell}_{t_\ell}(R).
\end{equation}
If $\v(f)(\rho_i) = t_i$, it follows that $f \in G^{\rho_1}_{t_1}(R) \cap \cdots \cap G^{\rho_\ell}_{t_\ell}(R) \subset G^{\sum_{i =1}^\ell \rho_i}_r(R)$, where $r = \sum_{i =1}^\ell t_i$, so that $\v(f)(\sum_{i =1}^\ell \rho_i) \geq r = \sum_{i =1}^\ell \v(f)(\rho_i)$.  This proves that $\v(f)$ is concave. 

Now we consider $F_m(\v) = \bigcap _{\rho \in \sigma \cap N} G^\rho_{\langle \rho, m \rangle}(R) \subset \bigcap_{\varrho_i \in \sigma(1)} G^{\rho_i}_{\langle \rho_i, m\rangle}(R)$.   Pick $\rho \in \sigma \cap N$ and $d \in \Z_{\geq 0}$ so that $d\rho = \sum_{\varrho_i \in \sigma(1)} d_i\rho_i$ for some $d_i \in \Z_{\geq 0}$.   Now suppose that $\v(f)(\rho_i) \geq \langle \rho_i, m \rangle$ for $m \in M$.  It follows that $\v(f)(d\rho) = \v(f)(\sum_{\varrho_i \in \sigma(1)} d_i\rho_i) \geq \sum_{\varrho_i \in \sigma(1)} \v(f)(d_i\rho_i)$.  We have $\v(f) \in \hat{\O}_\sigma$, so $d\v(f)(\rho) = \v(f)(d\rho) \geq \sum_{\varrho_i \in \sigma(1)} \v(f)(d_i\rho_i) \geq \sum_{\varrho_i \in \sigma(1)} \langle d_i\rho_i, m \rangle = d\langle \rho, m\rangle$, and $\v(f)(\rho) \geq \langle \rho, m \rangle$. As $\rho$ was arbitrary, it follows that $f \in F_m(\v)$.

\bibliographystyle{alpha}
\bibliography{Biblio}

\end{document}